\documentclass[12pt]{amsart}

\usepackage[margin=1in]{geometry}

\usepackage{amsmath,amssymb,amsthm,enumerate,mathtools,xcolor,hyperref,comment}

\usepackage{makecell}

\title{The least prime with a given cycle type}
\author{Peter J. Cho}
\address{Department of Mathematical Sciences, Ulsan National Institute of Science and \newline
	\indent Technology, UNIST-gil 50, Ulsan 44919, Korea}
\email{petercho@unist.ac.kr}

\author{Robert J. Lemke Oliver}
\address{Department of Mathematics, University of Wisconsin-Madison, Madison, WI 53706, USA\ \indent Department of Mathematics, Tufts University, Medford, MA 02155, USA}
\email{lemkeoliver@wisc.edu}

\author{Asif Zaman}
\address{Department of Mathematics, University of Toronto, Toronto, ON M5S 2E4, CANADA}
\email{asif.zaman@utoronto.ca}

\usepackage[capitalise]{cleveref}

\newtheorem{theorem}{Theorem}
\newtheorem{lemma}[theorem]{Lemma}
\newtheorem{corollary}[theorem]{Corollary}
\newtheorem{proposition}[theorem]{Proposition}

\theoremstyle{definition} \newtheorem{definition}[theorem]{Definition}
\theoremstyle{example} \newtheorem{example}[theorem]{Example}
\theoremstyle{remark} \newtheorem{remark}{Remark}[theorem]

\numberwithin{equation}{section}
\numberwithin{theorem}{section}

\renewcommand{\epsilon}{\varepsilon}
\renewcommand{\leq}{\leqslant}
\renewcommand{\geq}{\geqslant}

\newcommand{\N}{\mathrm{N}} % norm -- not natural numbers

\newcommand{\R}{\mathbb{R}} 
\newcommand{\Q}{\mathbb{Q}}

\newcommand{\kp}{\mathfrak{p}}
\newcommand{\kD}{\mathfrak{D}}
\newcommand{\kf}{\mathfrak{f}}
\newcommand{\Gal}{\mathrm{Gal}}
\newcommand{\Frob}{\mathrm{Frob}}
\newcommand{\Res}{\mathop{\mathrm{Res}}}
\newcommand{\Ind}{\mathrm{Ind}}
\newcommand{\Disc}{\mathrm{Disc}}
\newcommand{\sgn}{\mathrm{sgn}}
\newcommand{\std}{\mathrm{std}}

\begin{document}

	\begin{abstract}
		Let $G$ be a finite group. Let $K/k$ be a Galois extension of number fields with Galois group isomorphic to $G$, and let $C \subseteq \mathrm{Gal}(K/k) \simeq G$ be a conjugacy invariant subset.  It is well known that there exists an unramified prime ideal $\mathfrak{p}$ of $k$ with Frobenius element lying in $C$ and norm satisfying $\N\kp \ll |\mathrm{Disc}(K)|^{\alpha}$ for some constant $\alpha = \alpha(G,C)$. There is a rich literature establishing unconditional admissible values for $\alpha$, with most approaches proceeding by studying the zeros of $L$-functions. We give an alternative approach, not relying on zeros, that often substantially improves this exponent $\alpha$ for any fixed finite group $G$, provided $C$ is a union of rational equivalence classes. 
		As a particularly striking example, we prove that there exist absolute constants $c_1,c_2 > 0$ such that for any $n\geq 2$ and any conjugacy class $C \subset S_n$, one may take $\alpha(S_n,C) = c_1 e^{-c_2n}$. Our approach reduces the core problem to a question in character theory. 
	\end{abstract}
	
	\maketitle

%%%%%%%%
%%%%%%%%
\section{Introduction}
%%%%%%%%
%%%%%%%%

	Let $f$ be a monic irreducible polynomial of degree $n$ over a number field $k$, and let $K$ be its splitting field.  By means of its action on the $n$ roots of $f$ over an algebraic closure, we may regard the Galois group $G:=\mathrm{Gal}(K/k)$ as a subgroup of the symmetric group $S_n$.  Each element $\sigma \in G$ therefore has a cycle type $\lambda = (\lambda_1,\dots,\lambda_r)$ for some partition $\lambda$ of $n$.  In fact, if $\mathfrak{p}$ is a prime of $k$ that is unramified in $K$ and so that $f \pmod{\mathfrak{p}}$ is separable, then the cycle type of the Frobenius element $\mathrm{Frob}_\mathfrak{p}$ agrees with the factorization type of $f \pmod{\mathfrak{p}}$.
	
	The Frobenius density theorem (a precursor to the Chebotarev density theorem) implies that as $\mathfrak{p}$ varies, the cycle type of $\mathrm{Frob}_\mathfrak{p}$ is equidistributed; more concretely, it states for any cycle type $\lambda$ occuring in $G$ that
		\[
			\frac{ \#\{ \N\mathfrak{p}  \leq x : \mathrm{Frob}_\mathfrak{p} \text{ has cycle type $\lambda$}\}}{ \#\{\N\mathfrak{p} \leq x\}}
				\sim \frac{\#\{\sigma \in G : \sigma \text{ has cycle type $\lambda$}\}}{|G|}
		\]
	as $x \to \infty$.  For many applications of the Chebotarev density theorem, this generally weaker statement is sufficient.  See, for example, the beautiful article of Stevenhagen and Lenstra \cite{LenstraStevenhagen-Chebotarev} for more on this theorem and the history of these results.
	
	In this paper, we develop a new approach to bound the least prime for which $\mathrm{Frob}_\mathfrak{p}$ has a given cycle type $\lambda$.  Compared to previous methods, which we survey more properly in the next section, this new approach is comparatively elementary, yet yields bounds that are often substantial improvements.  We highlight a few consequences of our method.

	\begin{theorem}\label{thm:least-prime-general}
		Let $f$ be a monic irreducible polynomial of degree $n$ over a number field $k$. Let $K$ be its splitting field, and let $G := \mathrm{Gal}(K/k)$ be its Galois group, viewed as a subgroup of $S_n$ as described above.  Let $\lambda = (\lambda_1,\dots,\lambda_r)$ be a cycle type occuring in $G$, and set $\ell := \mathrm{lcm}\{\lambda_1,\dots,\lambda_r\}$.  Then for any $\epsilon>0$, there is a degree $1$ prime $\mathfrak{p}$ of $k$ such that $\N\mathfrak{p}$ does not divide $\mathrm{Disc}(K)$ (the absolute discriminant of $K$),  $\mathrm{Frob}_\mathfrak{p}$ has cycle type $\lambda$, and $\N\mathfrak{p}$ satisfies the bound
			\[
				\N\mathfrak{p}
					\ll_{n,[k:\mathbb{Q}],\epsilon} |\mathrm{Disc}(K)|^{\frac{2^{\omega(\ell)-1}}{\ell} + \epsilon}
					\leq |\mathrm{Disc}(K)|^{\frac{1}{2} + \epsilon},
			\]
		where $\omega(\ell)$ denotes the number of distinct prime factors of $\ell$.  The implied constant above is effectively computable if $\epsilon > \frac{2}{|G|[k:\mathbb{Q}]}$, or if $[k:\mathbb{Q}]$ is odd and $G$ does not admit a nontrivial quadratic character.
	\end{theorem}
	
	In general, the cycle type of an element $\sigma \in G$ does not determine its conjugacy class, but for certain groups it does.  This is notably the case when $G=S_n$, and here, we in fact give a much stronger bound than Theorem~\ref{thm:least-prime-general}.
	
	\begin{theorem} \label{thm:Sn-cycle-intro}
		Let $K/k$ be a Galois extension of number fields with $\mathrm{Gal}(K/k) \cong S_n$ for some $n \geq 2$.  For a conjugacy class $C \subset S_n$, let $\ell_1,\dots,\ell_m \geq 2$ denote the lengths of the nontrivial cycles in  $C$, and define
			\[
				\alpha(S_n,C)
					:= 2^{m-1} \prod_{i=1}^m \frac{(\ell_i-2)2^{\ell_i-2}+1}{\ell_i!}.
			\]
		Then for any $\epsilon>0$, there is a prime $\mathfrak{p}$ of $k$, unramified in $K$, with $\mathrm{Frob}_\mathfrak{p} \in C$ satisfying
			\[
				\N\mathfrak{p}
					\ll_{n,[k:\mathbb{Q}],\epsilon} |\mathrm{Disc}(K)|^{\alpha(S_n,C)+\epsilon}.
			\]
		The implied constant is effectively computable if $\epsilon > \frac{2^m \prod_{i=1}^m (\ell_i-1)}{n! [k:\mathbb{Q}]}$.
	\end{theorem}
	
	When $C = [(1\,2\dots n)]$ is the class of an $n$-cycle, then $\alpha(S_n,C) = \frac{(n-2)2^{n-2}+1}{n!}$, which decays more than exponentially as $n\to \infty$ on appealing to Stirling's formula.  By contrast, the previously best bound for large $n$ \cite{ThornerZaman-Explicit} only shows there is a prime $\mathfrak{p}$ with $\mathrm{Frob}_\mathfrak{p}$ an $n$-cycle satisfying $\N\mathfrak{p} \ll_{n,[k:\mathbb{Q}]} |\mathrm{Disc}(K)|^{c/n}$ for an explicit large constant $c>0$.  In fact, the class of an $n$-cycle is simultaneously the best case for our method and the worst case for previous methods, and by optimizing the two methods against each other we obtain the following (simplified) asymptotic improvement over what each method yields individually.

	\begin{theorem} \label{thm:Sn-asymptotic-simplified}
		For any Galois extension of number fields $K/k$ with $\mathrm{Gal}(K/k) \cong S_n$ for some $n \geq 2$ and any conjugacy class $C \subset S_n$, there is a prime $\mathfrak{p}$ with $\mathrm{Frob}_\mathfrak{p} \in C$ satisfying
			\[
				\N\mathfrak{p}
					\ll_{n,[k:\mathbb{Q}]} |\mathrm{Disc}(K)|^{\frac{500}{\exp(n/4)}}.
			\]
	\end{theorem}
	We also give strong bounds for \emph{small} symmetric groups that are often the best known; see Theorem~\ref{thm:small-Sn-intro} below.

%%%%%%%%
%%%%%%%%	
\section{History and discussion of results}
	\label{sec:Results}
%%%%%%%%
%%%%%%%%

	We now turn to the promised proper discussion of previous results and to the statements of our main theorems.

%%%%%%%%
\subsection{Linnik exponents} 
%%%%%%%%
	Fix a finite group $G$.  Let $K/k$ be a Galois extension of number fields with Galois group $\Gal(K/k) \cong G$ and relative discriminant ideal $\kD_{K/k} \subseteq  \mathcal{O}_k$. Let $C \subseteq G$ be a conjugacy invariant subset of $G$. Every prime ideal $\kp \nmid \kD_{K/k}$ of $k$ has an associated Frobenius element $\Frob_{\kp} \in G$, well-defined up to conjugacy. The celebrated Chebotarev density theorem states that these Frobenius elements are equidistributed:
	\[
	\frac{|\{ \N\kp \leq x : \kp \nmid \kD_{K/k}, \Frob_{\kp} \in C \}|}{|\{ \N\kp \leq x  : \kp \nmid \kD_{K/k} \}|} \longrightarrow \frac{|C|}{|G|} \qquad \text{ as } x \to \infty,
	\]
	where $\N = \mathrm{N}_{k/\Q}$ is the absolute norm of $k$ over $\Q$. 
	An effective unconditional version was first established by Lagarias--Odlyzko \cite{LagariasOdlyzko-EffectiveChebotarev} and strengthened by Thorner--Zaman \cite{ThornerZaman-Unified}. Conditional versions assuming the grand Riemann hypothesis (GRH) have been investigated by many authors, e.g. 
	Lagarias--Odlyzko \cite{LagariasOdlyzko-EffectiveChebotarev}, Oesterl\'e \cite{Oesterle-CDTGRH}, Murty--Murty--Saradha \cite{MurtyMurtySaradha-ChebotarevI}, Bella\"{i}che \cite{Bellaiche-Chebotarev}, Greni\'{e}--Molteni \cite{GrenieMolteni-ExplicitCDTGRH}.

	Inspired by Linnik's famous theorem on the least prime in arithmetic progressions and its many applications, there is a vast literature studying the least prime ideal with Frobenius element lying in $C$. Assuming the grand Riemann hypothesis (GRH), Lagarias--Odlyzko \cite{LagariasOdlyzko-EffectiveChebotarev} (cf. Bach--Sorenson \cite{BachSorenson-ResidueClasses}) also showed there exists a prime ideal $\kp$ of degree one which does not ramify in $K$ and with $\Frob_{\kp} \in C$ such that 
	\[
	\N\kp \ll (\log |\Disc(K)|)^2.
	\]
	Unconditional bounds for the least prime ideal take many different forms due to the many number field invariants, but we will focus on a convenient form: there exists a prime ideal $\kp$ of $k$ of degree one which does not ramify in $K$ and with $\Frob_{\kp} \in C$ such that 
	\begin{equation} \label{eqn:LinnikExponent}
	\N\kp \ll_{|G|,[k:\Q],\epsilon} |\Disc(K)|^{\alpha(G,C)+\epsilon},		
	\end{equation}
	where $\epsilon > 0$ is arbitrary and $\alpha(G,C) \geq 0$ is some constant depending only on the finite group $G \cong \Gal(K/k)$ and the conjugacy invariant subset $C \subseteq G$. For example, the value $\alpha(G,C) = 0$ is a weak consequence of GRH. When stating our theorems, we shall refer to the quantity $\alpha(G,C)$ as a \textit{Linnik exponent} for $G$ and $C$. Most arguments, including ours, can make the implied constant in \eqref{eqn:LinnikExponent} explicit, or at least more uniform in some parameters, but we suppress these technicalities for simplicity.
	
	We separate our discussion of previous results into those relying on studying the zeros of Hecke and Artin $L$-functions and those not relying on such an analysis.  In particular, the two general unconditional bounds that hold for any finite group $G$ and any conjugacy invariant subset $C \subseteq G$ both rely on studying zeros.  The first method was pioneered by Lagarias, Montgomery, and Odlyzko \cite{LagariasMontgomeryOdlyzko-LeastPrime}, who showed that there is \emph{some} Linnik exponent $\alpha(G,C)$ in the sense above, and indeed that it may be taken to be a constant independent of $G$ and $C$.  This was made explicit by Zaman \cite{Zaman-Least}, and the current record via this approach is due to Kadiri, Ng, and Wong \cite{KadiriNgWong-LeastPrime}, who showed that
	\begin{equation}
		\label{eqn:LinnikExponent-KNW}
		\alpha(G,C) = 16
	\end{equation}
	is admissible in \eqref{eqn:LinnikExponent} for every $G$ and $C$.  The second general method is essentially due to Weiss \cite{Weiss}, and from its refinement by Thorner and Zaman \cite{ThornerZaman-Explicit} one may deduce (see Lemma~\ref{lem:abelian-least-prime}) that if $A \leq G$ is an abelian subgroup intersecting $C$, then one may take
	\begin{equation}
		\label{eqn:LinnikExponent-TZ}
		\alpha(G,C) = \frac{1042}{|A|}.
	\end{equation}

	Results not relying on studying the zeros of $L$-functions have been heretofore more ad hoc and restricted, but when applicable, have often yielded better explicit exponents, particularly for small groups or special choices of $C$.  For general $G$, we are aware of two types of results.  In the special case of completely split primes, i.e. $C = \{1\}$, Ge, Milinovich, and Pollack \cite{GeMilinovichPollack-Split} showed
	\begin{equation} \label{eqn:LinnikExponent-GMP}
		\alpha(G,\{1\}) = \frac{1}{2} 
	\end{equation}
	is an admissible Linnik exponent in \eqref{eqn:LinnikExponent}; note that this improves over \eqref{eqn:LinnikExponent-KNW} always, and over \eqref{eqn:LinnikExponent-TZ} unless $G$ has an abelian subgroup of order at least $2085$.  
	In fact, though they write it somewhat differently (and work only over $k=\mathbb{Q}$, though this is not essential), we consider the approach of Ge, Milinovich, and Pollack to be an extremely special case of our method; we comment on this more in the next section.
	
	The second general result we are aware of is in the exactly complementary situation, when we are looking for primes that do not split completely, i.e. $C = G \setminus \{1\}$.  Building on work of Murty and Patankar \cite{MurtyPatankar-Tate} and refining a result of Li \cite{Li-NonSplit}, Zaman \cite{Zaman-NonSplit} showed that if $k=\mathbb{Q}$, then
	\begin{equation} \label{eqn:LinnikExponent-Zaman}
		\alpha(G, G \setminus \{1\} ) = \frac{1}{4(1-2|G|^{-2/3})(|G|-1)},
	\end{equation}
	is admissible in \eqref{eqn:LinnikExponent}.  
	This argument makes use of standard properties of completely split primes and Dedekind zeta functions, including  the non-negativity of its Dirichlet series coefficients, but does not seem ripe for generalization.  Additionally, we note that this is coarser than determining the cycle type of a prime unless $G$ is cyclic of prime order.  A version of the result does hold over a general base field $k \ne \mathbb{Q}$, but it requires the discriminant of the extension $K/k$ to be sufficiently large relative to that of $k$.  As such, it does not yield the level of uniformity in $k$ asked for by \eqref{eqn:LinnikExponent}, but it is nevertheless noteworthy in not relying on zeros.
	
	We are aware of one final result that applies to non-abelian $G$.  For each $n \geq 3$ and each non-identity conjugacy class $C \subset S_n$, Cho and Kim \cite{ChoKim-SignChange} construct a character $\chi_C$ such that $\chi_C(g) < 0$ if and only if $g \in C$.  By assuming the Artin conjecture and comparing the properties of $L(s,\chi_C)$ against those of $L(s,\chi_C \otimes \chi_C)$, they obtain a bound on $\alpha(S_n,C)$.  However, the Artin conjecture in this setting is known only if $n=3$ or $n=4$, or if $n=5$ and $K$ is an $S_5$ extension of $\mathbb{Q}$ satisfying certain local conditions.  As such, it is only in these cases that their method applies unconditionally.  We comment more directly on the comparison between their bounds and ours in Section~\ref{subsec:small-groups-intro}.

%%%%%%%%
\subsection{Our approach}
	\label{subsec:approach}
%%%%%%%%

	While we consider our approach new for non-abelian $G$, it is modeled on extremely classical work of Vinogradov on the least quadratic non-residue, or equivalently on the least inert prime in a quadratic field, or equivalently when $G \cong \mathbb{Z}/2\mathbb{Z}$ and $k =\mathbb{Q}$.  Special cases of cyclic $G$ and $k =\mathbb{Q}$ were also studied by   Vinogradov--Linnik \cite{VinogradovLinnik-LeastPrimeQuadraticResidue} as well as Elliott \cite{Elliott-LeastPrimePowerResidue}. When $G$ is abelian and $k=\mathbb{Q}$, Pollack \cite{Pollack-PrimeSplittingAbelian} extended these ideas further, giving bounds and arguments resembling our own approach in many ways. These works crucially rely on the holomorphy of Dirichlet $L$-functions. Our analysis of all groups $G$ generalizes these works despite limited progress towards Artin's holomorphy conjecture, and despite the more subtle representation theory of non-abelian groups.
	
	To describe a slightly simplified form of our approach, we begin by letting $C$ be any conjugacy invariant subset of $G$. Assume $\lambda_{\Psi_+}$ and $\lambda_{\Psi_-}$ are any two multiplicative functions supported on the squarefree ideals of $\mathcal{O}_k$ coprime to $\mathfrak{D}_{K/k}$ such that 
	\begin{equation}
		\label{eqn:simplified-constraint}
	\lambda_{\Psi_+}(\mathfrak{p}) = \lambda_{\Psi_-}(\mathfrak{p}) \text{ for all primes $\mathfrak{p}$ with $\mathrm{Frob}_\mathfrak{p} \not\in C$.}
	\end{equation}
	(The notation $\Psi_+$ and $\Psi_-$ will be explained shortly.) If there is some $x$ such that there is no prime $\mathfrak{p}$ with $\mathrm{Frob}_\mathfrak{p} \in C$ satisfying $\N\mathfrak{p} \leq x$, then we would have  
	\begin{equation}\label{eqn:simplified-approach}
			\sum_{\N\mathfrak{n} \leq x} \lambda_{\Psi_+}(\mathfrak{n}) 
				= \sum_{\N\mathfrak{n} \leq x} \lambda_{\Psi_-}(\mathfrak{n}).
	\end{equation}
	This leads to a contradiction for large enough $x$ if the asymptotic behaviors of the summations on the two sides of \eqref{eqn:simplified-approach} are different.  For example, if $k = \mathbb{Q}$ and $d$ is a fundamental discriminant, we may take $\lambda_{\Psi_+}(p)$ to be $0$ or $1$ according to whether or not $p \mid d$, and we may take $\lambda_{\Psi_-}(p) = (\frac{d}{p})$; hence $\lambda_{\Psi_+}(p) = \lambda_{\Psi_-}(p)$ unless $(\frac{d}{p}) = -1$, i.e. $p$ is inert in $\mathbb{Q}(\sqrt{d})$.  The left-hand side of \eqref{eqn:simplified-approach} then is $\gg x / (\log\log |d|)$, while the right-hand side is $O_{\epsilon}(x^{1/2} |
	d|^{1/4+\epsilon})$ by a convexity bound (or equivalently P\'olya--Vinogradov).  This is a contradiction for $x \gg_\epsilon |d|^{1/2+\epsilon}$, leading to the Linnik exponent $\alpha(S_2, (1\,2)) = 1/2$.  This is the core idea behind Vinogradov's bound, though he showed that one may do better over $\mathbb{Q}$ with a bit of multiplicative number theory; in fact, by also incorporating better bounds on short character sums due to Burgess, it is now known that we may take $\alpha(S_2,(1\, 2)) = 1/4\sqrt{e}$ when $k=\mathbb{Q}$.  However, obtaining a contradiction to \eqref{eqn:simplified-approach} remains the best known strategy toward this classical problem, and it is this strategy we adapt to the Chebotarev setting.
	
	 Given the constraint \eqref{eqn:simplified-constraint}, it is natural to assume that $\lambda_{\Psi_+}$ and $\lambda_{\Psi_-}$ arise as the Dirichlet series coefficients of Artin $L$-functions $L(s,\Psi_+)$ and $L(s,\Psi_-)$ associated to Artin characters $\Psi_+$ and $\Psi_-$ of $G$ satisfying the equivalent constraint 
	 \[
	 \Psi_+(g) = \Psi_-(g) \qquad \text{ for } g \not\in C. 
	 \] 
	To obtain the desired contradiction in \eqref{eqn:simplified-approach}, we then wish for the summation of one function (say $\lambda_{\Psi_-}$) to cancel and for the other to grow asymptotically in $x$. Analogous approaches have been used for other arithmetically interesting problems, including effective multiplicity one theorems for $\mathrm{GL}(n)$ cusp forms (e.g. Brumley \cite{Brumley-EffectiveMultiplicityOne}).  When dealing with $L$-functions whose analytic properties are understood, this is ensured by asking for $L(s,\Psi_+)$ to have a pole at $s=1$ and for $L(s,\Psi_-)$ to be entire.  This reveals the reason our approach does not yield fully general bounds on the least prime in the Chebotarev density theorem: for certain pairs $(G,C)$, the simplest of which is $(C_3, (1\,2\,3))$, there do not exist $\Psi_+$ and $\Psi_-$ as above for which $L(s,\Psi_+)$ has an integral order pole (even conjecturally). To obtain a contradiction in \eqref{eqn:simplified-approach} would thus apparently require something akin to a subtle analysis via the Selberg--Delange method, and we do not carry this out.
	
	Instead, by taking $C$ to be the union of conjugacy classes with the same cycle type (or, slightly better, those classes that are \emph{rationally equivalent}; see \cref{subsec:results-general-groups}), we show there are many characters $\Psi_+$, $\Psi_-$ for which \eqref{eqn:simplified-approach} provably leads to a contradiction.  The value of $x$ at which this contradiction is obtained is related to the conductors of the $L$-functions $L(s,\Psi_+)$ and $L(s,\Psi_-)$ (see Theorem~\ref{thm:general-approach}), so to optimize the resulting Linnik exponent, one wishes to find a ``contradictory pair'' $(\Psi_+,\Psi_-)$ whose conductors are as small as possible.\footnote{Mechanically, the conductors arise from an appeal to the convexity bound for holomorphic Artin $L$-functions.  Assuming a suitable form of the Lindel\"of hypothesis would yield $\alpha(G,C) = 0$ via our method, while a subconvex bound would afford an improvement over what we have stated.  We have not pursued such improvements, preferring to keep our results fully general, but we note that such improvements should be available for many small groups over $k=\mathbb{Q}$.}  This leads to a technical subtlety of our work: if all Artin $L$-functions attached to the extension $K/k$ are entire away from $s=1$, then there is a natural optimal choice of $(\Psi_+,\Psi_-)$.  We analyze the resulting Linnik exponent in Theorem~\ref{thm:least-prime-rational-class-artin-intro}.  However, the holomorphy of Artin $L$-functions is not understood in general (this is the content of the Artin conjecture), so this ``optimal'' result is not known to apply unconditionally.  Instead, we prove there is a reasonably efficient and general choice $(\Psi_+,\Psi_-)$ for which holomorphy is known, and this leads to Theorem~\ref{thm:least-prime-general}.  By exploiting the explicit representation theory of the symmetric group, we provide a much more efficient choice when $G=S_n$, and this leads to Theorem~\ref{thm:Sn-cycle-intro}.  
	
	Additionally, even in cases where the Artin conjecture is not known, it can happen that the optimal choices of $\Psi_+,\Psi_-$ are nevertheless provably holomorphic.  For example, this happens for any $G$ when $C=\{1\}$ is the conjugacy class of the identity, where the optimal choices satisfy $L(s,\Psi_+) = \zeta_K(s)$ and $L(s,\Psi_-) = 1$. Comparing these two functions is at the heart of the argument of Ge, Milinovich, and Pollack \cite{GeMilinovichPollack-Split} mentioned above, and it is in this sense that their approach is a special case of ours.  By happenstance, the optimal choices are also provably holomorphic in some special cases, including for certain conjugacy classes when $G\cong S_5$ and $G \cong S_6$; we record the consequences of these and other explicit choices for small symmetric groups in Section~\ref{subsec:small-groups-intro}.
	
	Lastly, while we have indicated (truthfully!) that the Linnik exponent $\alpha(G,C)$ one obtains from this approach depends on the conductors of the $L$-functions $L(s,\Psi_+)$ and $L(s,\Psi_-)$, there is one final subtlety: it also depends on lower bounds for the residue $\Res_{s=1}L(s,\Psi_+)$.  Essentially as a consequence of the Brauer--Siegel theorem, it is known that $\Res_{s=1}L(s,\Psi_+) \gg_{[K:\mathbb{Q}],\epsilon} |\mathrm{Disc}(K)|^{-\epsilon}$ for every $\epsilon > 0$, but where the implied constant is possibly ineffective for $\epsilon$ close to $0$; this would translate into an ineffective implied constant in the Linnik bound \eqref{eqn:LinnikExponent} for sufficiently small $\epsilon$.  However, it has been known since work of Stark \cite{Stark-EffectiveBrauerSiegel} that the source of this ineffectivity is confined to the quadratic subfields of $K$ (though see also \cite{CLOZ}).  Consequently, in our theorems below, we consider it an important point to indicate when the implied constant is effectively computable, but this sometimes entails a discussion of the quadratic subfields of the base field $k$ and the quadratic characters of the group $G$.

%%%%%%%%
\subsection{Results for general groups} \label{subsec:results-general-groups}
%%%%%%%%
	
	The approach described above applies to arbitrary Galois groups $G$ as follows.  Given a finite group $G$, two elements $g,h \in G$ are \emph{rationally equivalent} if they generate conjugate cyclic subgroups.  If $g$ and $h$ are conjugate, then they are also rationally equivalent, so each rational equivalence class breaks up as a union of conjugacy classes.  In the symmetric group $S_n$, two elements are rationally equivalent if and only if they are conjugate, so each rational equivalence class consists of a single conjugacy class, but this need not be true for arbitrary $G$.  In general, there is a sequence of implications
		\[
			\text{conjugacy} \Longrightarrow \text{rational equivalence} \Longrightarrow \text{same cycle decomposition},
		\]
	where the cycle decomposition of $g$ and $h$ is determined by any homomorphism $G \to S_n$.  In particular, rational equivalence is sufficient to determine the splitting types of primes in any subextension of the Galois extension $K/k$, i.e. its cycle type.  
	
	The construction is a little too technical to describe here, but we show in Proposition~\ref{prop:first-construction} below that there is always a contradictory pair $(\Psi_+,\Psi_-)$ in the sense above that suffices to detect an arbitrary rational equivalence class $C \subset G$.  This leads to the following generalization of Theorem~\ref{thm:least-prime-general}.
	
	\begin{theorem}\label{thm:least-prime-rational-class-intro}
		Let $G$ be any finite group, let $C$ be a rational equivalence class in $G$, let $n$ be the order of an element $g \in C$, and let $\omega(n)$ be the number of distinct primes dividing $n$.  The Linnik exponent
			\[
			\alpha(G,C) = \frac{2^{\omega(n)-1}}{n}
			\]
		is admissible in \eqref{eqn:LinnikExponent}, and the implied constant in \eqref{eqn:LinnikExponent} is effectively computable if $\epsilon > \frac{2}{
		|G|[k:\mathbb{Q}]}$, or for all $\epsilon>0$ if $k$ does not admit a quadratic subfield and there is no quadratic character $\chi$ of $G$ such that $\chi(g) = 1$ for some $g \in C$.
	\end{theorem}
	
	\begin{remark}
		Our method suffices also to give a lower bound on the number of primes produced in this range, and in particular, it produces more than one such prime.  See Theorem~\ref{thm:general-approach} for the full statement of such results.
	\end{remark}
		
	As a simple example, we have the following.
	
	\begin{corollary}
		Let $p$ be prime, let $F/k$ be a degree $p$ extension, and let $K$ be its normal closure.  
		Then there is a prime $\mathfrak{p}$ of $k$ that is inert in $F$, has degree $1$ over $\Q$, and satisfies
			\[
				\N\mathfrak{p} \ll_{[k:\Q],p,\epsilon} |\mathrm{Disc}(K)|^{1/p + \epsilon}.
			\]
		The implied constant is effectively computable for $\epsilon > \frac{2}{|G|[k:\mathbb{Q}]}$ in general, and for all $\epsilon > 0$ if $k$ does not contain a quadratic subfield and if $G$ does not admit a quadratic character.
	\end{corollary}
	\begin{remark}
		In the special case that $G$ is cyclic of order $p$ (or, more generally, a Frobenius group of degree $p$), then the exponent $1/p + \epsilon$ above may be improved to $1/2(p-1) + \epsilon$.  This follows from Theorem~\ref{thm:semidirect-product} below.
	\end{remark}
	
	The Chebotarev density theorem can be improved under the assumption of Artin's holomorphy conjecture (see, e.g., Murty \cite{Murty-ChebotarevII} and Thorner--Zhang \cite{ThornerZhang}). Our Linnik exponents can also be similarly improved for general groups $G$, as we may then use the optimal choices of $\Psi_+$ and $\Psi_-$. Below is a simplified version of what one obtains in this case. 

	\begin{theorem}\label{thm:least-prime-rational-class-artin-intro}
	Let $G$ be any finite group and let $C$ be a rational equivalence class in $G$. Let $H \leq G$ be a subgroup such that $H$ meets $C$. Let $C_H(g)$ be the centralizer subgroup in $H$ of an element $g \in C \cap H$. If the Artin $L$-function $L(s,\chi)$ is entire for each nontrivial irreducible character $\chi$ of $H$ such that $\chi(g) \ne 0$, then the Linnik exponent
		\[
			\alpha(G,C) = \frac{|C_H(g)|^{1/2}}{2|H|^{1/2}}
		\]
	is admissible in \eqref{eqn:LinnikExponent}.  The implied constant in \eqref{eqn:LinnikExponent} is effectively computable if $\epsilon > \frac{2}{
	|G|[k:\mathbb{Q}]}$, or if the subfield $K^H$ of $K$ fixed by $H$ does not admit a quadratic subfield and there is no quadratic character $\chi$ of $H$ such that $\chi(g) = 1$.
	\end{theorem}	

	We will use Theorem~\ref{thm:least-prime-rational-class-artin-intro} later in Section~\ref{subsec:alternative-Sn} to give unconditional bounds on certain classes in symmetric groups (including bounds on $\alpha(S_n, [(12)])$, for example) that are stronger than either Theorem~\ref{thm:Sn-cycle-intro} or Theorem~\ref{thm:least-prime-rational-class-intro}, but we give here a different example of its utility.
	
	\begin{corollary}
		\label{cor:unipotent}
		Let $q$ be a prime power, let $r \geq 2$ be an integer, and let $G = \mathrm{GL}_r(\mathbb{F}_q)$.  Let $C$ be the set of unipotent elements of $G$ consisting of a single Jordan block (equivalently, fixing a unique line in $\mathbb{F}_q^r$).  Then the Linnik exponent
			\[
				\alpha(G,C) = \frac{1}{2q^{(r-1)(r-2)/4}}
			\]
		is admissible in \eqref{eqn:LinnikExponent}. 
	\end{corollary}
	
	An astute reader might notice that when $G = S_n$ and $C$ is the class of an $n$-cycle, the unconditional bound given by \cref{thm:Sn-cycle-intro} is much stronger than that provided by the conditional bound \cref{thm:least-prime-rational-class-artin-intro}.  This is because Theorem~\ref{thm:Sn-cycle-intro} exploits the explicit representation theory of the symmetric group, while Theorem~\ref{thm:least-prime-rational-class-artin-intro} is proved for general groups $G$ without exploiting detailed knowledge of their representation theory. 
	We provide less simplified bounds than Theorem~\ref{thm:least-prime-rational-class-artin-intro} in Section~\ref{subsec:less-simplified} that may yield improved bounds with greater input from representation theory.  
	In our view, though, \cref{thm:Sn-cycle-intro,thm:least-prime-rational-class-artin-intro} are reasons for optimism that Linnik exponents for other groups $G$ can hopefully be unconditionally improved beyond what \cref{thm:least-prime-rational-class-intro} offers.
	
	We build on this by now exploring other explicit examples of what our method yields.

%%%%%%%%	
\subsection{Explicit results for small symmetric groups} 
	\label{subsec:small-groups-intro}
%%%%%%%%

	Besides Theorems~\ref{thm:Sn-cycle-intro} and \ref{thm:Sn-asymptotic-simplified}, which yield strong Linnik exponents for large symmetric groups, our approach also yields the best-known Linnik exponent for small symmetric groups outside of a few special cases when $k=\mathbb{Q}$ that we discuss below.  In particular, we have:
	
	\begin{theorem} 
		\label{thm:small-Sn-intro}
		Let $3 \leq n \leq 10$. For all non-identity conjugacy classes $C \subset S_n$, the Linnik exponent $\alpha(S_n,C)$ given in the following table is  admissible in \eqref{eqn:LinnikExponent}: 
		\begin{center}
			\begin{tabular}{|l|c|c|c|c|c|c|c|c|}
				\hline $n$ & $3$ & $4$ & $5$ & $6$ & $7$ & $8$ & $9$ & $10$ \\
				\hline 
				$\alpha(S_n,C)= $ & \bigg. $\dfrac{1}{4}$ & $\dfrac{1}{6} $ & $\dfrac{13}{120}$ & $\dfrac{31}{360}$ & $\dfrac{17}{252}$ & $\dfrac{331}{5760}$ & $\dfrac{17819}{362880}$ & $\dfrac{78941}{1814400}$ \\
				\hfill $\leq$ & $0.250$ & $0.167$ & $0.108$ & $0.086$ & $0.067$ & $0.057$ & $0.049$ & $0.043$ \\ \hline
				$\epsilon_\mathrm{eff}^\mathbb{Q}(n)$ \bigg. & $\dfrac{1}{3}$ & $\dfrac{1}{12}$ & $\dfrac{1}{30}$ & $0$ & $0$ & $\dfrac{1}{3360}$ & $\dfrac{11}{181440}$ & $\dfrac{1}{181440}$ \\ \hline
			\end{tabular}
		\end{center}
	Furthermore, if $k$ does not contain a quadratic subfield, the implied constant in \eqref{eqn:LinnikExponent} is effectively computable for $\epsilon > \frac{\epsilon_\mathrm{eff}^\mathbb{Q}(n)}{[k:\Q]}$. If $k$ does contain a quadratic subfield, it is effectively computable when $\epsilon > \max\left(\frac{2}{n! [k:\mathbb{Q}]}, \frac{\epsilon_\mathrm{eff}^\mathbb{Q}(n)}{[k:\Q]}\right)$.
	\end{theorem} 
	
	Several remarks are in order.  First, we give more refined Linnik exponents $\alpha(S_n,C)$ that depend on $C$ in Theorem~\ref{thm:small-Sn-bowels} below; the table above collects the \emph{worst case} bounds from that theorem.  Second, Theorem~\ref{thm:small-Sn-intro} excludes $C = \{1\}$, the conjugacy class of the identity, because our method reproduces the Linnik exponent $\alpha(S_n,\{1\}) = 1/2$ that was previously proved by Ge--Milinovich--Pollack \cite{GeMilinovichPollack-Split}. Third, it is an exercise to check (see Lemma~\ref{lem:abelian-bound}) that the largest abelian subgroup $A$ of $S_n$  satisfies $|A| < 2085$ for $n \leq 20$; as a result, the Linnik exponents \eqref{eqn:LinnikExponent-KNW},\eqref{eqn:LinnikExponent-TZ} given by zero-based approaches will be inferior to the bound $\alpha(S_n,\{1\}) = 1/2$, to those given in Theorem~\ref{thm:small-Sn-intro} when $n \leq 10$, and to those given by Theorem~\ref{thm:least-prime-rational-class-intro} when $11 \leq n \leq 20$.  We expect our method to yield even stronger results when $n \geq 11$ (as is partially justified by Theorems~\ref{thm:Sn-cycle-intro} and \ref{thm:Sn-asymptotic-simplified}) but Theorem~\ref{thm:small-Sn-intro} is proved by means of an explicit computation that we have not carried out for any $n \geq 11$.  However, when $n \geq 11$, better bounds than those provided by Theorem~\ref{thm:least-prime-rational-class-intro} (and sometimes even than Theorem~\ref{thm:Sn-cycle-intro}) are available by combining Theorem~\ref{thm:product-of-cycles-characters} and Lemma~\ref{lem:alternate-Sn-choices} as suggested in Lemma~\ref{lem:product-of-characters} below; we carry this out in an example in Section~\ref{subsec:alternative-Sn}.
	
	Lastly, when $k=\mathbb{Q}$, there are three cases where Cho and Kim \cite{ChoKim-SignChange} give stronger bounds than those provided by Theorem~\ref{thm:small-Sn-intro} -- or, indeed, its stronger variant Theorem~\ref{thm:small-Sn-bowels} below.  To our knowledge, these are the only cases where Theorem~\ref{thm:small-Sn-intro} is not the best known bound.  We discuss these three cases quickly in turn:
		\begin{itemize}
			\item $G=S_3$, $C=[(12)]$:  Let $p$ be a prime unramified in an $S_3$ extension of $\mathbb{Q}$. Then, the prime $p$ has has $\mathrm{Frob}_p \in C$ if and only if it is inert in the quadratic resolvent field $\mathbb{Q}(\sqrt{d})$. By exploiting the work of Vinogradov and Burgess mentioned earlier, they obtain the Linnik exponent $\alpha(S_3,[(1\,2)]) = {1}/{8\sqrt{e}}$.  

			\item $G=S_3$, $C=[(123)]$: If $K/\mathbb{Q}$ is a complex $S_3$ extension, by exploiting a result on sign changes of modular forms \cite{Matomaki-SignChange}, they obtain $\alpha(S_3,C) = 3/16$.  If instead $K/\mathbb{Q}$ is a real $S_3$ extension, by exploiting a  result on sign changes of Maass forms \cite{Qu-SignChange} using a subconvexity bound for Maass form $L$-functions \cite{MichelVenkatesh-SubconvexityGL2}, they obtain $\alpha(S_3,C) = 1/4 - \delta$ for some $\delta > 0$. 

			\item $G=S_4$, $C=[(123)]$: Let $p$ be a prime unramified in an $S_4$ extension of $\mathbb{Q}$.  Then $p$ has $\mathrm{Frob}_p \in C$ if and only if its Frobenius in the resolvent $S_3$ extension is also a $3$-cycle.  Exploiting this and the previous example, they obtain either $\alpha(S_4,C)=3/64$ or $\alpha(S_4,C)=1/16 - \delta/4$ depending on the signature of the $S_3$ resolvent.
		\end{itemize}
	
	Our method as described below would easily incorporate some of these improvements (notably those arising from subconvex bounds), while others would take more care and may not work fully uniformly over general number fields as we do.
	
\subsection{Explicit results for some non-symmetric groups}	
	
	While the applicability of our approach to symmetric groups is somewhat special, in that rational equivalence classes and conjugacy classes are the same in this case, we do not expect the \emph{strength} of the Linnik exponents our approach yields to be substantially worse for general groups.	
	This is already partially demonstrated in results like Corollary~\ref{cor:unipotent}, but we also have the following obtained from a computation analogous to that leading to Theorem~\ref{thm:small-Sn-intro}.

	\begin{theorem} \label{thm:simple-examples}
	Let $C$ be a non-identity rational equivalence class of a finite simple group $G$.
	\begin{itemize}
		\item If $G$ is isomorphic to one of the Mathieu groups $M_{11}, M_{12}, M_{22}, M_{23},$ and $M_{24}$, then the Linnik exponent $\alpha(G,C) =\frac{131}{2280} < 0.046$ is admissible in \eqref{eqn:LinnikExponent}. 
		\item If $G$ is isomorphic to $\mathrm{PSL}_2(\mathbb{F}_q)$ for some prime power $q$ satisfying $7 \leq q \leq 23$, then the Linnik exponent $\alpha(G,C) = \frac{19}{112} < 0.17$ is admissible in \eqref{eqn:LinnikExponent}.
	\end{itemize}
	In both cases, the implied constant is effectively computable for all $\epsilon > 0$ unless $k$ has a quadratic subfield, in which case it is effectively computable whenever $\epsilon > \frac{2}{|G|[k:\mathbb{Q}]}$.
	\end{theorem}

	\begin{remark}
		As with Theorem~\ref{thm:small-Sn-intro}, this theorem is intended as a summary; in Section~\ref{sec:other-groups}, we prove a more precise version that yields smaller Linnik exponents for most rational classes.
	\end{remark}
	
	We close with two more general examples, showing how our approach can meaningfully apply outside of the context of symmetric groups.
	\begin{theorem}
		\label{thm:trace-zero}
		Let $q$ be an odd prime power, and let $C \subset \mathrm{GL}_2(\mathbb{F}_q)$ consist of those matrices with trace $0$.  Then the Linnik exponent
			\[
				\alpha(\mathrm{GL}_2(\mathbb{F}_q),C)	
					= \frac{1}{4(q^2+1)} + \frac{1}{2(q^3-q)}
			\]
		is admissible in \eqref{eqn:LinnikExponent}.  If $k$ does not admit a quadratic subfield and $-1$ is not a square in $\mathbb{F}_q^\times$, the implied constant is effecitvely computable for all $\epsilon>0$.  Otherwise, it is effectively computable for all $\epsilon >  \frac{2}{(q^2-1)(q^2-q) [ k:\mathbb{Q}]}$.
	\end{theorem}
	
	\begin{theorem}
		\label{thm:semidirect-product}
		Let $p$ be prime, let $r \geq 1$, and let $G_0 \leq \mathrm{GL}_r(\mathbb{F}_p)$ act irreducibly on $\mathbb{F}_p^r$.  Let $G = \mathbb{F}_p^r \rtimes G_0$ and let $C = \mathbb{F}_p^r \setminus \{0\} \subset G$.  Then
			\[
				\alpha(G,C)
					= \frac{1}{2(p^r-1)}
			\]
		is an admissible Linnik exponent in \eqref{eqn:LinnikExponent}.  The implied constant is effectively computable for all $\epsilon > \frac{2}{p^r|G_0|[k:\mathbb{Q}]}$in general, and for all $\epsilon > 0$ if $k$ does not admit a quadratic subfield and $G_0$ does not admit a nontrivial quadratic character.
	\end{theorem}

%%%%%%%%%		
\subsection{Organization and layout of this paper}
%%%%%%%%%
	This paper is organized as follows.
	
	In Section~\ref{sec:prelim}, we prepare some preliminaries on Artin $L$-functions and basic estimates. 
	
	In Section~\ref{sec:approach}, we outline our general approach, and prove \cref{thm:general-approach} which reduces the problem of detecting the least prime in a rational equivalence class to a problem purely in the character theory of the finite group $G$.
	
	In Section~\ref{sec:class-function}, we analyze this character theory problem in general, which leads to a proof of Theorems~\ref{thm:least-prime-rational-class-intro} and \ref{thm:least-prime-rational-class-artin-intro}, as well as the less simplified Corollary~\ref{cor:less-simplified-conductor}.  
	
	In Section~\ref{sec:symmetric-groups}, we analyze this character theory problem for symmetric groups in particular, which leads to the proof of Theorem~\ref{thm:Sn-cycle-intro}.  We also carry out in this section the optimization between arguments that leads to the proof of Theorem~\ref{thm:Sn-asymptotic-simplified}.
	
	In Section~\ref{sec:small-symmetric}, we specialize to small symmetric groups and establish \cref{thm:small-Sn-intro}.  We also introduce in Section~\ref{subsec:alternative-Sn} alternative class functions that will yield better results for certain classes in general symmetric groups.
	
	In Section~\ref{sec:other-groups}, we carry out the proofs of Corollary~\ref{cor:unipotent} and Theorems~\ref{thm:simple-examples}--\ref{thm:semidirect-product}.

%%%%%%%%%
%%%%%%%%%	
\section*{Acknowledgements}
The authors would like to thank Jesse Thorner for many useful conversations, and Paul Pollack for suggesting related references when $G$ is abelian. 
PJC was supported by the National Research Foundation of Korea (NRF) grant funded by the Korea government (MSIT) (No. RS-2022-NR069491 and No. RS-2025-02262988 ).  RJLO was partially supported by NSF grant DMS-2200760 and by the Office of the Vice Chancellor for Research at the University of Wisconsin-Madison with funding from the Wisconsin Alumni Research Foundation.  AZ was partially supported by NSERC grant RGPIN-2022-04982.  
%%%%%%%%%
%%%%%%%%%

%%%%%%%%%
%%%%%%%%%
\section{Preliminaries on Artin $L$-functions}
	\label{sec:prelim}
%%%%%%%%%
%%%%%%%%%

To keep the exposition self-contained, we will fix notation for Artin $L$-functions as well as gather some standard lemmas. This discussion closely follows \cite[Section 4]{LemkeOliverThornerZaman-ApproximateArtin} and more details can be found in \cite[Chapter 2]{MurtyMurty-Nonvanishing}.  

\subsection{Artin characters} Let $K/k$ be a Galois extension of number fields with Galois group $G=\Gal(K/k)$. Let $\mathcal{O}_k$ be the ring of integers of $k$. Let $\N = \mathrm{N}_{k/\Q}$ denote the absolute norm for integral ideals of $k$. For each prime $\kp$ of $k$ and each prime $\mathfrak{P}$ of $K$ lying over $\mathfrak{p}$, let $D_{\mathfrak{P}}=\Gal(K_{\mathfrak{P}}/k_{\kp})$, where $K_{\mathfrak{P}}$ and $k_{\kp}$ are the completions of $K$ and $k$ at $\mathfrak{P}$ and $\kp$, respectively.  Let $F_\mathfrak{P}$ and $F_\mathfrak{p}$ denote the residue fields of $\mathfrak{P}$ and $\mathfrak{p}$.  There is a map from $D_{\mathfrak{P}}$ to $\Gal(F_{\mathfrak{P}}/F_{\kp})$ that is surjective by Hensel's lemma.  Define $I_{\mathfrak{P}}$ to be the kernel of this map; we then have an exact sequence
\[
1\to I_{\mathfrak{P}}\to D_{\mathfrak{P}}\to \Gal(F_{\mathfrak{P}}/F_{\kp})\to 1.
\]
The group $\Gal(F_{\mathfrak{P}}/F_{\kp})$ is cyclic with generator $x\mapsto x^{\N\kp}$.  Choose $\sigma_{\mathfrak{P}}\in D_{\mathfrak{P}}$ whose image in $\Gal(F_{\mathfrak{P}}/F_{\kp})$ is this generator; it is only defined modulo $I_{\mathfrak{P}}$.  We have $I_{\mathfrak{P}}=1$ for all unramified $\kp$, so for these $\kp$, $\sigma_{\mathfrak{P}}$ is well-defined.   If we choose another prime $\mathfrak{P}'$ above $\kp$, then $I_{\mathfrak{P}'}$ and $D_{\mathfrak{P}'}$ are conjugates of $I_{\mathfrak{P}}$ and $D_{\mathfrak{P}}$.  For $\kp$ unramified, we denote by $\sigma_{\kp}$ the conjugacy class of Frobenius automorphisms at primes $\mathfrak{P}$ above $\kp$.

Let $\psi : G \to \mathbb{C}$ be a character of $G$ with degree $\psi(1)$. Let $\rho_{\psi}\colon G \to\mathrm{GL}_{\psi(1)}(\mathbb{C})$ be its corresponding complex representation of $G$, and let $V$ be the underlying complex vector space on which $\rho_{\psi}$ acts.  We may restrict this action to the decomposition group $D_{\mathfrak{P}}$ and see that the quotient $D_{\mathfrak{P}}/I_{\mathfrak{P}}$ acts on the subspace $V^{I_{\mathfrak{P}}}$ of $V$ on which $I_{\mathfrak{P}}$ acts trivially.  Any $\sigma_{\mathfrak{P}}$ will have the same characteristic polynomial on this subspace.

Let $\mathbf{1}_G$ denote the character of the trivial representation on $G$. Let $\langle \cdot \, ,\,  \cdot \rangle_G$  denote the $G$-inner product on characters, so $\langle \psi_1, \psi_2 \rangle_G =  \frac{1}{|G|}\sum_{g \in G} \psi_1(g) \overline{\psi_2}(g)$ for any characters $\psi_1,\psi_2$ of $G$. We shall write $\langle \cdot \, ,\,  \cdot \rangle$ instead of $\langle \cdot \, ,\,  \cdot \rangle_G$ when there is no ambiguity about the group.

\subsection{Local and global Artin $L$-functions}
For the nonarchimedean place $v_{\kp}$ of $k$ associated to the prime ideal $\kp$, we define the local Artin $L$-function of $\psi$ at $\mathfrak{p}$ by  
\begin{equation}
	\label{eqn:Artin-LocalDefinition}
L_{v_{\kp}}(s,\psi)=\det(1 - \rho_{\psi}(\sigma_{\mathfrak{P}})|V^{I_{\mathfrak{P}}}\N \kp^{-s})^{-1} = \prod_{j=1}^{\psi(1)} \Big(1 - \frac{\alpha_{j,\psi}(\mathfrak{p})}{(\mathrm{N}\mathfrak{p})^s} \Big)^{-1} \qquad \text{for } \Re(s) > 0,
\end{equation}
where  $\{ \alpha_{j,\psi}(\mathfrak{p}) \}_{j=1}^{\psi(1)}$ are the complex numbers satisfying $|\alpha_{j,\psi}(\mathfrak{p})| \leq 1$. Hence,  $L_{\mathfrak{p}}(s,\psi)$ does not vanish for $\Re(s) > 0$. Note that the matrix $\rho_{\psi}(\sigma_{\mathfrak{P}})|V^{I_{\mathfrak{P}}}$ remains the same if one changes the prime $\mathfrak{P}$ lying above $\kp$; indeed, if $\kp$ is unramified, then $\rho_{\psi}(\sigma_{\mathfrak{P}})|V^{I_{\mathfrak{P}}}=\rho_{\psi}(\sigma_{\kp})$.  

Since $|\alpha_{j,\psi}(\kp)|\leq 1$ for all $j$ and $\kp$, we may define the (global) Artin $L$-function of $\psi$ to be 
\begin{equation}
	\label{eqn:Artin-EulerProduct}
L(s,\psi) = \prod_{\kp}L_{v_{\kp}}(s,{\psi})
\qquad \text{for } \Re(s) > 1,
\end{equation}
where the absolutely convergent Euler product runs over all prime ideals $\mathfrak{p}$ of $k$. By formally expanding this Euler product, this Artin $L$-function can be rewritten as a Dirichlet series over rational integers and an Euler product rational primes: 
\begin{equation}
	\label{eqn:Artin-RationalEulerProduct}
L(s,\psi) = \sum_{n=1}^{\infty} \lambda_\psi(n) n^{-s} = \prod_p \Big(1 + \sum_{j=1}^{\infty} \frac{\lambda_\psi(p^j)}{p^{js}} \Big) \qquad \text{for } \Re(s) > 1, 
\end{equation}
where $\lambda_{\psi}$ is a multiplicative function defined via this identity. Thus, for any prime $p$, 
\begin{equation} \label{eqn:Artin-LocalFactorComparison}
1 + \sum_{j=1}^{\infty} \frac{\lambda_\psi(p^j)}{p^{js}} = \prod_{\kp \, \mid \, (p)} \prod_{j=1}^{\psi(1)} \Big(1 - \frac{\alpha_{j,\psi}(\mathfrak{p})}{(\mathrm{N}\mathfrak{p})^s} \Big)^{-1}
\end{equation}
and, for any prime $p \nmid D_K$, 
\begin{equation} \label{eqn:Artin-PrimeCoefficient}
\lambda_{\psi}(p) = \sum_{\N\kp = p} \sum_{j=1}^{\psi(1)} \alpha_{j,\psi}(\kp)  = \sum_{\N\kp = p} \psi(\sigma_{\kp}). 
\end{equation}
Moreover, for any integer $n \geq 1$,  as there are $[k:\Q]$ prime ideals of $k$ above a rational prime, 
\begin{equation} \label{eqn:Artin-DivisorBound}
|\lambda_\psi(n)| \leq \tau_{\psi(1) [k:\mathbb{Q}]}(n),
\end{equation}
where $\tau_{m}$ is the $m$-fold divisor function over the integers with Dirichlet series equal to the $m$-th power of the Riemann zeta function.

\subsection{Completed Artin $L$-function and Artin's conjecture} 
Let $a = a(\psi)$ be the dimension of the $+1$ eigenspace of complex conjugation. For each archimedean place $v$ of $k$, define
\[
L_v(s,\psi) = \begin{cases}
	\Gamma_{\R}(s)^{n}\Gamma_{\R}(s+1)^n&\mbox{if $k_v=\mathbb{C}$,}\\
\Gamma_{\R}(s)^a \Gamma_{\R}(s+1)^{n-a}&\mbox{if $k_v=\R$,}
\end{cases}
\]
where $\Gamma_{\R}(s) := \pi^{-s/2}\Gamma(s/2)$. Define the gamma factor of $\psi$ by 
\[
\gamma(s,\psi) :=\prod_{\textup{$v$ archim.}}L_v(s,\psi).
\]
Let the integral ideal $\kf_{{\psi}}\subseteq\mathcal{O}_k$  denote the arithmetic conductor of ${\psi}$ over $k$ (see \eqref{eqn:artin-conductor-def}), and let $q(\psi)$ be the analytic conductor of $\psi$ over $k$ which is given by 
\begin{equation} \label{eqn:Artin-AnalyticConductor}
q(\psi) := |\Disc(k)|^{\psi(1)} \N\kf_{{\psi}}.	
\end{equation}
The completed Artin $L$-function of $\psi$ is therefore defined as
\begin{equation}
\label{eqn:Artin-Completed}
\Lambda(s,{\psi}):= q(\psi)^{s/2} (s(s-1))^{\langle \psi, \mathbf{1}_G\rangle} L(s,\psi)\gamma(s,\psi),
\end{equation}
where $\langle \psi, \mathbf{1}_G \rangle$ is the $G$-inner product of $\psi$ with the character $\mathbf{1}_G$ of the trivial representation of $G$. The root number $W(\psi)$ is a complex number of modulus one such that
\[
	\Lambda(s,\psi) = W(\psi) \Lambda(1-s,\overline{\psi})
\]
for all $s\in\mathbb{C}$ at which $\Lambda(s,\psi)$ is holomorphic, where $\overline{\psi}$ is the complex conjugate of $\psi$.

The Artin conjecture asserts that $\Lambda(s,\psi)$ is entire or, equivalently, that  $L(s,\psi)$ can be analytically continued to  $\mathbb{C}$ apart possibly from a pole at $s=1$ of order $\langle \psi, \mathbf{1}_G \rangle$.  However, this conjecture is unknown in general. As a special case, if $\psi$ is $1$-dimensional, then Artin reciprocity shows that $L(s,\psi)$ is a Hecke $L$-function and hence   $(s-1)^{\langle \psi, \mathbf{1}_G\rangle} L(s,\psi)$  is entire.

\subsection{Monomial characters} Artin $L$-functions respect representation theoretic properties. 
First, for any character $\psi$, if there are  coefficients $c_i \in \Q$ and characters $\chi_i$ of $G$ such that
\begin{equation} 
	\label{eqn:Artin-Summation}
	\psi = \sum_{i=1}^k c_i \chi_{i} \quad \text{ then } \quad	L(s,\psi) = \prod_{i=1}^k L(s,\chi_{i})^{c_i} \quad \text{ and }\quad q(\psi) = \prod_{i=1}^k q(\chi_i)^{c_i}.  
\end{equation}
Second, if $\phi$ is a character of a subgroup $H \subseteq G$, then $\Ind_H^G \phi$ is the character of $G$ induced by $\phi$ and there is an equality of the associated $L$-functions:
\begin{equation} 
	\label{eqn:Artin-Induction}
L(s,\phi) = L(s,\Ind_H^G \phi).
\end{equation}
Notice $L(s,\phi)$ and $L(s,\Ind_H^G \phi)$ are respectively associated to the extensions $K/K^H$ and $K/k$, where $K^H$ denotes the subfield fixed by $H$. 

A character of a group $G$ is monomial if it is induced from a $1$-dimensional character  of a subgroup $G$, i.e. $\Ind_H^G \phi$ is a monomial character if $\phi$ is $1$-dimensional. Class field theory implies that $L(s,\phi)$ is a Hecke $L$-function, hence entire apart from a simple pole at $s=1$ if $\phi$ is trivial.  We record two useful but standard lemmas, and introduce a convenient definition.

\begin{lemma}[Brauer] \label{lem:Brauer}
Every character $\psi$ of a finite group $G$ is an integer linear combination of monomial characters of $G$.  
As a consequence, every Artin $L$-function $L(s,\psi)$ continues meromorphically to all of $\mathbb{C}$ and the order of the possible pole at $s=1$ is precisely $\langle \psi, \mathbf{1}_G\rangle$.
\end{lemma}

\begin{lemma}
	\label{lem:Monomial}
	If  $K/k$ is a Galois extension of number fields with $G = \mathrm{Gal}(K/k)$ and $\psi$ is a character of $G$ expressible as a non-negative $\mathbb{R}$-linear combination of monomial characters of $G$, then $L(s,\psi)$ is entire apart possibly from a pole at $s=1$ of order $\langle \psi,\mathbf{1}_G\rangle$. 
\end{lemma}

\begin{definition} 
	The \textit{monomial cone} of a finite group $G$ is the set of characters of $G$ expressible as non-negative $\mathbb{R}$-linear combinations of monomial characters of $G$. 
\end{definition}

\subsection{Convexity and residues} We record two standard estimates for Artin $L$-functions. First, we cite a convexity estimate assuming holomorphy.

\begin{lemma} \label{lem:Convexity}
	Let $K/k$ be a Galois extension of number fields with $G = \mathrm{Gal}(K/k)$. Let $\psi$ be a character of $G$. If the function $(s-1)^{\langle \psi, \mathbf{1}_G\rangle} L(s,\psi)$	is entire then, for $\epsilon > 0$ and $s = \sigma+it$ with $0 \leq \sigma \leq 1$ and $t \in \R$, 
	\[
	\Big| \Big(\frac{s-1}{s+1}\Big)^{\langle \psi, \mathbf{1}_G\rangle} L(s,\psi)\Big| \ll_{\psi(1), [k:\Q],\varepsilon}   \big( q(\psi) |1+it|^{\psi(1) [k:\mathbb{Q}]} \big)^{\frac{1-\sigma}{2}+\varepsilon}. 
	\] 
\end{lemma}
\begin{proof}
See, for example, an explicit version by \cite[Lemma 4.1]{LemkeOliver-Smith-FaithfulArtin}. Note the implied constant's dependence on ${\langle \psi, \mathbf{1}_G\rangle} $ can be suppressed since $0 \leq {\langle \psi, \mathbf{1}_G\rangle}  \leq \psi(1)$. 
\end{proof}

Second, we record a recent effective lower bound on the residue of an Artin $L$-function. 
\begin{lemma} \label{lem:ResidueLowerBound}
	Let $K/k$ be a Galois extension of number fields with $G = \mathrm{Gal}(K/k)$. Let $\chi$ be any character of $G$ with $\langle \chi, \mathbf{1}_G \rangle= 1$. Let $\nu(\chi)$ be the maximum multiplicity $\langle \chi, \psi \rangle$ over all trivial or quadratic characters $\psi$ of $G$ whose Artin $L$-function $L(s,\psi)$ has a real zero $\beta_{\psi} > 1 - \frac{1}{4 \log |\mathrm{Disc}(K)|}$.   For $\epsilon > 0$,
	\[
	 \Res_{s=1} L(s,\chi)    \gg_{\chi(1),|G|,[k:\Q],\epsilon} |\mathrm{Disc}(K)|^{-\epsilon },
	\]
 where the implied constant is effectively computable if $\epsilon > \frac{\nu(\chi)}{[K:\Q]}$.
\end{lemma}
\begin{proof}
	This follows from the Brauer--Siegel theorem, work of Stark \cite{Stark-EffectiveBrauerSiegel}, and the authors \cite[Theorem~1.1]{CLOZ}. More specifically, if $\nu(\chi) = 0$ or $\epsilon > \frac{\nu(\chi)}{[K:\Q]}$, this follows immediately from \cite[Theorem~1.1]{CLOZ}.  If $\nu(\chi) \geq 1$, then by \cite{Stark-EffectiveBrauerSiegel}, there is a unique character $\psi$ with a real zero $\beta_{\psi} > 1 - \frac{1}{4 \log |\mathrm{Disc}(K)|}$, and $\psi$ is necessarily either trivial or quadratic.  If $\psi$ is trivial, then $\nu(\chi) = 1$ and $\Res_{s=1} L(s,\chi) = L(1,\chi-\psi) \Res_{s=1} \zeta_k(s)$, where we recall that the character $\chi-\psi$ is defined by \eqref{eqn:Artin-Summation}.  The first term may be bounded by \cite{CLOZ} and the latter by the Brauer--Siegel theorem.  If $\psi$ is quadratic, then it is associated with a quadratic extension $F/k$.  We then write $\Res_{s=1} L(s,\chi) = \Res_{s=1}L(s,\chi-\nu(\chi)\psi) \cdot \left(\frac{\Res_{s=1} \zeta_F(s)}{\Res_{s=1}\zeta_k(s)}\right)^{\nu(\chi)}$.  We then invoke \cite{CLOZ} for the first factor, Brauer--Siegel for the numerator of the second, and standard upper bounds on $\Res_{s=1} \zeta_k(s)$ for the denominator.  In either case, we obtain the ineffective bound for all $\epsilon>0$ when $\nu(\chi) \geq 1$.
\end{proof}

\begin{remark}
	If $\nu(\chi) \geq 1$, then it follows from Stark's work that the unique exceptional character $\psi$ can be the trivial character only if $k$ has a quadratic subfield.  In particular, if $k$ does not have a quadratic subfield, then the exceptional character must be quadratic, and this explains the generally better ranges of effectivity given in our main results.  
	In general, because we are pursuing results that apply over arbitrary number fields $k$ (both with and without quadratic subfields), we must allow for the possibility that a quadratic subfield of $k$ has a large discriminant.  It is this possibility that leads to the alternate lower bound when $k$ has a quadratic subfield.  When the quadratic subfields of $k$ are understood, and in particular if they are sufficiently small, then the need to treat this case separately may be eliminated.
\end{remark}

%%%%%%%%%
%%%%%%%%%
%%%%%%%%%
\section{The general approach}
	\label{sec:approach}
%%%%%%%%%
%%%%%%%%%
%%%%%%%%%

	The main goal of this section is to prove the following general theorem that is based on the strategy sketched in Section~\ref{subsec:approach}.

	\begin{theorem}\label{thm:general-approach}
		Let $K/k$ be a Galois extension of number fields with  $G = \mathrm{Gal}(K/k)$. For a rational equivalence class $C \subset G$, assume there exists characters $\Psi_+$ and $\Psi_-$ of $G$ such that:
		
			\begin{enumerate}[(i)]
				\item the difference $\Psi_+ - \Psi_-$ is supported on elements of the rational class $C$; 
				\item the Artin $L$-function $L(s, \Psi_+)$ is entire apart from a simple pole at $s=1$; and
				\item the Artin $L$-function $L(s, \Psi_-)$ is entire.
			\end{enumerate}
		For $A \geq 2$ and $0  < \epsilon < 1$, there exists a sufficiently small constant $b > 0$ and sufficiently large constant $B > 0$ which depend at most on $\Psi_+(1), \Psi_-(1), [k:\Q], A$, and $\epsilon$, such that  if 
		\begin{equation} \label{eqn:general-approach-threshold}
		x \geq  B \Big( \big(\Res_{s=1} L(s,\Psi_+) \big)^{-A} \max\{q(\Psi_+),q(\Psi_-)\}^{1/2} \Big)^{1+\epsilon}, 
		\end{equation}
		then there exists at least $b (\log x)^A$ prime ideals $\mathfrak{p}$ of $k$ whose norm $\N\kp$ is a rational prime not dividing $|\Disc(K)|$ (and hence $\kp$ is unramified in $K$), and $\sigma_\mathfrak{p} = C$.
	\end{theorem}
 	
 	Combined with Lemma~\ref{lem:ResidueLowerBound}, this yields a natural corollary when $A=2$.
 	
 	\begin{corollary} \label{cor:general-approach} 
		Keep the assumptions of Theorem \ref{thm:general-approach} and define $\nu(\Psi_+)$ via \cref{lem:ResidueLowerBound}. Then for any $\epsilon>0$ there exists a prime ideal $\kp$ of $k$ such that $\sigma_{\kp} \in C$, its norm $\N\kp$ is a rational prime not dividing $|\mathrm{Disc}(K)|$, and 
 		\[
			\N\kp \ll_{\Psi_+(1),\Psi_-(1),|G|,[k:\Q],\epsilon} \max\{ q(\Psi_+), q(\Psi_-)\}^{1/2 + \epsilon} , 
 		\] 
 		where the implied constant is effectively computable when $\epsilon > \frac{2\nu(\Psi_+)}{|G|[k:\mathbb{Q}]}$.  
 	\end{corollary}
	\begin{remark}
		 Pollack \cite[Theorem 4.1]{Pollack-PrimeSplittingAbelian} established a similar result when $G$ is abelian. 
	\end{remark}
	
	The construction of characters $\Psi_+$ and $\Psi_-$ satsifying the three hypotheses of Theorem~\ref{thm:general-approach} is essentially one of group theory or, perhaps more precisely, of character theory.  Since the difference of any two characters must take algebraically conjugates values on rationally equivalent conjugacy classes, if $C^\prime$ is a conjugacy class but not a full rational equivalence class, then there can be no characters $\Psi_+$ and $\Psi_-$ as in Theorem~\ref{thm:general-approach} whose difference is supported on $C^\prime$.  This is the fundamental reason why our results are restricted to rational classes.  However, granting this restriction, it is straightforward to find characters $\Psi_+$ and $\Psi_-$ satisfying Condition (i), as we now explain.
	
	For simplicity, suppose $C$ is a rational equivalence class comprised of a single conjugacy class.  This implies that $\chi(C) \in \mathbb{Z}$ for every irreducible character $\chi$.  Considering orthogonality, a natural choice arises:
		\[
			\frac{|G|}{|C|} \mathbf{1}_C = \sum_{\chi} \chi(C) \chi =  \underbrace{ \Big(\sum_{\chi(C) > 0}  \chi(C) \chi \Big)}_{\Psi_+} - \underbrace{\Big(\sum_{\chi(C) <  0}  -\chi(C) \chi \Big)}_{\Psi_-}.
		\]
	That is, $\Psi_+$ is the contribution from those irreducibles (including trivial) for which $\chi(C) > 0$ and $\Psi_-$ is the contribution from those with $\chi(C) < 0$.  Moreover, these are readily seen to be the optimal choices of characters satisfying the first hypothesis, in that any other choices $\Psi_+^\prime$ and $\Psi_-^\prime$ must contain $\Psi_+$ and $\Psi_-$ as summands.  (The difference $\Psi_+^\prime - \Psi_-^\prime$ must be an integral multiple of $\Psi_+ - \Psi_-$.)
	
	Conditions (ii) and (iii) remain and are more substantial.  Over general number fields, the state of the art toward understanding the holomorphy of Artin $L$-functions is essentially represented by Lemma~\ref{lem:Monomial}, and the optimal choices of $\Psi_+$ and $\Psi_-$ need not lie in the monomial cone.\footnote{Strictly speaking, it is possible to bootstrap from the few additional known holomorphy results over $\mathbb{Q}$ or totally real fields to slightly expand the class of $L$-functions for which holomorphy is understood, but such results are sufficiently rare they will not substantially extend the reach of our approach in general.}
	To rectify this, we show in Section~\ref{sec:class-function} that there is an ``offset'' character $\Psi_0$ so that both $\Psi_+ + \Psi_0$ and $\Psi_- + \Psi_0$ do lie in the monomial cone, and so that the order of any pole at $s=1$ does not change.  We are then able to apply Theorem~\ref{thm:general-approach} to these ``shifted'' characters.
	
	To maximize the efficiency of this approach, one therefore wishes to choose the offset $\Psi_0$ so that the conductor $q(\Psi_0)$ is as small as possible.  Finding the optimal offset is an interesting problem in character theory that depends substantially on both $G$ and the rational class $C$.  We expect it is possible to prove that the choice of offset we give in Section~\ref{sec:symmetric-groups} for certain classes $C \subset S_n$ is the optimal choice, but it is unclear to us whether there can be a fully systematic understanding of the optimal choice for given $G,C$.  For this reason, the choice of offset we give in Section~\ref{sec:class-function} is not typically the optimal one, but it is systematic and reasonably efficient.

\subsection{Proof of \cref{thm:general-approach}}
We first establish a basic  estimate for an Euler product that will  sift for squarefree norms without ramified factors. 

\begin{lemma}	\label{lem:FiniteEulerProduct}	
Let $K/k$ be a Galois extension of number fields. Let $\psi : G \to \mathbb{C}$ be a character of $G=\mathrm{Gal}(K/k)$. For any integer $m \geq 1$, the Euler product
\begin{equation} \label{eqn:H}
H_{m}(s,\psi)  := 
 \prod_{\substack{p \,\nmid \,  m}} \Big(1 + \frac{\lambda_{\psi}(p)}{p^s} \Big) \Big( 1 + \sum_{j=1}^{\infty} \frac{\lambda_{\psi}(p^j)}{p^{js}} \Big)^{-1}  
 \times  \prod_{\substack{p \,\mid\, m }}\Big( 1 + \sum_{j=1}^{\infty} \frac{\lambda_{\psi}(p^j)}{p^{js}} \Big)^{-1}  
\end{equation}
converges absolutely and is hence holomorphic in the region $\Re(s) > 1/2$. Moreover, for fixed $\sigma > 1/2$ and arbitrary $t \in \mathbb{R}$, 
\[
e^{-2\psi(1) [k:\Q] \sum_{p \mid m} p^{-\sigma}}     \ll_{\psi(1), [k:\Q], \sigma} |H_{m}(\sigma+it,\psi)| \ll_{\psi(1), [k:\Q],\sigma}   e^{2\psi(1) [k:\Q]  \sum_{p \mid m} p^{-\sigma} }.
\]
\end{lemma}
\begin{proof} The absolute convergence claim will follow from the proof of the upper bound. To establish the upper bound for $|H(s,\psi)|$ with $s = \sigma+it$, we analyze the two products in \eqref{eqn:H}. For any prime $p$ and $\sigma > 1/2$, 
\[
\Big(1 + \frac{\lambda_{\psi}(p)}{p^s} \Big) \Big( 1 + \sum_{j=1}^{\infty} \frac{\lambda_{\psi}(p^j)}{p^{js}} \Big)^{-1} = 1 + O_{\psi(1),[k:\Q],\sigma}(p^{-2\sigma}) 
\]
because \eqref{eqn:Artin-DivisorBound} implies, say, $\lambda_{\psi}(p^j) \leq (\psi(1) [k:\Q])^j$. Also, since there are at most $[k:\Q]$ prime ideals $\kp$ above a rational prime $p$, it follows by \eqref{eqn:Artin-LocalFactorComparison} for any prime $p$ that   
\[
\Big| \Big( 1 + \sum_{j=1}^{\infty} \frac{\lambda_{\psi}(p^j)}{p^{js}} \Big)^{-1}\Big| \leq \prod_{\substack{\kp \mid (p)  }} \big( 1  +  (\mathrm{N}\mathfrak{p})^{-\sigma}\big)^{\psi(1)} \leq  \big(1+p^{-\sigma} \big)^{\psi(1) [k:\Q]}. 
\]
Combining these observations implies that 
\begin{align*}
|H(\sigma+it)| 
& \leq \prod_{\substack{p \,\nmid \,  m}} \Big(1 + O_{\psi(1),[k:\Q]}(p^{-2\sigma}) \Big)  \times \prod_{\substack{ p \,\mid \,  m}} (1+p^{-\sigma})^{\psi(1)[k:\Q]}  \\
& \ll_{\psi(1),[k:\Q],\sigma}   e^{  2 \psi(1)[k:\Q] \sum_{p \mid m} p^{-\sigma}}
\end{align*} 
as required. The last step follows from the identity $\log(1+u) < 2u$ for $u > 0$ and the fact that $\sum_p p^{-2\sigma}$ is convergent for any fixed $\sigma > 1/2$. A similar argument holds for the lower bound with the additional identity $\log(1-u) > -2u$ for $0 < u \leq 1/\sqrt{2}$.  
\end{proof}

Now, we may proceed with the main argument.
Fix $x \geq 10$ and $A \geq 2$. Denote
\begin{equation} \label{eqn:proof-general-parameters}
Q := \max\{ q(\Psi_+), q(\Psi_-) \}, 
\qquad \text{and} \qquad 
d := \max\{ \Psi_+(1), \Psi_-(1) \}.
\end{equation}
Fix $0 < \epsilon < 1$ and $\sigma = 1 - \frac{1}{A} + \frac{\epsilon}{2A} \in (\frac{1}{2},1)$. Define
\begin{equation}
\begin{aligned}
\mathcal{P} = \mathcal{P}(x) &  := \{ p \leq x : p \mid D_K \text{ or there exists $\kp \mid (p)$ such that $\N\kp = p$ and } \mathrm{\sigma}_{\kp} =  C  \}  \\ 
m = m(x) & := \prod_{p \in \mathcal{P}} p.  	
\end{aligned}
\label{eqn:proof-bad-primes}
\end{equation}
Since the number of primes $p$ dividing $D_K$ is $O(\log D_K)$ and $A \geq 2$, it suffices to prove that 
\[
\mathcal{|P|} \gg_{d,[k:\Q],\epsilon,A} (\log x)^A
\]
 provided \eqref{eqn:general-approach-threshold} holds with some constant $B = B_{d,[k:\Q],\epsilon,A} > 0$ sufficiently large. We claim we may assume without loss of generality that
\begin{equation} \label{eqn:FewPrimes}
\sum_{p \mid m} p^{-\sigma} \leq   \frac{\epsilon \log x}{4 A d [k:\Q]}
\end{equation}
because otherwise, by H\"{o}lder's inequality, we would have that 
\begin{align*}
\frac{\epsilon \log x}{4 A d [k:\Q]} \leq \sum_{p \mid m} p^{-\sigma} 
& \leq \Big( \sum_{p \mid m} 1 \Big)^{1-\frac{\sigma}{1+\epsilon/2A}} \Big( \sum_{p \mid m} p^{-1-\epsilon/2A} \Big)^{\frac{\sigma}{1+\epsilon/2A}}  \\
& \ll_{\epsilon,A} |\mathcal{P}|^{1-\frac{\sigma}{1+\epsilon/2A}}  = |\mathcal{P}|^{1/(A+\epsilon/2)} \leq |\mathcal{P}|^{1/A}.
\end{align*}
This would imply $|\mathcal{P}| \gg_{d,[k:\Q],\epsilon,A} (\log x)^{A}$, proving the claim.  Therefore, to complete the proof, we will seek a contradiction to \eqref{eqn:general-approach-threshold} assuming \eqref{eqn:FewPrimes} holds. 

Let $\psi : G \to \mathbb{C}$ be an Artin character such that $L(s,\psi) = L(s,\psi,K/k)$ is holomorphic except possibly for a simple pole at $s=1$. In other words, $(s-1) L(s,\psi)$ is entire and  
\[
\kappa(\psi) := \mathop{\mathrm{Res}}_{s=1} L(s,\psi)
\]
is well-defined (and possibly equal to zero).    Fix a non-negative $C^{\infty}$ function $\varphi$ that is compactly supported in the open set $(0,1) \subseteq \R$ and satisfies $\int_0^1 \varphi(t) dt = 1$. Thus, its Mellin transform $\widehat{\varphi}(s) := \int_0^{\infty} \varphi(t) t^{s-1} dt$ is entire, and satisfies $\widehat{\varphi}(1) = 1$ and $|\widehat{\varphi}(s)| \ll_{\varphi,j} |s|^{-j}$ for any integer $j \geq 0$ and $\Re(s) > 1/2$. Define
\[
S_m(x,\psi) := \sum_{\substack{n \leq x \\ (n,m)=1}} \lambda_{\psi}(n) \mu^2(n)   \varphi(n/x), 
\]
where $\mu^2(\,\cdot\,)$ is the indicator function on squarefree integers. By Mellin inversion and \eqref{eqn:Artin-RationalEulerProduct}, 
\[
S_m(x,\psi)  = \frac{1}{2\pi i} \int_{(2)} L(s,\psi) H_m(s,\psi) \widehat{\varphi}(s) x^s ds, 
\]
where $H_m(s,\psi)$ is given by Lemma~\ref{lem:FiniteEulerProduct}. 
Since $(s-1)L(s,\psi)$ is entire by assumption, we may shift the above contour to $\Re(s) = \sigma \in (1/2,1)$ with a possible simple pole at $s=1$. By Lemmas \ref{lem:Convexity} and \ref{lem:FiniteEulerProduct} and \eqref{eqn:FewPrimes}, this yields 
\begin{equation} \label{eqn:general-approach-contourshift}
\begin{aligned} 
\Big|S_m(x,\psi)  -  \kappa(\psi)  H_m(1,\psi)  x \Big| 
& \leq \frac{1}{2\pi} \int_{-\infty}^{\infty} |L(\sigma+it,\psi) H_m(\sigma+it,\psi) \widehat{\varphi}(\sigma)| x^{\sigma}   dt \\
& \ll_{\psi(1),[k:\Q],\epsilon,A}   x^{\sigma+ \frac{\epsilon}{2A}}   \int_{-\infty}^{\infty} \frac{ ( q(\psi)   |1+it|^{\psi(1)[k:\Q]} )^{\frac{1-\sigma}{2}+\frac{\epsilon}{2A}}  }{|\sigma+it|^{\psi(1) [k:\Q]+10}}  dt \\
&  \ll_{\psi(1),[k:\Q],\epsilon,A}   x^{\sigma+ \frac{\epsilon}{2A}} q(\psi)^{\frac{1-\sigma}{2} + \frac{\epsilon}{2A}}
\end{aligned}
\end{equation}
by applying the bound $|\widehat{\varphi}(s)| \ll_{\varphi,j} |s|^{-j}$ for $\Re(s) > 1/2$ when $j = \psi(1) [k:\Q]+10$, say.

By assumptions (ii) and (iii) of Theorem \ref{thm:general-approach}, we may apply the estimate \eqref{eqn:general-approach-contourshift} twice: once for $\psi = \Psi_-$ and once for $\psi =\Psi_+$. As $\kappa(\Psi_-) = 0$ and $\kappa(\Psi_+) \neq 0$ by assumption, we find upon recalling $Q$ and $d$ from  \eqref{eqn:proof-general-parameters} that
\begin{align*}
|S_m(x,\Psi_+)| & \geq  |\kappa(\Psi_+) H(1,\Psi_+)|  x + O_{d,[k:\Q],\epsilon,A}\big(x^{\sigma+ \frac{\epsilon}{2A}} Q^{\frac{1-\sigma}{2} + \frac{\epsilon}{2A}} \big) , \text{ and } \\
|S_m(x,\Psi_-)| & =  O_{d,[k:\Q],\epsilon,A}\big(x^{\sigma+ \frac{\epsilon}{2A}} Q^{\frac{1-\sigma}{2} + \frac{\epsilon}{2A}} \big).
\end{align*}

On the other hand, for every prime $p \nmid m$ and $p \leq x$, we have by \eqref{eqn:proof-bad-primes} that $p \nmid D_K$ and, for every prime ideal $\kp$ above $p$ with $\N\kp =p$, the Frobenius element $\sigma_{\kp}$ satisfies $\sigma_{\kp} \subseteq G \setminus C$ in which case $\Psi_+(\sigma_{\kp}) = \Psi_-(\sigma_{\kp})$ by assumption (i) of Theorem \ref{thm:general-approach}.

 Thus, by \eqref{eqn:Artin-PrimeCoefficient}, we have
\[
\lambda_{\Psi_+}(p) = \sum_{\N\kp = p} \Psi_+(\sigma_{\kp}) = \sum_{\N\kp = p} \Psi_-(\sigma_{\kp}) = \lambda_{\Psi_-}(p)
\]
for all primes $p \leq x$ with $p \nmid m$. By multiplicativity, $\lambda_{\Psi_+}(n) \mu^2(n) = \lambda_{\Psi_-}(n) \mu^2(n)$ for every integer $n \leq x$ with $(n,m) = 1$. This implies 
\[
S_m(x,\Psi_+) = S_m(x,\Psi_-), 
\]
so $|S_m(x,\Psi_+)| = |S_m(x,\Psi_-)|$ in particular. 

Combining this key observation with the prior inequalities yields
\[ 
|\kappa(\Psi_+) H(1,\Psi_+)|  x \ll_{d,[k:\Q],\epsilon,A}  Q^{\frac{1-\sigma}{2} + \frac{\epsilon}{2A}} x^{\sigma+ \frac{\epsilon}{2A}}. 
\]
Using $\sigma = 1 - \frac{1}{A}+\frac{\epsilon}{2A}$ and the lower bound on $H(1,\Psi_+)$ from Lemma~\ref{lem:FiniteEulerProduct},we conclude that
\[
x \ll_{d,[k:\Q],\epsilon,A}  \Big(  |\kappa(\Psi_+)|^{-1} Q^{\frac{1}{2A}}\Big)^{\frac{A}{1-\epsilon}}. 
\]
Rescaling $\epsilon$ appropriately and choosing $B = B_{d,[k:\Q],A,\epsilon} \geq 1$ sufficiently large yields the desired contradiction to \eqref{eqn:general-approach-threshold}. This completes the proof of Theorem \ref{thm:general-approach}.
\hfill \qed

%%%%%%%%%
%%%%%%%%%
%%%%%%%%%	
\section{The fundamental class function}
	\label{sec:class-function}
%%%%%%%%%
%%%%%%%%%
%%%%%%%%%

	\allowdisplaybreaks

	We now turn to the problem of constructing the characters $\Psi_+$ and $\Psi_-$ of $G$ that satisfy the conditions of Theorem~\ref{thm:general-approach}.  The starting point is the following simple lemma.
	
	\begin{lemma} \label{lem:class-function-construction}
		Let $G$ be a finite group, let $g \in G$, and let $\langle g \rangle \leq G$ be the cyclic subgroup generated by $g$.  Let $\xi \colon \langle g \rangle \to \mathbb{C}^\times$ be the unique irreducible character of $\langle g \rangle$ such that $\xi(g) = \exp\left(\frac{2\pi i}{n}\right)$, where $n := |g|$.  Define a class function $\phi_g$ of $\langle g\rangle$ by means of the expression
			\begin{equation} \label{eqn:phi-def}
				\phi_g := \prod_{p \mid n} \left(1 - \xi^{\frac{n}{p}}\right),
			\end{equation}
		and let $\Delta_g := \mathrm{Ind}_{\langle g\rangle}^G \phi_g$.
		
		Then, as a class function of $G$, $\Delta_g$ is a non-zero function supported on the rational equivalence class of $g$.  Moreover, $\langle \Delta_g, \chi \rangle$ is an integer for every irreducible character $\chi$ of $G$.
	\end{lemma}
	\begin{proof}
		We first observe that $\phi_g(g^i) = 0$ unless $\mathrm{gcd}(i,n) = 1$, for if not, there is some prime $p \mid \mathrm{gcd}(i,n)$, and we would have $\xi^{n/p}(g^i) = 1$ and hence that $1-\xi^{n/p}(g^i)=0$.  As a result, $\phi_g$ is supported on the generators of the cyclic group $\langle g \rangle$, and hence its induction $\Delta_g$ is supported on the conjugates of these generators, which together comprise exactly the rational equivalence class of $g$.  To see that $\Delta_g$ is non-zero, we observe that $\langle \phi_g, \mathbf{1}_{\langle g\rangle} \rangle_{\langle g\rangle} = 1$, so that, by Frobenius reciprocity, we find that $\langle \Delta_g, \mathbf{1}_G\rangle_G = 1$ as well.  This implies that $\Delta_g \ne 0$.  The final claim follows on observing that $\phi_g$ is a difference of characters, so $\Delta_g$ must be too.
	\end{proof}
	
	While we shall not need it immediately, we nevertheless consider it illuminating to note at this stage the following lemma.  
	\begin{lemma} \label{lem:inner-product-when-conjugacy}
		Let $G$, $g$, and $\Delta_g$ be as in Lemma~\ref{lem:class-function-construction}.  Then for any irreducible character $\chi$ of $G$ that is constant on the rational class of $g$, we have
			\[
				\langle \Delta_g, \chi\rangle
					= \chi(g).
			\]
		In particular, if the rational class of $g$ comprises a single conjugacy class, then $\Delta_g$ is $|C_G(g)|$ times the indicator function of the conjugacy class of $g$, where $C_G(g)$ denotes the centralizer of $g$ in $G$.
	\end{lemma}
	\begin{proof}
		By Frobenius reciprocity, we have
			\[
				\langle \Delta_g, \chi\rangle
					= \frac{1}{n} \sum_{j=1}^n \phi_g(g^j) \bar\chi(g^j).
			\]
		The function $\phi_g$ is supported on those powers $g^j$ with $j \leq n$ coprime to $n$, i.e. those $g^j$ that are rationally equivalent to $g$.  However, by our assumption on $\chi$, we must have that $\bar\chi(g^j) = \bar\chi(g) = \chi(g^{-1}) = \chi(g)$ for every such $g$.  Hence, we find
			\[
				\langle \Delta_g, \chi\rangle
					= \frac{\chi(g)}{n} \sum_{j=1}^n \phi_g(g^j) 
					= \chi(g),
			\]
		where the last equality follows since $\langle \phi_g, \mathbf{1}_{\langle g \rangle} \rangle_{\langle g \rangle} = 1$.  The second claim follows from the orthogonality relations for characters.
	\end{proof}
	
	Lemma~\ref{lem:class-function-construction} leads to a first unconditional construction of $\Psi_+$ and $\Psi_-$.
	
	\begin{proposition} \label{prop:first-construction}
		Fix notation as in Lemma~\ref{lem:class-function-construction}.  Let $\Psi_+$ be the induction to $G$ of the sum of the terms appearing in the expansion of the product \eqref{eqn:phi-def} with a positive sign, and let $\Psi_-$ be the induction to $G$ of the sum of the terms appearing in the product \eqref{eqn:phi-def} with a negative sign.  More explicitly, take
			\[
				\Psi_+ := \mathrm{Ind}_{\langle g\rangle}^G \left[ \sum_{d \mid n} \frac{\mu(d)^2 + \mu(d)}{2} \xi^{\sum_{p \mid d} \frac{n}{p}} \right]
			\]
		and
			\[
				\Psi_- := \mathrm{Ind}_{\langle g\rangle}^G \left[ \sum_{d \mid n} \frac{\mu(d)^2 - \mu(d)}{2} \xi^{\sum_{p \mid d} \frac{n}{p}} \right].
			\]
		Then $\Psi_+$ and $\Psi_-$ are characters of $G$ satisfying the conditions of Theorem~\ref{thm:general-approach}.
	\end{proposition}
	\begin{proof}
		By construction, both $\Psi_+$ and $\Psi_-$ are non-negative integral linear combinations of monomial characters, hence are characters, and are such that the associated Artin $L$-functions are entire except possibly for a pole at $s=1$.  The presence and order of a pole at $s=1$ is detected by the multiplicity of the trivial representation inside $\Psi_+$ and $\Psi_-$, and we observe that it has multiplicity $1$ inside $\Psi_+$ and multiplicity $0$ inside $\Psi_-$.  This shows that $\Psi_+$ and $\Psi_-$ satisfy conditions (ii) and (iii) of Theorem~\ref{thm:general-approach}.  Moreover, by construction, we have that $\Psi_+ - \Psi_- = \Delta_g$, so they also satisfy condition (i) by Lemma~\ref{lem:class-function-construction}.
	\end{proof}

	Now, despite the fact that Theorem~\ref{thm:general-approach} is stated with respect to the characters $\Psi_+$ and $\Psi_-$, we actually consider the difference $\Psi_+ - \Psi_- = \Delta_g$ to be the more fundamental object of consideration.  
	This is made clearer by two simple yet useful lemmas that may be used to provide improvements over Proposition~\ref{prop:first-construction}.

	\begin{lemma} \label{lem:class-functions-basepoint}
		Let $G$ and $\Delta_g$ be as in Lemma~\ref{lem:class-function-construction}, and let $\Psi_0$ be any character of $G$ so that both $\Psi_0$ and $\Psi_0 + \Delta_g$ are expressible as a non-negative rational linear combination of monomial characters, and so that the trivial character has multiplicity $0$ in $\Psi_0$.  Then the choices $\Psi_- = \Psi_0$ and $\Psi_+ = \Psi_0 + \Delta_g$ satisfy the conditions of Theorem~\ref{thm:general-approach}.
	\end{lemma}
	\begin{proof}
		By construction, condition (i) is automatic from Lemma~\ref{lem:class-function-construction}.  The assumptions on $\Psi_0$ ensure that conditions (ii) and (iii) are satisfied as well.
	\end{proof}
	
	\begin{lemma} \label{lem:class-functions-artin}
		Let $G$ and $\Delta_g$ be as in Lemma~\ref{lem:class-function-construction}, and suppose that the Artin holomorphy conjecture holds for all Galois $G$-extensions of number fields.  Then the characters
			\[
				\Psi_+ = \sum_{\chi \in \mathrm{Irr}(G)} \max\left\{ \langle \chi, \Delta_g \rangle, 0 \right\} \cdot \chi
			\]
		and
			\[
				\Psi_- = - \sum_{\chi \in \mathrm{Irr}(G)} \min\left\{ \langle \chi, \Delta_g \rangle, 0 \right\} \cdot \chi
			\]
		satisfy the conditions of Theorem~\ref{thm:general-approach}.
	\end{lemma}
	\begin{proof}
		We observe that
			\[
				\Psi_+ - \Psi_-
					= \sum_{\chi \in \mathrm{Irr}(G)} \langle \chi, \Delta_g \rangle \cdot \chi
					= \Delta_g,
			\]
		since the irreducible characters form an orthonormal basis for the space of class functions.  As a result, condition (i) holds as a consequence of Lemma~\ref{lem:class-function-construction}.  For conditions (ii) and (iii), we observe from the proof of Lemma~\ref{lem:class-function-construction} that $\langle \Delta_g, \mathbf{1} \rangle_G = 1$, which implies by the explicit construction of $\Psi_+$ and $\Psi_-$ above that $\langle \Psi_+, \mathbf{1} \rangle_G = 1$ and $\langle \Psi_-, \mathbf{1} \rangle_G = 0$.  Hence, conditions (ii) and (iii) are satisfied by our assumption of the Artin holomorphy conjecture.
	\end{proof}

	We will implicitly exploit Lemma~\ref{lem:class-functions-basepoint} in our study of the symmetric groups and the proof of Theorems~\ref{thm:Sn-cycle-intro} and \ref{thm:Sn-asymptotic-simplified}, but for the remainder of this section, our goal is to evaluate the degrees and conductors of Artin $L$-functions associated with the choices of $\Psi_+$ and $\Psi_-$ produced by Proposition~\ref{prop:first-construction} and Lemma~\ref{lem:class-functions-artin}.
	Therefore, throughout the remainder of this section, we fix a Galois extension of number fields $K/k$ with $\mathrm{Gal}(K/k) \simeq G$ for some finite group $G$.  We analyze the characters arising in Proposition~\ref{prop:first-construction} first, as our bounds will be mechanically much simpler.
	
\subsection{The unconditional characters from Proposition~\ref{prop:first-construction}, and the proof of Theorem~\ref{thm:least-prime-rational-class-intro}}
	We begin by observing that if $g = 1$, then the unconditional choice provided by Proposition~\ref{prop:first-construction} is optimal (albeit nearly trivial).
	
	\begin{lemma} \label{lem:identity-class}
		Let $G$ be a finite group, let $g=1$ (the identity of $G$), let $\Delta_g$ be as in Lemma~\ref{lem:class-function-construction}, and let $\Psi_+$ and $\Psi_-$ be as in Proposition~\ref{prop:first-construction}.
		
		Then $\Psi_+$ is the regular representation of $G$ and $\Psi_-$ is the zero character.  Consequently, the degree of $\Psi_+$ is $|G|$ and its Artin conductor is $\mathfrak{D}_{K/k}$, the relative discriminant of $K/k$.  The degree of $\Psi_-$ is $0$ and its Artin conductor is the trivial ideal $\mathcal{O}_k$ generated by $1$. 
	\end{lemma}
	\begin{proof}
		In this case, we observe that $\langle g \rangle$ is the trivial subgroup, and that $\phi_g$ is the trivial character of the trivial group.  It follows that $\Psi_+$ is the regular representation of $G$ and that $\Psi_-$ is zero.  The rest is standard.
	\end{proof}

	For a general character $\chi$ of $G$, not necessarily irreducible, the Artin conductor $\mathfrak{f}_\chi$ is an ideal of $k$ defined locally by
		\begin{equation} \label{eqn:artin-conductor-def}
			\mathfrak{f}_\chi 
				= \prod_{\mathfrak{p}} \mathfrak{p}^{v_\mathfrak{p}(\mathfrak{f}_\chi)},
		\end{equation}
	where
		\begin{equation} \label{eqn:local-artin-conductor}
			v_\mathfrak{p}(\mathfrak{f}_\chi) = \sum_{i \geq 0} \frac{|G_i|}{|G_0|} \left(\chi(1) - \frac{1}{|G_i|} \sum_{\sigma \in G_i} \chi(\sigma)\right),
		\end{equation}
	with $G_i$ for $i \geq 0$ being the (lower) ramification groups associated with a prime of $K$ lying over $\mathfrak{p}$ \cite[VI \S 2 Corollary 1]{Serre-LocalFields}.  In particular, the sum in \eqref{eqn:local-artin-conductor} is supported on those $i$ for which $G_i \ne 1$, and the product in \eqref{eqn:artin-conductor-def} is supported on those primes $\mathfrak{p}$ that are ramified in $K/k$.  As an important special case, recall that $\mathfrak{f}_\mathrm{reg} = \mathfrak{D}_{K/k}$, where $\mathrm{reg}$ is the character of the regular representation of $G$.  Using \eqref{eqn:local-artin-conductor}, we find
			\begin{equation} \label{eqn:disc-val}
				v_\mathfrak{p}(\mathfrak{D}_{K/k}) 
					= v_\mathfrak{p}(\mathfrak{f}_\mathrm{reg})
					= |G| \sum_{i \geq 0} \frac{|G_i|-1}{|G_0|}.
			\end{equation}	
	We first give an easy bound on the Artin conductor in terms of the degree.  A similar bound has appeared recently in \cite[Lemma 4.2]{FiorilliJouve}, but we provide a complete proof because it is short and previews useful notions.

	\begin{lemma} \label{lem:conductor-vs-degree}
		Let $\chi$ be a character of the finite group $G$.  Then
			\[
				v_\mathfrak{p}(\mathfrak{f}_\chi)
					\leq \frac{ 2 \chi(1)}{|G|} v_\mathfrak{p}(\mathfrak{D}_{K/k}).
			\]
	\end{lemma}
	\begin{proof}
		Using the trivial inequality $|\chi(\sigma)| \leq \chi(1)$, it follows from \eqref{eqn:local-artin-conductor} that
			\begin{align*}
				v_\mathfrak{p}(\mathfrak{f}_\chi)
					&= \frac{1}{|G_0|}\sum_{i \geq 0} \left( (|G_i|-1)\chi(1) - \sum_{\substack{ \sigma \in G_i \\ \sigma \ne 1}} \chi(\sigma)\right) \\
					&\leq \frac{1}{|G_0|}\sum_{i \geq 0} 2\chi(1)(|G_i|-1) \\
					&= \frac{2 \chi(1)}{|G|} v_\mathfrak{p}(\mathfrak{D}_{K/k}),
			\end{align*}
		as claimed.
	\end{proof}
	
	With this, we can give an easy bound on the conductors of the characters $\Psi_+$ and $\Psi_-$ provided by Proposition~\ref{prop:first-construction}.
	
	\begin{lemma} \label{lem:unconditional-conductors}
		Let $G$ be a finite group, let $g \in G$, and let $\Psi_+$ and $\Psi_-$ be the characters of $G$ constructed in Proposition~\ref{prop:first-construction}.  Then we have 
			\[
				\N\mathfrak{f}_{\Psi_+}, \N\mathfrak{f}_{\Psi_-}
					\leq (\N\mathfrak{D}_{K/k})^{\frac{2^{\omega(n)}}{n}}
			\]
		where $n = |g|$, and hence also
			\[
				q(\Psi_+),q(\Psi_-)
					\leq |\mathrm{Disc}(K)|^{\frac{2^{\omega(n)}}{n}}.
			\]
	\end{lemma}
	\begin{proof}
		If $n=1$, this follows from Lemma~\ref{lem:identity-class}.  If $n \geq 2$, then both $\Psi_+$ and $\Psi_-$ are the sum of $2^{\omega(n)-1}$ characters induced from one-dimensional characters of the subgroup $H = \langle g \rangle$.  The result thus follows from Lemma~\ref{lem:conductor-vs-degree}.
	\end{proof}
	
	With this, the proof of Theorem~\ref{thm:least-prime-rational-class-intro} is now straightforward.
	
	\begin{proof} [Proof of Theorem~\ref{thm:least-prime-rational-class-intro}]
		The claimed Linnik exponent follows immediately from Corollary~\ref{cor:general-approach} when applied to the characters $\Psi_+$ and $\Psi_-$ produced by Proposition~\ref{prop:first-construction} on using Lemma~\ref{lem:unconditional-conductors} to bound their conductors.
		
		It remains to consider the range of effectivity in $\epsilon$.  Let $\psi_{K/k}$ denote the exceptional character of $K/k$ if it exists.  Supposing first that $\psi_{K/k}(g) = 1$, by Frobenius reciprocity, we find $\langle \Psi_+, \psi_{K/k} \rangle_G = \langle \Psi_+|_{\langle g \rangle}, \mathbf{1}_{\langle g \rangle} \rangle_{\langle g \rangle} = 1$.  On the other hand, if $\psi_{K/k}(g) = -1$, then $\psi_{K/k}|_{\langle g \rangle} = \xi^{n/2}$, where $\xi$ is as in Lemma~\ref{lem:class-function-construction}, and hence $\langle \Psi_+, \psi_{K/k} \rangle_G = 0$.  This completes the proof on appealing to Corollary~\ref{cor:general-approach}.
	\end{proof} 
	
\subsection{The conditional characters from Lemma~\ref{lem:class-functions-artin}, and the proof of Theorem~\ref{thm:least-prime-rational-class-artin-intro}}	
	
	We now turn to analyzing the conductors of the conditional characters from Lemma~\ref{lem:class-functions-artin}.  Let $g \in G$ be a non-identity element.  We let $N_G(\langle g \rangle)$ be the normalizer in $G$ of the cyclic subgroup generated by $g$, and we let $C_G(g)$ denote its centralizer.  So doing, by means of the conjugation action on $\langle g \rangle$ there is a natural sequence of maps
		\[
			N_G(\langle g \rangle) 
				\to N_G(\langle g \rangle) / C_G(g)
				\to \mathrm{Aut}(\langle g \rangle)
				\cong (\mathbb{Z}/n\mathbb{Z})^\times,
		\]
	where $n = |g|$.  We exploit this in the following key lemma.
	
	\begin{lemma} \label{lem:delta-l2}
		With notation as above, assume that $g \neq 1$ has order $n$.  Let $A \leq (\mathbb{Z}/n\mathbb{Z})^\times$ be the image of $N_G(\langle g\rangle)$ in $\mathrm{Aut}(\langle g \rangle)$ via conjugation.  Then
			\[
				\|\Delta_g\|_2^2
					:= \sum_{\chi \in \mathrm{Irr}(G)} |\langle \Delta_g, \chi \rangle|^2
					= \frac{|C_G(g)|}{n} \sum_{d \mid n} \mu(d)^2 \#\{a \in A : a \equiv 1 \pmod{d}\},
			\]
		and in particular, $\|\Delta_g\|_2 \leq |C_G(g)|^{1/2}$.
	\end{lemma}
	\begin{proof}
		By Frobenius reciprocity, we find that
			\[
				\langle \Delta_g, \chi \rangle
					= \langle \phi_g, \mathrm{Res}_{\langle g \rangle} ^G \chi \rangle_{\langle g \rangle}
					= \frac{1}{n} \sum_{j = 1}^n \phi_g(g^j) \bar\chi(g^j),
			\]
		hence
			\begin{align*}
				\sum_{\chi \in \mathrm{Irr}(G)} |\langle \Delta_g, \chi\rangle|^2
					&= \frac{1}{n^2} \sum_{j,k \leq n} \phi_g(g^j)\overline{\phi_g(g^k)} \sum_{\chi \in \mathrm{Irr}(G)} \bar\chi(g^j)\chi(g^k) \\
					&= \frac{|C_G(g)|}{n^2} \sum_{\substack{ j,k \leq n \\ g^j \sim g^k}} \phi_g(g^j) \overline{\phi_g(g^k)},
			\end{align*}
		where the sum runs over those $j$ and $k$ so that $g^j$ and $g^k$ are conjugate.  
		(Here, we have used the orthogonality relations and the well known equality
			\[
				\sum_{\chi \in \mathrm{Irr}(G)} |\chi(\sigma)|^2 = |C_G(\sigma)|.
			\]
		See \cite[Chapter 2 Proposition 7]{Serre-LinearRepresentations}, for example.)
		Letting $A$ be the image of $N_G(\langle g\rangle)$ in $\mathrm{Aut}(\langle g \rangle) \cong (\mathbb{Z}/n\mathbb{Z})^\times$ as in the statement of the lemma, it follows that $g^j \sim g^k$ if and only if there is some $a \in A$ so that $k = aj$.  Thus, we have 
			\begin{align} \label{eqn:exponential-sum}
				\sum_{\chi \in \mathrm{Irr}(G)} |\langle \Delta_g, \chi\rangle|^2 \notag
					&= \frac{|C_G(g)|}{n^2} \sum_{a \in A} \sum_{j \leq n} \phi_g(g^j) \overline{\phi_g(g^{aj})} \\
					&= \frac{|C_G(g)|}{n^2} \sum_{a \in A} \sum_{d,e \mid n} \mu(d)\mu(e)\sum_{j \leq n} \exp\left(2\pi i \left( \sum_{p \mid d} \frac{j}{p} - \sum_{p \mid e} \frac{aj}{p} \right)\right).
			\end{align}
		We observe that since $a \in A \leq (\mathbb{Z}/n\mathbb{Z})^\times$, the argument of the exponential on the right-hand side of \eqref{eqn:exponential-sum} will have a denominator unless $d=e$ and $a \equiv 1 \pmod{d}$.  The sum over $j$ will therefore cancel completely unless these conditions are satisfied.  We therefore conclude that
			\[
				\sum_{\chi \in \mathrm{Irr}(G)} |\langle \Delta_g, \chi\rangle|^2 
					= \frac{|C_G(g)|}{n} \sum_{d \mid n} \mu(d)^2 \#\{a \in A : a \equiv 1 \pmod{d}\}.
			\]
		This is exactly the first claim of the lemma.  The second follows simply on evaluating the summation above when $A = (\mathbb{Z}/n\mathbb{Z})^\times$.
	\end{proof}
	
	With this, we now give a bound on the conductors of $\Psi_+$ and $\Psi_-$.
	
	\begin{lemma} \label{lem:conditional-conductor}
		With notation as above, assume that $g \ne 1$, and let $\Psi_+$ and $\Psi_-$ be as in Lemma~\ref{lem:class-functions-artin}.  Then $\Psi_+(1) = \Psi_-(1) \leq \frac{1}{2} |C_G(g)|^{1/2} |G|^{1/2}$ and
			\[
				q(\Psi_+),q(\Psi_-)
					\leq |\mathrm{Disc}(K)|^{\frac{|C_G(g)|^{1/2}}{|G|^{1/2}}}.
			\]
	\end{lemma}
	\begin{proof}
		Since $g \ne 1$, we have $\Psi_-(1) = \Psi_+(1)$, and hence
			\[
				\Psi_+(1)
					= \frac{\Psi_+(1) + \Psi_-(1)}{2} 
					= \frac{1}{2} \sum_{\chi \in \mathrm{Irr}(G)} |\langle \Delta_g,\chi\rangle| \chi(1)
					\leq \frac{1}{2} |C_G(g)|^{1/2} |G|^{1/2}
			\]
		by Cauchy--Schwarz and Lemma~\ref{lem:delta-l2}.  The same inequality also obviously holds for $\Psi_-(1)$.  This also yields the bound on $q(\Psi_+)$ and $q(\Psi_-)$ on appealing to Lemma~\ref{lem:conductor-vs-degree}.  
	\end{proof}
	
	With this lemma in hand, the proof of Theorem~\ref{thm:least-prime-rational-class-artin-intro} is immediate.
	
	\begin{proof} [Proof of Theorem~\ref{thm:least-prime-rational-class-artin-intro}]
		This follows from Corollary~\ref{cor:general-approach}, Lemma~\ref{lem:inner-product-when-conjugacy}, Lemma~\ref{lem:class-functions-artin}, and Lemma~\ref{lem:conditional-conductor}.
	\end{proof}
	
\subsection{Less simplified bounds} \label{subsec:less-simplified}

	The bound on the conductors given in Lemma~\ref{lem:conditional-conductor} can be shown to be essentially optimal in certain cases (notably if $g$ has order $2$ and is contained in the center of $G$), but need not be in general.  For specific reasonably small groups $G$, it will be better to compute the conductors exactly; we do so in Section~\ref{sec:small-symmetric} for some small symmetric groups.  For general groups, this may not be feasible, and our aim in this section is to provide less simplified bounds that may still be applicable to general groups $G$.  The key idea is hinted at by the proof of Lemma~\ref{lem:conditional-conductor}, and it is that we may instead equivalently bound the conductors $q(\Psi_+ + \Psi_-)$ and $q(\Psi_+ - \Psi_-) = q(\Delta_g)$, the latter expression making sense by linearity of the local Artin conductor \eqref{eqn:local-artin-conductor}.  In fact, we will see by analyzing $q(\Delta_g)$ that, if $g \ne 1$, then we necessarily have $q(\Psi_+)$ dividing (and hence bounded by) $q(\Psi_-)$.  We begin with this argument, though it is less crucial to our bounds overall.
	
	\begin{lemma} \label{lem:local-artin-conductor-delta}
		Let $G$ be a finite group, let $g \in G$ be a non-identity element, and let $\Delta_g$ be as in Lemma~\ref{lem:class-function-construction}.  Then for any prime $\mathfrak{p}$, there holds
			\[
				v_\mathfrak{p}(\mathfrak{f}_{\Delta_g})
					= -\frac{ |C_G(g)| }{r \cdot |G_0|} \sum_{i \geq 0} \#\{\sigma \in G_i : \sigma \text{ is rationally equivalent to $g$}\},
			\]
		where $C_G(g)$ is the centralizer of $g$ and $r = \frac{\phi(n)}{[N_G(\langle g\rangle) : C_G(g)]}$ is the number of distinct conjugacy classes comprising the rational class of $g$; here, $n=|g|$ and $N_G(\langle g \rangle)$ is the normalizer of the subgroup $\langle g \rangle$.  
		
		In particular, $v_\mathfrak{p}(\mathfrak{f}_{\Delta_g}) \leq 0$ and $v_\mathfrak{p}(\mathfrak{f}_{\Delta_g}) \geq - \frac{|C_G(g)|}{r |G|} v_\mathfrak{p}( \mathfrak{D}_{K/k})$ for every prime $\mathfrak{p}$.
	\end{lemma}
	\begin{proof}
		Since $\Delta_g$ is supported on the rational class of the non-identity element $g$, we find that $\Delta_g(1) = 0$.  Consequently, \eqref{eqn:local-artin-conductor} implies that
			\[
				v_\mathfrak{p}(\mathfrak{f}_{\Delta_g})
					= -\frac{1}{|G_0|} \sum_{i \geq 0} \sum_{\sigma \in G_i} \Delta_g(\sigma).
			\]
		We begin by observing that if $\sigma$ is in the rational class of $g$, then so are $\sigma^j$ for each $1 \leq j \leq n$ coprime to $n$, and we have that
			\begin{equation} \label{eqn:swap-to-g}
				\sum_{\substack{ j \leq n \\ \mathrm{gcd}(j,n) = 1}} \Delta_g(\sigma^j)
					= \sum_{\substack{ j \leq n \\ \mathrm{gcd}(j,n) = 1}} \Delta_g(g^j).
			\end{equation}
		Letting $\phi_g$ be the class function on $\langle g \rangle$ from Lemma~\ref{lem:class-function-construction}, we then find that
			\begin{align*}
				\sum_{\substack{ j \leq n \\ \mathrm{gcd}(j,n) = 1}} \Delta_g(g^j)
					& = \sum_{\substack{ j \leq n \\ \mathrm{gcd}(j,n) = 1}} \left(\mathrm{Ind}_{\langle g \rangle}^G \phi_g\right)(g^j) \\
					& = \frac{1}{n} \sum_{\substack{ j \leq n \\ \mathrm{gcd}(j,n) = 1}} \sum_{\sigma \in N_G(\langle g \rangle)} \phi_g(\sigma^{-1}g^j\sigma) \\
					& = \frac{1}{n} \sum_{\sigma \in N_G(\langle g \rangle)} \sum_{\substack{ j \leq n \\ \mathrm{gcd}(j,n) = 1}} \phi_g( \sigma^{-1} g^j \sigma) \\
					& = \sum_{\sigma \in N_G(\langle g \rangle)} \left( \frac{1}{n} \sum_{\substack{ j^\prime \leq n \\ \mathrm{gcd}(j^\prime,n) = 1}} \phi_g( g^{j^\prime} )\right) \\
					& = \sum_{\sigma \in N_G(\langle g \rangle)} 1 \\
					& = |N_G(\langle g \rangle)|,
			\end{align*}
		where the second-to-last equality follows from the fact that $\langle \phi_g, \mathbf{1}_{\langle g \rangle}\rangle_{\langle g \rangle} = 1$.  
		
		Because $\phi(n)$ elements of $G_i$ are considered at a time in \eqref{eqn:swap-to-g}, as a result, we find that 
			\[
				\sum_{\sigma \in G_i} \Delta_g(\sigma)
					= \frac{|N_G(\langle g \rangle)|}{\phi(n)} \#\{ \sigma \in G_i : \sigma \text{ is rationally equivalent to $g$}\}.
			\]
		Lastly, the number of conjugacy classes comprising the rational equivalence class of $g$ is $r = \frac{\phi(n)}{[N_G(\langle g \rangle) : C_G(g)]}$, which implies the first claim on simplifying.  The second follows on noting that for every $G_i$, the number of elements $\sigma \in G_i$ rationally equivalent to $g$ is at most $|G_i|-1$ since the identity is not equivalent to $g$.
	\end{proof}

	Since $v_\mathfrak{p}(\mathfrak{f}_{\Delta_g}) \leq 0$ for every prime $\mathfrak{p}$, and we will always choose characters $\Psi_+$ and $\Psi_-$ so that $\Psi_+ = \Psi_- + \Delta_g$, it follows that when $g \ne 1$, we will always have $q(\Psi_+) \leq q(\Psi_-)$  In particular, the conclusion of Lemma~\ref{lem:local-artin-conductor-delta} holds also for any $\Psi_+$ and $\Psi_-$ constructed by Lemma~\ref{lem:class-functions-basepoint}.
	
	We now turn to the analysis of $q(\Psi_+ + \Psi_-)$.
	
	\begin{lemma} \label{lem:sum-conductor}
		With notation as above, suppose that $g \ne 1$.  Let $\Psi_+$ and $\Psi_-$ be as in Lemma~\ref{lem:class-functions-artin}.  Then
			\[
				q(\Psi_+ + \Psi_-) 
					= \prod_{\chi \in \mathrm{Irr}(G)} q(\chi)^{ |\langle \Delta_g, \chi\rangle|} 
					\leq |\mathrm{Disc}(K)|^{2 \sum_{\chi \in \mathrm{Irr}(G)} \frac{|\langle \Delta_g, \chi\rangle| \chi(1)}{|G|}}.
			\]
		Moreover, we also have
			\[
				q(\Psi_+ + \Psi_-)
					\leq |\mathrm{Disc}(K)|^{ \frac{ \sqrt{2} \, \|\Delta_g\|_2}{|G|^{1/2}}}
					\leq |\mathrm{Disc}(K)|^{ \frac{ \sqrt{2} \, |C_G(g)|^{1/2} }{|G|^{1/2}}}.
			\]
	\end{lemma}
	\begin{proof}
		The first claim follows simply from the linearity of the Artin conductor, the fact that
			\[
				\Psi_+ + \Psi_-
					= \sum_{\chi \in \mathrm{Irr}(G)} |\langle \Delta_g, \chi\rangle| \chi,
			\]
		and Lemma~\ref{lem:conductor-vs-degree}.  For the second, we have for any prime $\mathfrak{p}$ of $k$ that
			\begin{equation}\label{eqn:conductor-cs}
				v_\mathfrak{p}(\mathfrak{f}_{\Psi_+ + \Psi_-})
					= \sum_{\chi \in \mathrm{Irr}(G)} |\langle \Delta_g, \chi\rangle| v_\mathfrak{p}(\mathfrak{f}_\chi)
					\leq \|\Delta_g\|_2 \left( \sum_{\chi \in \mathrm{Irr}(G)} v_\mathfrak{p}(\mathfrak{f}_\chi)^2\right)^{1/2}.
			\end{equation}
		Appealing to Lemma~\ref{lem:conductor-vs-degree} and the decomposition of the regular representation, we find
			\[
				\sum_{\chi \in \mathrm{Irr}(G)} v_\mathfrak{p}(\mathfrak{f}_\chi)^2
					\leq \frac{2 v_\mathfrak{p}(\mathfrak{D}_{K/k}) }{|G|} \sum_{\chi \in \mathrm{Irr}(G)} \chi(1) v_\mathfrak{p}(\mathfrak{f}_\chi) = \frac{2 v_\mathfrak{p}(\mathfrak{D}_{K/k})^2}{|G|}.
			\]
		Substituting this into \eqref{eqn:conductor-cs} and appealing to Lemma~\ref{lem:delta-l2}, we obtain the second claim.
	\end{proof}
	
	\begin{corollary}
		\label{cor:less-simplified-conductor}
		With notation as above, let $\Psi_+$ and $\Psi_-$ be as in Lemma~\ref{lem:class-functions-artin}.  Assume $g \ne 1$, and define
			\[
				\alpha(g)
					:= \frac{\sqrt{2} \|\Delta_g\|_2 }{2 |G|^{1/2}} + \frac{|C_G(g)|}{2r|G|} \max\left\{ \frac{\#\{ \sigma \in H : \sigma \text{ rat. equiv to $g$}\}}{|H|-1} : H \leq G \text{ $p$-by-cyclic}\right\},
			\]
		where $\|\Delta_g\|_2 \leq |C_G(g)|^{1/2}$ is given exactly in Lemma~\ref{lem:delta-l2}.  Then
			\[
				q(\Psi_+) 
					\leq q(\Psi_-)
					\leq |\mathrm{Disc}(K)|^{\alpha(g)}.
			\]
	\end{corollary}
	\begin{proof}
		This follows from Lemmas~\ref{lem:sum-conductor} and \ref{lem:local-artin-conductor-delta}.
	\end{proof}
	
	\begin{remark}
		The maximum over subgroups $H \leq G$ is taken over $p$-by-cyclic subgroups, i.e. extensions of cyclic groups by $p$-groups, and is relevant because each ramification group $G_i$ associated with a prime $\mathfrak{p}$ is $p$-by-cyclic, where $p$ is the rational prime lying below $\mathfrak{p}$.  In particular, if we define $\alpha^\mathrm{tame}(g)$ by exactly the same expression as $\alpha(g)$ except restricting $H \leq G$ further to cyclic subgroups (as is necessarily the case for the primes of tame ramification), one may obtain the alternate bound $q(\Psi_-) \ll_{G,[k:\mathbb{Q}]} |\mathrm{Disc}(K)|^{\alpha^{\mathrm{tame}}(g)}$, since the contribution from the primes of wild ramification may be bounded solely in terms of $G$ and $[k:\mathbb{Q}]$.
	\end{remark}

%%%%%%%%%
%%%%%%%%%
%%%%%%%%%
\section{Unconditional results for symmetric groups}
	\label{sec:symmetric-groups}
%%%%%%%%%
%%%%%%%%%
%%%%%%%%%

	We now show how one may do substantially better than the general unconditional results of the previous section, at least in the special case of symmetric groups, by carrying out the proofs of Theorem~\ref{thm:Sn-cycle-intro} and Theorem~\ref{thm:Sn-asymptotic-simplified}.  The key technical result in this section is the following, which, when combined with Theorem~\ref{thm:general-approach} and Lemma~\ref{lem:conductor-vs-degree}, yields Theorem~\ref{thm:Sn-cycle-intro}.
	
	\begin{theorem} \label{thm:product-of-cycles-characters}
		Let $n \geq 2$, and let $G = S_n$ be the symmetric group of degree $n$.  Let $C \subseteq S_n$ be a nontrivial conjugacy class, and let $\ell_1, \dots, \ell_m \geq 2$ denote the lengths of the nontrivial cycles in the cycle decomposition of the class $C$.  Then there are characters $\Psi_+$ and $\Psi_-$ of $S_n$ such that:
			\begin{enumerate}[a)]
				\item the difference $\Psi_+ - \Psi_-$ is supported on the conjugacy class $C$;
				\item $\Psi_+$ and $\Psi_-$ are non-negative integral linear combinations of monomial characters;
				\item the multiplicity of the trivial representation inside $\Psi_+$ is $1$, and its multiplicity inside $\Psi_-$ is $0$;
				\item $\Psi_+(1) = \Psi_-(1) = 2^{m-1} n! \cdot \prod_{i=1}^m \frac{(\ell_i-2)2^{\ell_i-2} + 1}{\ell_i!}$; and
				\item $\langle \Psi_+, \mathrm{sgn} \rangle \leq 2^{m-1} \prod_{i=1}^m (\ell_i-1)$.
			\end{enumerate}
	\end{theorem}
	
	Our approach here takes advantage of the explicit nature of the representation theory of the symmetric group.  See \cite[Chapter 4]{FultonHarris} for a general source.

\subsection{Preliminaries on the representation theory of the symmetric group}
	
	Fix $n \geq 2$, and let $S_n$ denote the symmetric group of degree $n$.  By means of their cycle decomposition, conjugacy classes of $S_n$ are in correspondence with partitions $\mu$ of $n$.  As is standard, we will write $\mu \vdash n$ to denote that $\mu$ is a partition of $n$.  
	
	The irreducible characters of $S_n$ are also naturally in correspondence with partitions of $n$, and if $\lambda \vdash n$, then the \emph{hook length formula} gives a formula for the degree of the character $\chi_\lambda$ associated with $\lambda$.  Consider the Ferrers diagram associated with $\lambda$.  A \emph{hook} $\eta$ in $\lambda$ is associated with a cell in the Ferrers diagram, and consists of the cell, all cells below it in the same column, and all cells to its right in the same row.  The length of the hook is the number of cells in the hook, which we shall denote by $|\eta|$.  There are $n$ hooks $\eta_1,\dots,\eta_n$ associated with any partition $\lambda \vdash n$, and the hook length formula \cite[Formula 4.12]{FultonHarris} states that
		\[
			\chi_\lambda(1)
				= \frac{ n!}{\prod_{i=1}^n |\eta_i|}.
		\]
	As examples of particular importance to us, for any $0 \leq i \leq n-1$, let $\lambda_i \vdash n$ be the partition consisting of the single part $n-i$, together with $i$ $1$'s.  Thus, the Ferrers diagram of $\lambda_i$ is itself a single hook, and the characters associated with these $\lambda_i$ are sometimes referred to as hook characters.  Let $\chi_i := \chi_{\lambda_i}$ be the character of $\lambda_i$, we find by the hook length formula that
		\[
			\chi_i(1)
				= \binom{n-1}{i}.
		\]
	In fact, though we shall mostly not need this fact, the character $\chi_i$ is the character of the $i$-th exterior power of the standard representation (see, e.g., \cite[Exercise 4.6]{FultonHarris}).  Thus, $\chi_0$ is the trivial representation, while $\chi_{n-1}$ is the sign representation.  
	
	There are multiple ways of determining the value of the character $\chi_\lambda$ at a non-trivial conjugacy class associated with some $\mu \vdash n$, but for our purposes, the most important is the \emph{Murnaghan--Nakayama rule} \cite[Problem 4.45]{FultonHarris}.  Let $\nu$ be the partition of some $m < n$ obtained by removing a single part from $\mu$.  Then the Murnaghan--Nakayama rule says that
		\begin{equation} \label{eqn:murnaghan-nakayama}
			\chi_\lambda(\mu)
				= \sum_{\rho} (-1)^{\mathrm{ht}(\rho)} \chi_{\lambda - \rho} (\nu),
		\end{equation}
	where the summation runs over ``skew hooks'' (or ``rim hooks'') $\rho$ of $\lambda$ of size $|\rho| = n - |\nu|$.  A \emph{skew hook} of $\lambda$ is a contiguous piece of the southeast boundary of the Ferrers diagram of $\lambda$.  When a skew hook $\rho$ is removed from the Ferrers diagram of $\lambda$, one obtains the Ferrers diagram of a partition $\lambda - \rho$ of $n - |\rho|$.  Hence, the term $\chi_{\lambda - \rho}(\nu)$ in \eqref{eqn:murnaghan-nakayama} makes sense, and the Murnaghan--Nakayama rule thus gives a recursive means of computing the character table of $S_n$, provided one has at hand the character table of smaller degree symmetric groups.  The remaining unexplained term in \eqref{eqn:murnaghan-nakayama}, $\mathrm{ht}(\rho)$, is the height of the skew hook $\rho$, which is defined to be the number of rows involved in $\rho$ minus $1$.
	
	For our purposes, the key example is when $\mu$ is the partition of $n$ with a single part (so that $\mu$ is associated with the conjugacy class of an $n$-cycle in $S_n$).  Because $\mu$ consists of a single part, the only partition $\nu$ obtained by removing a part from $\mu$ is the trivial partition of $0$.  Letting $\lambda \vdash n$ be arbitrary, it follows that the summation in \eqref{eqn:murnaghan-nakayama} is supported over those skew hooks of $\lambda$ of size $n$.  In particular, it will be $0$ unless $\lambda$ is itself a skew hook, which implies that $\lambda$ must be a single hook and thus equal to one of the $\lambda_i$, $0 \leq i \leq n-1$, considered above.  Moreover, the Ferrers diagram associated with the partition $\lambda_i$ has $i+1$ rows, so we have $\mathrm{ht}(\lambda_i) = (-1)^i$, and thus $\chi_i(\mu) = (-1)^i$.
	
	We will appeal to some further facts about the representation theory of $S_n$ shortly, but with this, we are nearly ready to construct the key players in this section, which will be characters $\Psi_+$ and $\Psi_-$ detecting the class of an $n$-cycle.  In particular, the first goal of this section is to prove Theorem~\ref{thm:product-of-cycles-characters} in the special case of an $n$-cycle.
	
\subsection{Characters detecting $n$-cycles in $S_n$}

	The following lemma records the upshot of the above discussion.
	
	\begin{lemma} \label{lem:n-cycle-evaluation}
		Let $n \geq 2$ and let $g \in S_n$ be an $n$-cycle.  Let $\chi$ be an irreducible character of $S_n$.  
		Then $\chi(g) = 0$ unless there is $0 \leq i \leq n-1$ such that $\chi = \chi_i$ is the character of $S_n$ associated with the partition $\lambda_i \vdash n$ given by $\lambda_i := (n-i) + 1 + 1 + \dots + 1$, where there are $i$ $1$'s.  Moreover, for any $0 \leq i \leq n-1$, we have $\chi_i(g) = (-1)^i$ and $\chi_i(1) = \binom{n-1}{i}$.
	\end{lemma}
	\begin{proof}
		The claim that $\chi(g) = 0$ unless $\chi = \chi_i$ for some $0 \leq i \leq n-1$ follows from the Murnaghan--Nakayama formula \eqref{eqn:murnaghan-nakayama} as explained above, since we must have $\nu = 0$, as does the evaluation $\chi_i(g) = (-1)^i$.  The evaluation $\chi_i(1) = \binom{n-1}{i}$ follows from the hook length formula, again as explained above.
	\end{proof}
	
	As a consequence, we have the following, which leads to strong bounds for primes whose Frobenius is an $n$-cycle, but only assuming the Artin conjecture.  Thus, we view this as the limitation of our method for the class of an $n$-cycle.  While this result is not required for any of our main theorems, we nevertheless consider it a useful benchmark.  Additionally, a similar result was obtained by differnt methods by Zhu in his Ph.D. thesis \cite[Corollary~1.8]{Zhu-Thesis}.
	
	\begin{lemma} \label{lem:n-cycle-artin}
		Let $n \geq 2$ and for each $0 \leq i \leq n-1$, let $\chi_i$ be the irreducible character of $S_n$ from Lemma~\ref{lem:n-cycle-evaluation}.  Define characters $\Psi_+$ and $\Psi_-$ of $S_n$ by
			\[
				\Psi_+ = \sum_{\substack{ 0 \leq i \leq n-1 \\ i \text{ even}}} \chi_i
					\quad \text{and} \quad
				\Psi_- = \sum_{\substack{ 0 \leq i \leq n-1 \\ i \text{ odd}}} \chi_i.
			\]
		Then the difference $\Psi_+ - \Psi_-$ is supported on the conjugacy class of an $n$-cycle, and we have $\Psi_+(1) = \Psi_-(1) = 2^{n-2}$.  
		
		If moreover the Artin holomorphy conjecture holds for the Artin $L$-functions associated with $\chi_i$ for a Galois $S_n$ extension of number fields $K/k$, then the characters $\Psi_+$ and $\Psi_-$ satisfy conditions (ii) and (iii) of Theorem~\ref{thm:general-approach}, and we have $\max\{ q(\Psi_+), q(\Psi_-) \} \leq |\mathrm{Disc}(K)|^{\frac{2^{n-1}}{n!}}$.
	\end{lemma}
	\begin{proof}
		This follows from Lemma~\ref{lem:inner-product-when-conjugacy}, Lemma~\ref{lem:n-cycle-evaluation}, and Lemma~\ref{lem:conductor-vs-degree}, when combined with the standard fact that
			\[
				\sum_{\substack{ 0 \leq i \leq n-1 \\ i \text{ even}}} \binom{n-1}{i} = 
				\sum_{\substack{ 0 \leq i \leq n-1 \\ i \text{ odd}}} \binom{n-1}{i} =
				2^{n-2}.
			\]
	\end{proof}
	
	We now wish to prove Theorem~\ref{thm:product-of-cycles-characters}. The special case that $C$ is the class of an $n$-cycle turns out to be the most important, and requires an analogue of Lemma~\ref{lem:n-cycle-artin} but for characters $\Psi_+$ and $\Psi_-$ whose $L$-functions are provably holomorphic away from $s=1$.  The starting point is the following.
	
	\begin{lemma} \label{lem:consecutive-monomial}
		Let $n \geq 2$, and for each $0 \leq i \leq n-1$, let $\chi_i$ be the character of $S_n$ as in Lemma~\ref{lem:n-cycle-artin}.  Then $\chi_{i-1} + \chi_{i}$ is a monomial character of $S_n$ for every $1 \leq i \leq n-1$.
	\end{lemma}
	\begin{proof}
		Let $H = S_i \times S_{n-i} \leq S_n$, and let $\psi$ be the character of $H$ given by $\mathrm{sgn}_{S_i} \times \mathbf{1}_{S_{n-i}}$.  Note that $\psi$ is $1$-dimensional, so its induction is by definition a monomial character of $S_n$.  We claim that in fact $\mathrm{Ind}_H^{S_n} \psi = \chi_{i-1} + \chi_i$.
		
		The sign character of $S_i$ is associated with the partition $\mu_i \vdash i$ into $i$ $1$'s, i.e. $\mu_i = 1 + \dots + 1$, while the trivial character of $S_{n-i}$ is associated with the partition $\mu_{n-i}$ into the single part $n-i$.  We may therefore compute the induction $\mathrm{Ind}_{H}^{S_n} \chi_{\mu_i} \times \chi_{\mu_{n-i}}$ by means of Pieri's formula \cite[(A.7)]{FultonHarris}.  In particular, every irreducible constituent of $\mathrm{Ind}_{H}^{S_n} \chi_{\mu_i} \times \chi_{\mu_{n-i}}$ will have multiplicity $1$, and the constituents appearing will be associated with Ferrers diagrams obtained from that of $\mu_i$ by adding $n-i$ cells, no two of which are in the same column.  But only $\lambda_{i-1}$ and $\lambda_i$ are obtainable in this manner, and thus we find $\mathrm{Ind}_{H}^{S_n} \chi_{\mu_i} \times \chi_{\mu_{n-i}} = \chi_{i-1} + \chi_i$, as claimed.
	\end{proof}
	
	With this, we can prove Theorem~\ref{thm:product-of-cycles-characters}.
	
	\begin{proof} [Proof of Theorem~\ref{thm:product-of-cycles-characters}]
		Begin by supposing that $C$ is the conjugacy class of an $n$-cycle.  We then define
			\begin{align*}
				\Psi_+ 
					&= \chi_0 + \sum_{1 \leq i \leq n-1} 2\left\lceil{ \frac{i-1}{2}}\right\rceil \cdot \chi_i \\
					&= \chi_0 + 2(\chi_2 + \chi_3) + 4 (\chi_4 + \chi_5) + \dots,
			\end{align*}
		where the summation above ends with $(n-3)\cdot(\chi_{n-3} + \chi_{n-2}) + (n-1)\chi_{n-1}$ if $n$ is odd and with $(n-2) \cdot (\chi_{n-2} + \chi_{n-1})$ if $n$ is even.  In either case, since both $\chi_0$ and $\chi_{n-1}$ are $1$-dimensional characters of $S_n$ and hence monomial, it follows from Lemma~\ref{lem:consecutive-monomial} that $\Psi_+$ is a non-negative integral linear combination of monomial characters.  We also define
			\begin{align*}
				\Psi_- 
					&= \sum_{1 \leq i \leq n-1} \left(2 \left\lfloor \frac{i-1}{2}\right\rfloor + 1 \right) \cdot \chi_i \\
					&= (\chi_1 + \chi_2) + 3(\chi_3+\chi_4) + \dots,
			\end{align*}
		where the summation ends with $(n-2)(\chi_{n-2}+\chi_{n-1})$ if $n$ is odd and with $(n-3)(\chi_{n-3}+\chi_{n-2}) + (n-1) \chi_{n-1}$ if $n$ is even.  As with $\Psi_+$, Lemma~\ref{lem:consecutive-monomial} implies that $\Psi_-$ is a non-negative linear combination of monomial characters.  Moreover, the difference $\Psi_+ - \Psi_- = \chi_0 - \chi_1 + \chi_2 - \dots$ is $\Delta_g$ by construction (where $g$ is an $n$-cycle), so the difference $\Psi_+ - \Psi_-$ is supported on the conjugacy class $C$.  Finally, we observe that
			\[
				\Psi_+(1) = \Psi_-(1) = \frac{\Psi_+(1)+\Psi_-(1)}{2}
			\]
		and that
			\[
				 \Psi_+(1)+\Psi_-(1)
					= 1 + \sum_{i=1}^{n-1} (2i-1) \binom{n-1}{i}
					= 2 + (n-2)2^{n-1}.
			\]
		Hence, $\Psi_+(1) = \Psi_-(1) = 1 + (n-2)2^{n-2}$, nearly completing the proof if $C$ is an $n$-cycle.
		It remains only to note that, since $\mathrm{sgn} = \chi_{n-1}$, $\langle \Psi_+, \mathrm{sgn} \rangle = 2\lceil \frac{n-2}{2}\rceil \leq n-1$.  However, we also find it convenient to note here that $\langle \Psi_-,\mathrm{sgn} \rangle = 2\lfloor \frac{n-2}{2}\rfloor+1 \leq n-1$ as well.
		
		Now suppose that $C \subset S_n$ is an arbitrary non-trivial conjugacy class and let $\ell_1,\dots,\ell_m \geq 2$ be the lengths of the non-trivial cycles in the cycle decomposition of the class $C$.  Let $H = S_{\ell_1} \times \dots \times S_{\ell_m}$, which we may view as a subgroup of $S_n$.
		For each $1 \leq i \leq r$, let $\Psi_{i,+}$ and $\Psi_{i,-}$ be the characters of $S_{\ell_i}$ of degree $1+(\ell_i-2)2^{\ell_i-2}$ constructed above so that the difference $\Psi_{i,+} - \Psi_{i,-}$ is supported on the class of an $\ell_i$-cycle.  By means of the projection map $H \to S_{\ell_i}$, each $\Psi_{i,+}$ and $\Psi_{i,-}$ may also be viewed as a character of $H$, and we define
			\[
				\Psi_{H,+} = \sum_{ \substack{ \Sigma \subseteq \{1,\dots,m\} \\ |\Sigma| \text{ even}}} \left[\prod_{i \not \in \Sigma} \Psi_{i,+} \,\cdot\, \prod_{\substack{ 1 \leq i \leq r \\ i \in \Sigma}} \Psi_{i,-} \right]
					\quad \text{and} \quad
				\Psi_{H,-} = \sum_{ \substack{ \Sigma \subseteq \{1,\dots,m\} \\ |\Sigma| \text{ odd}}} \left[\prod_{i \not\in \Sigma} \Psi_{i,+} \,\cdot\, \prod_{\substack{ 1 \leq i \leq r \\ i \in \Sigma}} \Psi_{i,-} \right].	
			\]
		These are characters of $H$ defined precisely so that 
			\[
				\Psi_{H,+} - \Psi_{H,-} 
					= \prod_{i=1}^m (\Psi_{i,+} - \Psi_{i,-}).
			\]
		If we define $\Psi_+ = \mathrm{Ind}_H^{S_n} \Psi_{H,+}$ and $\Psi_- = \mathrm{Ind}_H^{S_n} \Psi_{H,-}$, the difference $\Psi_+ - \Psi_-$ is therefore supported on the class $C$.  Moreover, we see that 
			\[
				\Psi_{H,+}(1) = \Psi_{H,-}(1) = 2^{m-1} \prod_{i=1}^m (1 + (\ell_i-2)2^{\ell_i-2}).
			\]
		We similarly find that $\langle \Psi_+, \mathrm{sgn} \rangle \leq 2^{m-1} \prod_{i=1}^m (\ell_i-1)$, since the restriction of the sign character to $H$ is the product of the sign characters of the factors $S_{\ell_i}$.
		Finally, since the index of $H$ in $S_n$ is $\frac{n!}{\ell_1! \cdot \dots \cdot \ell_m!}$, the result follows.
	\end{proof}
	
%%%%%%%%%	
\subsection{Bounds arising from abelian subgroups and the proof of Theorem~\ref{thm:Sn-asymptotic-simplified}}
%%%%%%%%%
	
	We now turn to the proof of Theorem~\ref{thm:Sn-asymptotic-simplified}, which we in fact establish in a more optimized form:
	
	\begin{theorem} \label{thm:optimized-Sn-asymptotic}
		There are absolute constants $c_1,c_2>0$ so that for any $n \geq 2$ and any conjugacy class $C \subset S_n$, the Linnik exponent 
			\[
				\alpha(S_n,C) = c_1 \exp(-c_2 n)
			\]
		is admissible in \eqref{eqn:LinnikExponent}, and with an effectively computable implied constant for every $\epsilon>0$.  Explicitly, we may take $c_1 = 1504$ and $c_2 = 0.295$.
	\end{theorem}

	As indicated in the introduction, the proof relies on leveraging Theorem~\ref{thm:Sn-cycle-intro} against a complementary approach, in particular that of Thorner and Zaman \cite{ThornerZaman-Explicit}, which relies on abelian subgroups meeting the desired conjugacy class $C$.  The following lemma is a simplified version of their work that made earlier work of Weiss \cite{Weiss} explicit.
	
	\begin{lemma} \label{lem:abelian-least-prime}
		Let $K/k$ be a Galois extension of number fields with Galois group $G$.  Let $C \subseteq G$ be a conjugacy class, and let $A \leq G$ be an abelian subgroup such that $A \cap C$ is non-empty.  Then there is a degree one prime $\mathfrak{p}$ of $k$ that is unramified in $K$ with $\mathrm{Frob}_\mathfrak{p} \in C$ satisfying
			\[
				\mathrm{Nm}_{k/\mathbb{Q}} \mathfrak{p}
					\ll_{G,[k:\mathbb{Q}]} |\mathrm{Disc}(K)|^{\frac{1042}{|A|}}.
			\]
	\end{lemma}
	\begin{proof}
		Let $K^A$ be the subfield of $K$ fixed by $A$, and define $Q := \max\{ \mathrm{Nm}_{K^A/\mathbb{Q}} \mathfrak{f}_\chi\}$, where the maximum is taken over the irreducible characters of $A$.  By Lemma~\ref{lem:conductor-vs-degree}, we see that $Q \leq |\mathrm{Nm}_{K^A/\mathbb{Q}} \mathfrak{D}_{K/K^A}|^{\frac{2}{|A|}}$.  Hence, from Thorner and Zaman \cite[Theorem 1.1]{ThornerZaman-Explicit}, we find that there is a prime with $\mathrm{Frob}_\mathfrak{p} \in C$ satisfying
			\begin{align*}
				\mathrm{Nm}_{k/\mathbb{Q}} \mathfrak{p}
					& \ll_{G,[k:\mathbb{Q}]} |\mathrm{Disc}(K^A)|^{694} Q^{521}  \\
					& \leq |\mathrm{Disc}(K^A)|^{694} |\mathrm{Nm}_{K^A/\mathbb{Q}} \mathfrak{D}_{K/K^A}|^{\frac{1042}{|A|}} \\
					& = \left( |\mathrm{Disc}(K^A)|^{\frac{694 |A|}{1042}} |\mathrm{Nm}_{K^A/\mathbb{Q}} \mathfrak{D}_{K/K^A}| \right)^{\frac{1042}{|A|}} \\
					& \leq |\mathrm{Disc}(K)|^{\frac{1042}{|A|}}
			\end{align*}
		on using the conductor-discriminant formula.
	\end{proof}
	
	To prove Theorem~\ref{thm:optimized-Sn-asymptotic}, we therefore wish to consider the largest abelian subgroup of $S_n$ meeting a given class $C$.  This is essentially a standard exercise, but we provide the gist of the proof.
	
	\begin{lemma} \label{lem:abelian-bound}
			Let $n \geq 2$, let $C \subseteq S_n$ be a conjugacy class, and let $\ell_1,\dots,\ell_m \geq 2$ denote the lengths of the nontrivial cycles in the cycle decomposition of elements in $C$.  Set $\ell := \ell_1 + \dots + \ell_m$ (taken to be $0$ if $C$ is the class of the identity).  Then the largest abelian subgroup of $S_n$ containing an element in the class $C$ has order
				\[
					\prod_{i=1}^m \ell_i \cdot \begin{cases}
						1, & \text{if } n-\ell = 0,1, \\
						3^{\lfloor \frac{n-\ell}{3} \rfloor}, & \text{if } n-\ell \equiv 0 \pmod{3}, \\
						4 \cdot 3^{\lfloor \frac{n-\ell}{3} \rfloor - 1}, & \text{if } n-\ell \equiv 1 \pmod{3}, \text{ and}\\
						2 \cdot 3^{\lfloor \frac{n-\ell}{3} \rfloor}, & \text{if } n-\ell \equiv 2 \pmod{3}.
					\end{cases}
				\]
	\end{lemma}
	\begin{proof}
		Let $A$ be an abelian subgroup such that $A \cap C$ is non-empty, and let $g$ be an element in the intersection.  It follows that $A$ must be a subgroup of the centralizer $C_{S_n}(g)$, which is isomorphic to the direct product $A_0 \times S_{n-\ell}$, where $A_0$ is the subgroup generated by the disjoint cycles in $g$.  The largest choice for $A$ will therefore be isomorphic to $A_0 \times A_1$, where $A_1$ is the largest abelian subgroup of $S_{n-\ell}$.  A standard exercise shows that if $n-\ell \geq 2$, then the largest abelian subgroup of $S_{n-\ell}$ is dictated by $n-\ell \pmod{3}$.  If $n-\ell \equiv 0 \pmod{3}$, it is generated by $\lfloor \frac{n-\ell}{3}\rfloor$ disjoint $3$-cycles; if $n - \ell \equiv 2 \pmod{3}$, it is generated by $\lfloor \frac{n-\ell}{3}\rfloor$ disjoint $3$-cycles plus an additional disjoint transposition; and if $n-\ell \equiv 1 \pmod{3}$, it is generated by $\lfloor \frac{n-\ell}{3}\rfloor - 1$ disjoint $3$-cycles plus two additional disjoint transpositions.  The lemma follows.
	\end{proof}
	
	\begin{proof}[Proof of Theorem~\ref{thm:optimized-Sn-asymptotic}]
		Let $C \subseteq S_n$ be a conjugacy class.  For each integer $\ell \geq 1$, let $\lambda_\ell$ be $1/n$ times the number of cycles of length $\ell$ in $C$.  Recalling the formula for $\alpha(C)$ in Theorem~\ref{thm:Sn-cycle-intro}, we see that if we set $a_\ell := \log\left(\frac{\ell!}{(\ell-2)2^{\ell-1}+2} \right)$ (which equals $0$ if $\ell=1$), then
			\[
				\alpha(C)
					= 2^{m-1} \prod_{i=1}^m \frac{(\ell_i-2)2^{\ell_i-2}+1}{\ell_i!}
					= \frac{1}{2} \exp\left(-n \sum_{\ell=1}^\infty a_\ell\lambda_\ell \right).
			\]
		Similarly, if we define $b_1 = \frac{\log 3}{3}$ and $b_\ell = \log \ell$ for $\ell \geq 2$, by Lemma~\ref{lem:abelian-bound}, we also see that the largest abelian subgroup $A$ meeting the class $C$ has order satisfying
			\[
				|A|^{-1}
					\leq 3^{1/3} \exp\left( -n \sum_{\ell=1}^\infty b_\ell \lambda_\ell\right).
			\]
		It therefore follows on combining Theorem~\ref{thm:Sn-cycle-intro} with Lemma~\ref{lem:abelian-least-prime} that Theorem~\ref{thm:optimized-Sn-asymptotic} will hold for the class $C$ with any constant $c_2$ satisfying
			\begin{equation} \label{eqn:c2-fixed-class}
				c_2 \leq \max\left\{ \sum_{\ell=1}^\infty a_\ell \lambda_\ell, \sum_{\ell=1}^\infty b_\ell \lambda_\ell \right\}.
			\end{equation}
		To obtain a value of $c_2$ that is valid for every conjugacy class $C$, it therefore suffices to find the minimum of the expression on the right-hand side of \eqref{eqn:c2-fixed-class} taken over sequences $\{ \lambda_\ell\}_{\ell=1}^\infty$ of non-negative real numbers subject to the constraint
			\[
				\sum_{\ell=1}^\infty \ell \lambda_\ell = 1.
			\]
		
		This is a straightforward linear optimization problem.  By continuity, a minimizing sequence must satisfy
			\begin{equation} \label{eqn:hyperplane}
				\sum_{\ell=1}^\infty a_\ell \lambda_\ell
					= \sum_{\ell=1}^\infty b_\ell \lambda_\ell,
			\end{equation}
		and we note that $a_\ell > b_\ell$ for $\ell \geq 7$, $a_\ell < b_\ell$ for $1 \leq \ell \leq 6$, and the differences $a_\ell - b_\ell$ are pairwise distinct.  From these considerations, it follows that the minimizing sequence will be non-zero at exactly two indices, with one index $\leq 6$ and the other $\geq 7$.  It is then straightforward to verify that the minimizer has $\lambda_7 \ne 0$, with a finite computation then revealing that also $\lambda_2 \ne 0$.  Explicitly, then, the minimizing sequence is given by
			\[
				\lambda_2 
					= \frac{\log 360 - \log 161}{\log 16588800 - \log 25921}
					\approx 0.12454,
			\]
			\[
				\lambda_7
					= \frac{\log 2}{\log 16588800 - \log 25921} 
					\approx 0.10728,
			\]
		and $\lambda_\ell = 0$ for $\ell \ne 2,7$.  On using \eqref{eqn:c2-fixed-class}, this shows that we may take any
			\[
				c_2 
					\leq \frac{\log 2 \cdot \left(\log 360 - \log 23 \right)}{\log 16588800 - \log 25921} 
					\approx 0.29507.
			\]
		We also see that we may take $c_1 = 3^{1/3} \cdot 1042 \approx 1502.8$.
		
		This all works without yet considering effectivity, but it is straightforward to modify the above to provide an effective bound, especially since the bound given by Lemma~\ref{lem:abelian-bound} is always effective.  Since $\langle \Psi_+, \mathrm{sgn} \rangle \leq 2^{m-1} \prod_{i=1}^m (\ell_i -1)$, the exponent will need to be increased by at most
			\[
				\frac{2^m}{n!} \prod_{i=1}^m (\ell_i - 1)
					\leq \left(\frac{4^{1/3} e}{n}\right)^n,
			\]
		as follows on observing that the product is dominated by taking each $\ell_i = 3$ and $m = n/3$ and invoking the inequality $n! \geq \left(\frac{n}{e}\right)^n$.  It is then straightforward to check that
			\[
				\exp\left(c_2n + n\log\left(\frac{4^{1/3}e}{n}\right)\right)
			\]
		is $< 9$ for every $n \geq 2$, and in fact is $< 1$ for every $n \geq 6$.  This yields the claim with $c_1 = 1504$ for every $n \geq 6$, and it is routine to verify that $c_1 = 1504$ is also admissible for every $n \leq 5$; for example, this follows from Theorem~\ref{thm:least-prime-rational-class-intro}.
	\end{proof}

%%%%%%%%%
%%%%%%%%%
%%%%%%%%%
\section{Alternative bounds for symmetric groups}	
	\label{sec:small-symmetric}
%%%%%%%%%
%%%%%%%%%
%%%%%%%%%

	In this section, we record the results of the explicit computations leading to Theorem~\ref{thm:small-Sn-intro}.  We also provide alternative class functions than those used in the previous section that will, for certain classes $C$, yield even stronger bounds on $\alpha(S_n,C)$ than we have stated.  As in Section~\ref{sec:Results}, we restrict our discussion to the non-identity conjugacy classes.
	
\subsection{Symmetric groups $S_n$ with $n \leq 10$ and the proof of Theorem~\ref{thm:small-Sn-intro}}

	The main result in this section is the following more refined result that implies Theorem~\ref{thm:small-Sn-intro}.
	
	\begin{theorem}\label{thm:small-Sn-bowels}
		Let $3 \leq n \leq 10$ and let $C \subset S_n$ be a non-identity conjugacy class.  For $3 \leq n \leq 6$, the Linnik exponent $\alpha(S_n,C)$ given in Table~\ref{tbl:small-Sn-intro} is admissible in \eqref{eqn:LinnikExponent}, for $7 \leq n \leq 9$, the exponent given in Table~\ref{tbl:small-Sn-appendix} is admissible, and for $n=10$, the exponent given in Table~\ref{tbl:S10-appendix} is admissible.		
		If $k$ does not admit a quadratic subfield, then the implied constant in \eqref{eqn:LinnikExponent} is effectively computable when $\epsilon > {\epsilon_{\mathrm{eff}}^{\mathbb{Q}}(C)}/{[k:\mathbb{Q}]}$, where $\epsilon_\mathrm{eff}^\mathbb{Q}(C)$ is also given in Tables~\ref{tbl:small-Sn-intro}, \ref{tbl:small-Sn-appendix}, and \ref{tbl:S10-appendix}.  If $k$ does admit a quadratic subfield, then the implied constant is effective when $\epsilon > \max\left\{ \frac{\epsilon_{\mathrm{eff}}^{\mathbb{Q}}(C)}{[k:\mathbb{Q}]}, \frac{2}{n! [k:\mathbb{Q}]} \right\}$.
	\end{theorem}

	\begin{proof}[Proof of Theorem~\ref{thm:small-Sn-bowels}]
		Suppose first we have a pair $(\Psi_+,\Psi_-)$ of characters of $S_n$ satisfying the conditions of Theorem~\ref{thm:general-approach} for a conjugacy class $C \ne \{1\}$, and so that $\Psi_+ - \Psi_- = \Delta_g$ for $g \in C$.  By Lemma~\ref{lem:local-artin-conductor-delta}, we would then have $q(\Psi_+) \leq q(\Psi_-)$, so it suffices to bound $q(\Psi_-)$ in terms of $|\mathrm{Disc}(K)|$ to obtain a Linnik exponent $\alpha(S_n,C)$ by means of Corollary~\ref{cor:general-approach}.  In fact, it suffices only to do so for the contribution from the primes of tame ramification, since the contribution from the wild primes may be bounded in terms of $n$ and $[k:\Q]$.  Appealing to the formula for the Artin conductor \eqref{eqn:local-artin-conductor}, and noting that for primes of tame ramification, the group $G_0$ is the only non-trivial ramification group and must be cyclic, we find
		\begin{equation} \label{eqn:tame-conductor-bound}
			\frac{\log(q(\Psi_-))}{\log |\mathrm{Disc}(K)|}
				\leq \max\left\{ \frac{\Psi_-(1) - \frac{1}{|I|} \sum_{g \in I} \Psi_-(g)}{|G| - \frac{|G|}{|I|}} : 1 \ne I \leq G \text{ cyclic} \right\} + O_{n,[k:\Q]}(1).
		\end{equation}
		Computing this upper bound for a specific choice of $\Psi_-$ leads to the claimed Linnik exponent $\alpha(S_n,C)$, which is simply half of the above maximum.  The range of effectivity in $\epsilon$ is determined by the multiplicity of the sign representation in $\Psi_+$, which is easily read off for any specific choice of $\Psi_+$.
		Thus, to complete the proof of Theorem~\ref{thm:small-Sn-intro} it suffices to explain our choices for $(\Psi_+,\Psi_-)$ and how to certify that they are holomorphic away from $s=1$.  		
		We focus on the cases that $n \geq 5$, as for $S_3$ and $S_4$, the Artin conjecture is known, and $\Psi_+$ and $\Psi_-$ are given by Lemma~\ref{lem:class-functions-artin}.  
		
		We begin by computing in Magma \cite{Magma} the set of monomial characters of $S_n$, which we express in terms of their inner products with the irreducible characters.  We then pass the vectors of these inner products to Python for two linear programming steps.  The first step finds a basis for the monomial cone (i.e., a minimal set of monomial characters so that any real class function expressible as a non-negative linear combination of monomial characters is expressible as such a combination of characters in the basis).  This is a classical linear programming problem.  In the special case $n=5$, we have given this basis in Table~\ref{tbl:S5-monomial-basis} to demonstrate the idea, but we have not reproduced the bases for $6 \leq n \leq 10$ here.  This is because they consist of $33$, $51$, $83$, $152$, and $289$ characters, respectively.  (We note that these bases, along with all of our code described here and the result of running it, are available on a Github repository associated with this paper \cite{github}.)

		The second linear programming step finds for each conjugacy representative $g \in G$ a vector $\mathbf{v}_g$ of non-negative integers (again corresponding to coefficients of irreducibles) of minimal degree such that both $\mathbf{v}_g$ and $\mathbf{v}_g + \Delta_g$ lie in the monomial cone.  This is a mixed integer linear programming problem, since the vector $\mathbf{v}_g$ is required to be integral but its representation in terms of the basis of the monomial cone need not be.  It is this step that most benefits from the passage to Python, as we may employ the \verb^scipy.optimize.milp^ package.  We then import the resulting vectors $\mathbf{v}_g$ back into Magma to complete the analysis of the Linnik exponents described above.
		
		We describe the output of this computation in more detail for $S_5$ and $S_6$.  For three non-trivial classes in $S_5$ and two non-trivial classes in $S_6$, the computation above reveals that the optimal class functions turn out to be expressible as non-negative integral linear combinations of monomial characters.  We have recorded these classes, along with the decomposition of the optimal class functions in terms of irreducibles, in Table~\ref{tbl:monomial-cone}.  (We have also included here the explicit optimal choices for $S_3$ and $S_4$.)  For the three remaining classes in $S_5$ and the $8$ remaining classes in $S_6$, the optimal class functions $\Psi_+$ and $\Psi_-$ are not in the monomial cone.  We have recorded the degree-optimal choices of $\Psi_+$ and $\Psi_-$ in Tables~\ref{tbl:basepoints-5} and \ref{tbl:basepoints-6}.  Additionally, for $S_5$ in particular, it is straightforward to verify by hand using the provided basis of the monomial cone that each of the provided $\Psi_+$ and $\Psi_-$ do indeed satsify conditions (ii) and (iii) of Theorem~\ref{thm:general-approach}. 
		
		For $7 \leq n \leq 10$, we have not reproduced the degree-optimal choices here, but they are available at \cite{github}.
	\end{proof}
	
	{\tiny
	\begin{table}[h]
		\begin{tabular}{|c|c|c|}
			\hline
			Character & Subgroup & $1$-dim. char \\ \hline
			$1$ & $S_5$ & $1$ \\
			$\sgn$ & $S_5$ & $\sgn$ \\
			$1+\std_4$ & $S_4$ & $1$  \\
			$\sgn + \std_4 \otimes \sgn$ & $S_4$ & $\sgn$ \\
			$1+\std_5$ & $\mathbb{F}_5 \rtimes \mathbb{F}_5^\times$ & $1$ \\
			$\sgn + \std_5\otimes \sgn$ & $\mathbb{F}_5 \rtimes \mathbb{F}_5^\times$ & $\sgn$ \\
			$\wedge^2 \std_4$ & $\mathbb{F}_5 \rtimes \mathbb{F}_5^\times$ & $\chi_4$ \\
			$\std_4 + \wedge^2 \std_4$ & $S_3 \times S_2$ & $1 \times \sgn$ \\
			$\std_4\otimes \sgn + \wedge^2\std_4$ & $S_3 \times S_2$ & $\sgn \times 1$ \\
			$1 + \std_4 + \std_5\otimes\sgn$ & $S_3 \times S_2$ & $1 \times 1$ \\
			$\sgn + \std_4\otimes\sgn + \std_5$ & $S_3 \times S_2$ & $\sgn \times \sgn$ \\
			$\std_5 + \std_5\otimes\sgn$ & $A_4$ & $\chi_3$ \\
			$\std_4 + \std_5\otimes\sgn + \wedge^2\std_4$ & $D_4$ & $\chi_2$ \\
			$\std_4\otimes\sgn + \std_5 + \wedge^2\std_4$ & $D_4$ & $\chi_2 \cdot \sgn$ \\
			$\std_4 + \std_4\otimes\sgn + \std_5 + \std_5\otimes\sgn + \wedge^2\std_4$ & $C_5$ & $\chi_5$ \\
			\hline
		\end{tabular}
		\begin{caption}{\label{tbl:S5-monomial-basis}A basis for the monomial cone in $S_5$, along with the subgroups and characters from which they are induced.  A character denoted $\chi_n$ is a $1$-dimensional character of order $n$, with $\chi_2 \ne \sgn$.}\end{caption}
	\end{table}}
	
	{\tiny
	\begin{table}[h]
		\begin{tabular}{|l|c|c|c|} \hline
			$n$ & $C$ & $\Psi_+$ & $\Psi_-$  \\ \hline \hline
				$3$ 
					& $[(12)]$ & {\tiny $1$} & {\tiny $\mathrm{sgn}$} \\ \hline
					& $[(123)]$ & {\tiny $1+\mathrm{sgn}$} & {\tiny $\mathrm{std}_2$} \\ \hline \hline
				$4$ 
					& $[(12)]$ & {\tiny $1+\mathrm{std}_3$} & {\tiny $\mathrm{sgn} + \mathrm{std}_3\otimes \mathrm{sgn}$} \\ \hline
					& $[(123)]$ & {\tiny $1+\mathrm{sgn}$} & {\tiny $\mathrm{std}_2$} \\ \hline
					& $[(1234)]$ & {\tiny $1+\mathrm{std}_3\otimes \mathrm{sgn}$} & {\tiny $\mathrm{sgn} + \mathrm{std}_3$} \\ \hline
					& $[(12)(34)]$ & {\tiny $1 + \mathrm{sgn} + 2\cdot \mathrm{std}_2$} & {\tiny $\mathrm{std}_3+\mathrm{std}_3 \otimes \mathrm{sgn}$} \\ \hline\hline
				$5$ 
					& $[(123)]$ & {\tiny $1 + \mathrm{sgn} + \mathrm{std}_4 + \mathrm{std}_4\otimes \mathrm{sgn}$} & {\tiny$ \mathrm{std}_5 + \mathrm{std}_5\otimes \mathrm{sgn}$} \\ \hline
					& $[(1234)]$ & {\tiny$1 + \mathrm{std}_5$} & {\tiny $\mathrm{sgn} + \mathrm{std}_5\otimes \mathrm{sgn}$}  \\ \hline
					& $[(12)(34)]$ & {\tiny $1 + \mathrm{sgn} + \mathrm{std}_5 + \mathrm{std}_5 \otimes \mathrm{sgn}$} & {\tiny $2 \cdot \wedge^2 \mathrm{std}_4$} \\ \hline \hline
				$6$
					& $[(12)(34)]$ & {\tiny \makecell[t]{$1 + \mathrm{sgn} + \mathrm{std}_5 + \mathrm{std}_5\otimes \mathrm{sgn} + \widetilde{\mathrm{std}_5}$ \\ $ + \widetilde{\mathrm{std}_5} \otimes \mathrm{sgn} + \mathrm{std}_9 + \mathrm{std}_9 \otimes \mathrm{sgn}$}} & {\tiny $2\cdot \wedge^2 \mathrm{std}_5 + 2\cdot \wedge^2 \mathrm{std}_5 \otimes \mathrm{sgn}$} \\ \hline
					& $[(1234)(56)]$ & {\tiny $1 + \mathrm{sgn} + \mathrm{std}_9 + \mathrm{std}_9 \otimes \mathrm{sgn}$} & {\tiny $\mathrm{std}_5+\mathrm{std}_5\otimes\mathrm{sgn} + \widetilde{\mathrm{std}_5}+\widetilde{\mathrm{std}_5}\otimes\mathrm{sgn}$} \\ \hline
				
		\end{tabular}
		\begin{caption}{\label{tbl:monomial-cone} Nontrivial classes in $S_n$, $3 \leq n \leq 6$, for which the optimal $\Psi_+$ and $\Psi_-$ are in the monomial cone.}
		\end{caption}
	\end{table}}
	
	{\tiny
	\begin{table}[h]
		\begin{tabular}{|c|c|c|} \hline
			$C$ & $\Psi_+$ & $\Psi_-$ \\ \hline \hline
				$[(12)]$ & {\tiny $1 + 2\cdot \mathrm{std}_4 + \mathrm{std}_5 \otimes \mathrm{sgn} + \wedge^2 \mathrm{std}_4$} 
				& {\tiny $\mathrm{sgn} + 2\cdot\mathrm{std}_4\otimes \mathrm{sgn} + \mathrm{std}_5 + \wedge^2 \mathrm{std}_4$} \\ \hline
				$[(12345)]$ & {\tiny $1 + 2\cdot \mathrm{sgn} + 2\cdot \wedge^2\mathrm{std}_4$} 
				& {\tiny $\mathrm{sgn} + \mathrm{std}_4 + \mathrm{std}_4\otimes \mathrm{sgn} + \wedge^2 \mathrm{std}_4$} \\ \hline
				$[(123)(45)]$ & {\tiny $1 + 2\cdot\mathrm{sgn} + 2\cdot\mathrm{std}_4\otimes \mathrm{sgn} + \mathrm{std}_5\otimes \mathrm{sgn} + \wedge^2 \mathrm{std}_4$} 
				& {\tiny $3\cdot \mathrm{sgn} + \mathrm{std}_4 + \mathrm{std}_4 \otimes \mathrm{sgn} + \mathrm{std}_5 + \wedge^2 \mathrm{std}_4$} \\ \hline
		\end{tabular}
		\begin{caption}{\label{tbl:basepoints-5} Degree optimal choices for the remaining classes in $S_5$.}
		\end{caption}
	\end{table}}
	
	{\tiny
	\begin{table}[h]
		\begin{tabular}{|c|c|c|} \hline
			$C$ & $\Psi_+$  & $\Psi_-$ \\ \hline \hline
				$[(12)]$ & {\tiny \makecell[t]{$1 + 3\cdot \mathrm{std}_5 + \widetilde{\mathrm{std}_5}\otimes \mathrm{sgn} + $ \\ $3\cdot \mathrm{std}_9 + 2\cdot \wedge^2 \mathrm{std}_5 + 2 \cdot \chi_{16}$}} 
				& {\tiny \makecell[t]{
					$\sgn + 3\cdot \std_5\otimes\sgn + \widetilde{\std_5}+3\cdot \std_9\otimes\sgn +$ \\ $ 2\cdot \wedge^2\std_5\otimes\sgn + 2\cdot \chi_{16}$
				}}\\ \hline
				$[(123)]$ & {\tiny \makecell[t]{$1 + \sgn + 2\cdot \std_5 + 2\cdot\std_5\otimes\sgn + \std_9 +$ \\ $ \std_9 \otimes \sgn + \wedge^2 \std_5 + \wedge^2\std_5\otimes\sgn$}}
				& {\tiny \makecell[t]{
					$\widetilde{\std_5}+\widetilde{\std_5}\otimes\sgn + \std_9 + \std_9\otimes\sgn + 2\cdot \chi_{16}$
				}}\\ \hline
				$[(1234)]$ & {\tiny $1+\std_5+\widetilde{\std_5} + \std_9 \otimes \sgn + \chi_{16} $} 
				& {\tiny \makecell[t]{
					$\sgn + \widetilde{\std_5} + \std_5\otimes\sgn + \std_9 + \chi_{16} $
				}}
				\\ \hline
				$[(12345)]$ & {\tiny \makecell[t]{$1 + 5\cdot \sgn + \std_5\otimes\sgn + \widetilde{\std_5}\otimes\sgn +$ \\ $ 2\cdot \std_9 \otimes \sgn + \chi_{16}$} }
				& {\tiny \makecell[t]{
					$4\cdot \sgn + \std_5\otimes\sgn + \widetilde{\std_5}\otimes\sgn + \std_9 +$ \\ $ 3\cdot \std_9\otimes\sgn + \chi_{16}$
				}}
				\\ \hline
				$[(123456)]$ & {\tiny \makecell[t]{$1 + \std_5\otimes\sgn + \widetilde{\std_5} + \widetilde{\std_5}\otimes\sgn + $ \\ $ \wedge^2\std_5$}}
				& {\tiny \makecell[t]{
					$\sgn + \std_5 + \widetilde{\std_5}+\widetilde{\std_5}\otimes\sgn + \wedge^2\std_5\otimes\sgn $
				}}
				\\ \hline
				$[(123)(45)]$ & {\tiny \makecell[t]{$1 + \std_5 + \std_5\otimes\sgn + \widetilde{\std_5}\otimes\sgn +$ \\ $ \wedge^2\std_5\otimes\sgn $}}
				& {\tiny \makecell[t]{
					$\sgn + \std_5 + \std_5\otimes\sgn + \widetilde{\std_5} + \wedge^2\std_5 $
				}}
				\\ \hline
				$[(123)(456)]$ & {\tiny \makecell[t]{$1 + \sgn + 2\cdot \widetilde{\std_5} + 2\cdot\widetilde{\std_5} \otimes \sgn + \std_9 +$ \\ $ \std_9 \otimes \sgn + \wedge^2 \std_5 + \wedge^2\std_5 \otimes \sgn$ }}
				& {\tiny \makecell[t]{
					$\std_5 + \std_5\otimes\sgn + \std_9+\std_9\otimes\sgn + 2\cdot \chi_{16} $
				}}
				\\ \hline
				$[(12)(34)(56)]$ & {\tiny \makecell[t]{$1 + \std_5\otimes\sgn + 3\cdot \widetilde{\std_5} + 3\cdot \std_9 +$ \\ $ 2\cdot \wedge^2\std_5\otimes\sgn + 2\cdot \chi_{16}$} }
				& {\tiny \makecell[t]{
					$\sgn + \std_5 + 3\cdot \widetilde{\std_5}\otimes\sgn + 3\cdot \std_9\otimes\sgn +$ \\ $ 2\cdot \wedge^2\std_5 + 2\cdot \chi_{16} $
				}}
				\\ \hline
		\end{tabular}
		\begin{caption}{\label{tbl:basepoints-6} Degree optimal choices for the remaining classes in $S_6$.}
		\end{caption}
	\end{table}}
	
	In Tables~\ref{tbl:monomial-cone}--\ref{tbl:basepoints-6}, our convention for labeling characters is as follows.  The characters $1$ and $\mathrm{sgn}$ are the trivial and sign characters of $S_n$, respectively.  We let $\mathrm{std}_{n-1}$ be the character of the $(n-1)$-dimensional standard representation of $S_n$.  Additionally, we view $\mathrm{std}_2$ as a character of $S_4$ by means of its quotient to $S_3$, $\mathrm{std}_5$ as a character of $S_5$ by means of its inclusion as a transitive subgroup of $S_6$, and $\mathrm{std}_9$ as a character of $S_6$ by means of its inclusion as a transitive subgroup of $S_{10}$.  We also write $\widetilde{\mathrm{std}_5}$ for the composition of $\mathrm{std}_5$ with the outer automorphism of $S_6$.

	{\tiny \begin{table}[h]
		\begin{minipage}[t]{0.35\textwidth}
		\begin{tabular}{|l|l|c|c|}
			\hline $n$ & $C$ & $\alpha(S_n,C)$ & $\epsilon_\mathrm{eff}^\mathbb{Q}(C)$ \\ \hline
				$3$ & $[(12)]$ & $1/6$ & $0$ \\
					& $[(123)]$ & $1/4$ & $1/3$ \\ \hline
				$4$ & $[(12)]$ & $1/8$ & $0$ \\
					& $[(12)(34)]$ & $1/6$ & $1/12$ \\
					& $[(123)]$ & $1/16$ & $1/12$ \\
					& $[(1234)]$ & $1/9$ & $0$ \\
					\hline
				$5$ & $[(12)]$ & $13/120$ & $0$ \\
					& $[(12)(34)]$ & $1/15$ & $1/60$ \\
					& $[(123)]$ & $1/20$ & $1/60$ \\
					& $[(1234)]$ & $1/36$ & $0$ \\
					& $[(12345)]$ & $1/15$ & $1/30$ \\
					& $[(123)(45)]$ & $13/120$ & $1/30$ \\ \hline
		\end{tabular}
		\end{minipage}
		\begin{minipage}[t]{0.35\textwidth}
		\begin{tabular}{|l|l|c|c|}
			\hline $n$ & $C$ & $\alpha(S_n,C)$ & $\epsilon_\mathrm{eff}^\mathbb{Q}(C)$ \\ \hline
				$6$ & $[(12)]$ & $31/360$ & $0$ \\
					& $[(12)(34)]$ & $1/30$ & $1/360$ \\
					& $[(12)(34)(56)]$ & $31/360$ & $0$ \\
					& $[(123)]$ & $11/240$ & $1/360$ \\
					& $[(123)(456)]$ & $11/240$ & $1/360$ \\
					& $[(1234)]$ & $7/270$ & $0$ \\
					& $[(1234)(56)]$ & $2/135$ & $1/360$ \\
					& $[(12345)]$ & $13/240$ & $1/72$ \\
					& $[(123456)]$ & $1/54$ & $0$ \\
					& $[(123)(45)]$ & $1/54$ & $0$ \\
					\hline
		\end{tabular}
		\end{minipage}
		\begin{caption} {\label{tbl:small-Sn-intro} Linnik exponents in the sense of \eqref{eqn:LinnikExponent} for small symmetric groups, along with the range of effectivity in $\epsilon$.  See Theorem~\ref{thm:small-Sn-bowels}.}
		\end{caption}
	\end{table} }

\subsection{Alternative class functions}
	\label{subsec:alternative-Sn}
	
	The class functions we introduced in Section~\ref{sec:symmetric-groups} are efficient at detecting $\ell$-cycles inside $S_\ell$, and are likely the optimal choices of such functions expressible in terms of monomial characters.  We thus expect expect their products to be close to optimal for classes containing large cycles.  However, they are not optimal in the opposite extreme, say for the class of a transposition.  The following lemma introduces a general class of class functions which shows our methods at least suffice to give decaying exponents for such classes.
	
	\begin{lemma}\label{lem:alternate-Sn-choices}
		Let $n \geq 3$, and suppose that $g \in S_n$ may be expressed as a product of $d$ disjoint $\ell$-cycles for some $d \geq 1$ and $\ell \geq 2$.  Set $m := \lfloor \frac{n}{\ell}\rfloor$.  
		
		If $d <m$, then there are class functions $\Psi_+,\Psi_-$ of $S_n$ satisfying the hypotheses of Theorem~\ref{thm:general-approach} whose difference is supported on the class of $g$ and so that
			\[
				\Psi_+(1) = \Psi_-(1) \leq \frac{n!}{2\phi(\ell)^{d/2} \sqrt{m}}.
			\]
		Moreover, $\langle \Psi_+,\sgn\rangle = 1$ if $\sgn(g)=1$ and $\langle \Psi_+,\sgn\rangle = 0$ if $\sgn(g) = -1$.
		
		If $d=m$, then there is a different choice of $\Psi_+$ and $\Psi_-$ as above, except with
			\[
				\Psi_+(1) = \Psi_-(1) \leq \frac{n!}{2d^{\frac{\ell-1}{2}}}.
			\]
		As in the earlier case, $\langle \Psi_+,\sgn\rangle = 1$ if $\sgn(g)=1$ and $\langle \Psi_+,\sgn\rangle = 0$ if $\sgn(g) = -1$.
	\end{lemma}
	\begin{proof}
		Suppose first that $d \ell \leq n-\ell$, and set $m = \lfloor \frac{n}{\ell} \rfloor$.  Let $G = F_\ell \wr C_m$, where $F_\ell = C_\ell \rtimes \mathrm{Aut}(C_\ell)$, and note that we may regard $g$ naturally as an element of $G$.  By our assumption that $d \ell \leq n-\ell$ (or equivalently that $d < m$), we see that $g$ is not central since it does not commute with any nontrivial element of the natural section of $C_m$.  In particular, the centralizer of $g$ in $G$ is isomorphic to $C_\ell^d \times F_\ell^{m-d}$, which has index $m \cdot \phi(\ell)^d$ in $G$, and since $G$ is supersolvable, it is monomial, which yields the claim in this case on appealing to Lemma~\ref{lem:conditional-conductor}.
		
		Now suppose that $d=m$.  In this case, we instead consider $G = C_d \wr C_\ell = C_d^\ell \rtimes C_\ell$ which is also a monomial group.  By considering the natural permutation representation of $G$ of degree $d\ell$, we see that the generator of $C_\ell$ acts as a product of $d$ disjoint $\ell$-cycles, so in particular is conjugate to $g$.  Moreover, the centralizer subgroup of this element is $\Delta \times C_\ell$, where $\Delta < C_d^\ell$ is the diagonal subgroup.  It follows that the index of the centralizer in $G$ is $d^{\ell-1}$, which gives the claim.
	\end{proof}
	
	It is then possible to take products of these class functions, both with others provided by Lemma~\ref{lem:alternate-Sn-choices} and with those provided by Theorem~\ref{thm:product-of-cycles-characters}.
	
	\begin{lemma}\label{lem:product-of-characters}
		Let $n \geq 3$, and let $C$ be a non-identity conjugacy class in $S_n$.  Suppose there is some $m < n$ such that $g = g_1 g_2$, where $g_1 \in S_m$, $g_2 \in S_{n-m}$, with neither $g_1$ nor $g_2$ equal to the identity, and we view the product $g_1g_2$ as living inside $S_m \times S_{n-m} \leq S_n$.  Suppose there are characters $\Psi_{+,1}, \Psi_{-,1}$ of $S_m$ and $\Psi_{+,2}, \Psi_{-,2}$ of $S_{n-m}$ expressible as non-negative integral linear combinations of monomial characters such that $\Psi_{+,i}-\Psi_{-,i}$ is supported on the conjugacy class of $g_i$.  Then if we take $\Psi_{+}$ to be the induction to $S_n$ of $\Psi_{+,1} \times \Psi_{+,2} + \Psi_{-,1} \times \Psi_{-,2}$ and $\Psi_-$ to be the induction of $\Psi_{+,1} \times \Psi_{-,2} + \Psi_{-,1} \times \Psi_{+,2}$, then $\Psi_+$ and $\Psi_-$ are non-negatitve integral linear combinations of monomial characters, and the difference $\Psi_+ - \Psi_-$ is supported on $g$.
		
		Moreover, we have
			\[
				\frac{\Psi_+(1)}{n!} = 2 \cdot \frac{\Psi_{+,1}(1)}{m!} \cdot \frac{\Psi_{+,2}(1)}{(n-m)!}
			\]
		and $\langle \Psi_+,\sgn\rangle = \langle \Psi_{+,1},\sgn\rangle \langle \Psi_{+,2},\sgn\rangle + \langle \Psi_{-,1},\sgn\rangle \langle \Psi_{-,2},\sgn\rangle$.
	\end{lemma}
	\begin{proof}
		This is immediate once unpacked.
	\end{proof}
	
	We give an example of how this may be used to provide Linnik exponents for certain classes $C$ that are superior to those given by Theorem~\ref{thm:Sn-cycle-intro}, Theorem~\ref{thm:optimized-Sn-asymptotic}, and the zero-based \eqref{eqn:LinnikExponent-TZ}.
	
	\begin{example} \rm
		Assume that $n\geq 6$ is even and let $C$ be the conjugacy class of $S_n$ whose elements are a product of a cycle of length $\frac{n}{2}$ and a disjoint cycle of length $2$. Then, a direct implication of Theorem~\ref{thm:product-of-cycles-characters} is 
		\begin{align} \label{direct conseq}
			\alpha(S_n,C) \leq \frac{(n/2-2)2^{n/2-2}+1}{(n/2)!}.
		\end{align} 
		We may also view an element $g=g_1 g_2$ of $C$ as an element of $S_{n/2} \times S_{n/2} < S_n$, where $g_1$ is a cycle of length $\frac{n}{2}$ and $g_2$ is a transposition.  Following the notation in Lemma \ref{lem:product-of-characters}, let $\Psi_{+,1}$ and $\Psi_{-,1}$ be the characters produced by Theorem~\ref{thm:product-of-cycles-characters}, and let $\Psi_{+,2}$ and $\Psi_{-,2}$ be produced by Lemma~\ref{lem:alternate-Sn-choices} . Then $\Psi_{+,1}(1) \leq \left( \frac{n}{2}-2 \right)2^{\frac{n}{2}-2}+1$ by Theorem ~\ref{thm:product-of-cycles-characters}, and $\Psi_{+,2}(1) \leq \frac{\frac{n}{2}!}{2\sqrt{\lfloor n/4 \rfloor}}$ by Lemma \ref{lem:alternate-Sn-choices}. Hence, by Lemma \ref{lem:product-of-characters}, we have $\Psi_{+}(1)\leq \frac{n!\left[(n/2-2)2^{n/2-2}+1 \right]}{(n/2)!\sqrt{\lfloor n/4 \rfloor}}$, which implies that 
		\begin{align*}
			\alpha(S_n,C) \leq \frac{(n/2-2)2^{n/2-2}+1}{(n/2)!\sqrt{\lfloor n/4 \rfloor}},
		\end{align*}
		which improves over $(\ref{direct conseq})$ by a factor of $\sqrt{\lfloor n/4 \rfloor}$. 
	\end{example}
	
%%%%%%%%%
%%%%%%%%%
%%%%%%%%%
\section{Results for other groups}
	\label{sec:other-groups}
%%%%%%%%%
%%%%%%%%%
%%%%%%%%%

	In this section, we provide proofs of Corollary~\ref{cor:unipotent}, Theorem~\ref{thm:simple-examples}, Theorem~\ref{thm:trace-zero}, and Theorem~\ref{thm:semidirect-product}.
	
	\begin{proof}[Proof of Corollary~\ref{cor:unipotent}]
		Recall that, in the context of this corollary, $G = \mathrm{GL}_r(\mathbb{F}_q)$, that $C$ consists of the unipotent elements of $G$ with a single Jordan block, and that we wish to establish the admissibility of the Linnik exponent
			\[
				\alpha(G,C) = \frac{1}{2q^{(r-1)(r-2)/4}}
			\]
		using Theorem~\ref{thm:least-prime-rational-class-artin-intro}.
		
		First, note that $C$ does comprise a single rational equivalence class: all such elements are conjugate in $G$ and have order $p$ (the characteristic prime).  The latter point guarantees also that $g^i \in C$ for each $g \in C$ and $i$ relatively prime to $p$, establishing that $C$ comprises a rational equivalence class.
		
		Now let $H$ be a maximal unipotent subgroup of $G$; explicitly, we take $H$ to consist of the upper triangular matrices with all diagonal entries equal to $1$.  This group has order $q^{\frac{r(r-1)}{2}}$, and in particular is a $p$-group.  Thus, the Artin holomorphy conjecture is known for $H$.  Moreover, $H \cap C$ is nontrivial, since for example it contains the element $g$ with $1$'s on the diagonal and leading subdiagonal, and $0$'s everywhere else.  It is an exercise to check that $C_H(g)$ consists of those matrices in $H$ with each subdiagonal having a constant entry.  There are $q^{r-1}$ such matrices, so Theorem~\ref{thm:least-prime-rational-class-artin-intro} gives
			\[
				\alpha(G,C)
					= \frac{|C_H(g)|^{1/2}}{2|H|^{1/2}}
					= \frac{1}{2q^{(r-1)(r-2)/4}}
			\]
		as claimed.
	\end{proof}
	
	We next prove Theorem~\ref{thm:trace-zero} concerning elements in $\mathrm{GL}_2(\mathbb{F}_q)$ ($q$ an odd prime power) with trace $0$.
	
	\begin{proof}[Proof of Theorem~\ref{thm:trace-zero}]
		Let $G = \mathrm{GL}_2(\mathbb{F}_q)$ and let $B \leq G$ be the subgroup of upper triangular matrices.  Letting $C \subset G$ consist of the matrices with trace $0$, we see that
			\[
				C \cap B = \left\{
				\begin{pmatrix}
					a & * \\ 0 & -a
				\end{pmatrix}
					: a \in \mathbb{F}_q^\times \right\}.
			\]
		The abelian characters of $B$ factor through the quotient to the diagonal subgroup, so are of the form $(\chi_1,\chi_2)$ for characters $\chi_1,\chi_2 \colon \mathbb{F}_q^\times \to \mathbb{C}^\times$; explicitly, we have 
			\[
				(\chi_1,\chi_2)\left(\begin{pmatrix} a & b \\ 0 & d \end{pmatrix}\right)
					= \chi_1(a) \chi_2(d).
			\]
		It is then straightforward to check using orthogonality that
			\[
				\sum_{\chi\colon \mathbb{F}_q^\times \to \mathbb{C}^\times} \chi(-1) \cdot (\chi, \overline{\chi})
					= (q-1) \mathbf{1}_{C \cap B},
			\]
		where $\mathbf{1}_{C \cap B}$ is the indicator function of $C \cap B$.  We thus take
			\[
				\Psi_+ = \sum_{\substack{ \chi\colon \mathbb{F}_q^\times \to \mathbb{C}^\times \\ \chi(-1) = 1}} \mathrm{Ind}_B^G (\chi,\overline{\chi}) 
					\quad \text{and} \quad
				\Psi_- = \sum_{\substack{ \chi\colon \mathbb{F}_q^\times \to \mathbb{C}^\times \\ \chi(-1) = -1}} \mathrm{Ind}_B^G (\chi,\overline{\chi}).
			\]
		We see that $\Psi_+(1) = \Psi_-(1) = \frac{(q-1)(q+1)}{2}$ since $[G:B] = q+1$.  We also observe that $G$ admits a unique quadratic character $\psi$, and that this character $\psi$ factors through the determinant map.  As a result, we see that $\langle \Psi_+,\psi \rangle = 1$ if $-1$ is a square in $\mathbb{F}_q$ and $\langle \Psi_+,\psi\rangle = 0$ if not.
		
		It remains to compute the resulting Linnik exponent.  (Note that the claimed exponent is better by a factor of about $1/2$ than what would be given just by appealing to the degree bound on the conductors of $\Psi_+$ and $\Psi_-$.)  Let $C_0 \subset G$ consist of those elements conjugate to elements of $C \cap B$, and let $K/k$ be a $G$-extension.  Since the antidiagonal characters $(\chi,\overline{\chi})$ form a subgroup of the abelian characters of $B$, they cut out a cyclic degree $q-1$ extension $F/K^B$.  As a result, we have $L(s,\Psi_+)L(s,\Psi_-) = \zeta_F(s)$, so 
			\[
				q(\Psi_+ + \Psi_-)
					= |\mathrm{Disc}(F)|
					\leq |\mathrm{Disc}(K)|^{\frac{1}{q^2-q}},
			\]
		as follows from observing that $[K:F] = \frac{|B|}{q-1} = q^2-q$.  Now, let $\Delta = \Psi_+ - \Psi_-$ and let $\mathfrak{p}$ be a prime of $k$ that is tamely ramified in $K$, and write $I = \langle \sigma \rangle$ for the associated inertia subgroup.  Let $\mu \in \mathbb{F}_{q^2}^\times$ be the ratio of the eigenvalues of $\sigma$, and observe that $\sigma^i$ has trace $0$ if and only if $\mu^i = -1$.  In particular, at most half of the elements of $I$ can lie in $C_0$.  Moreover, for any $g \in C_0$, we have either $\Delta(g) =  2(q-1)$ or $\Delta(g) = q-1$ according to whether or not $g$ is conjugate to a diagonal element of $C \cap B$ (or, more pertinently, whether $g$ fixes $2$ or $1$ line(s) in $\mathbb{F}_q^2$).  All told, we see from \eqref{eqn:local-artin-conductor} that
			\[
				0 
					\geq v_\mathfrak{p}(\mathfrak{f}_\Delta)
					\geq - \frac{1}{|I|} \cdot \frac{|I|}{2} \cdot 2(q-1)
					= -(q-1).					
			\]
		Since the least power of a tamely ramified prime dividing the discriminant of $K$ is $|G|/2$, we conclude that
			\[
				q(\Psi_- - \Psi_+)
					=1/q(\Delta)
					\ll_{G,[k:\mathbb{Q}]} |\mathrm{Disc}(K)|^{\frac{2(q-1)}{|G|}}
					= |\mathrm{Disc}(K)|^{\frac{2}{q^3-q}},
			\]
		and hence that
			\[
				q(\Psi_+) 
					\leq q(\Psi_-)
					\ll_{G,[k:\mathbb{Q}]} |\mathrm{Disc}(K)|^{\frac{1}{2(q^2-q)} + \frac{1}{q^3-q}}.
			\]
		This completes the proof on appealing to Corollary~\ref{cor:general-approach}.
	\end{proof}
	
	We now turn to the proof of Theorem~\ref{thm:semidirect-product} concerning semidirect products.
	
	\begin{proof}[Proof of Theoerm~\ref{thm:semidirect-product}]
		As in the statement of the theorem, let $N = \mathbb{F}_p^r$ for some prime $p$ and some $r \geq 1$.  Let $G_0 \leq \mathrm{GL}_r(\mathbb{F}_p)$ act irreducibly on $N$, and let $G = N \rtimes G_0$.  Let $C = N \setminus \{0\} \subset G$.  We first claim that $C$ comprises a single rational equivalence class: viewing $N$ as a vector space, by the assumption that $G_0$ acts irreducibly on $N$, we see that conjugation by $G_0$ is sufficient to take any element of $C$ onto the line containing any other.  But each non-zero element of a line generates the same cyclic subgroup, so this guarantees rational equivalence.
		
		We now take $\Psi_+ = \mathrm{Ind}_N^G 1$ and $\Psi_- = \mathrm{Ind}_N^G \chi$ where $\chi$ is any cyclic degree $p$ character of $N$.  Since $N$ is normal, these characters are supported on it, and they may differ only at $N \setminus \{0\} = C$.  (And evidently they do, since $\langle \Psi_-,1\rangle = 0$ by Frobenius reciprocity.)  Now letting $K/k$ be a $G$-extension, we have that $L(s,\Psi_+) = \zeta_{K^N}(s)$.  Additionally, for the same reason that $N\setminus \{0\}$ comprises a single rational class, we see that $L(s,\Psi_-) = L(s,\chi)$ where $\chi$ is a cyclic degree $p$ character over $K^N$, and that the conductors of all such characters are the same.  As a result, we find $q(\Psi_+) \leq q(\Psi_-)$ and
			\[
				q(\Psi_+) q(\Psi_-)^{p^r-1} = |\mathrm{Disc}(K)|.
			\]
		This implies that $q(\Psi_-) \leq |\mathrm{Disc}(K)|^{\frac{1}{p^r-1}}$, which yields the theorem on via  Corollary~\ref{cor:general-approach}.
	\end{proof}
	
	Finally, we return to the proof of Theorem~\ref{thm:simple-examples}, which we provide in a more refined form.
	
	\begin{theorem} \label{thm:better-simple-examples}
		Let $G$ be a finite simple group, isomorphic to one of the Mathieu groups $M_{11}$, $M_{12}$, $M_{22}$, $M_{23}$, and $M_{24}$.  If $C \subset G$ is a non-identity rational class, then the Linnik exponent $\alpha(G,C)$ provided in Table~\ref{tbl:simple-exponents-Mathieu} is admissible in \eqref{eqn:LinnikExponent}.  If instead $G$ is isomorphic to one of the finite simple groups $\mathrm{PSL}_2(\mathbb{F}_q)$ with $7 \leq q \leq 23$, then the Linnik exponent $\alpha(G,C)$ provided in Table~\ref{tbl:simple-exponents-PSL2} is admissible in \eqref{eqn:LinnikExponent}.  If $k$ does not admit a quadratic subfield, then the implied constant is effectively computatable for all $\epsilon>0$, while if $k$ does admit a quadratic subfield, it is admissible for $\epsilon > \frac{2}{|G|[k:\mathbb{Q}]}$.
	\end{theorem}
	\begin{proof}
		This follows from a computation analogous to Theorem~\ref{thm:small-Sn-bowels}.  The code is available on the Github repository associated with this project \cite{github}.
	\end{proof}

\clearpage
\appendix

\section{Linnik exponents for symmetric groups $S_n$ with $7 \leq n \leq 10$}
	\label{sec:appendix-Sn}

	{\tiny
	\begin{table}[h]
		\begin{minipage}[t]{0.45\textwidth}
		\begin{tabular}{|l|l|c|c|} \hline
			$n$ & $C$ & $\alpha(S_n,C)$ & $\epsilon_\mathrm{eff}^\mathbb{Q}(C)$ \\ \hline
				$7$ & $[(12)]$ & $17/252$ & $0$ \\
					& $[(12)(34)]$ & $29/1260$ & $1/2520$ \\
					& $[(12)(34)(56)]$ & $5/144$ & $1/2520$ \\
					& $[(123)]$ & $143/3360$ & $1/1260$ \\
					& $[(123)(456)]$ & $113/5040$ & $1/630$ \\
					& $[(1234)]$ & $113/5040$ & $1/1260$ \\
					& $[(1234)(56)]$ & $1/72$ & $1/2520$ \\
					& $[(12345)]$ & $131/5040$ & $1/630$ \\
					& $[(123456)]$ & $13/720$ & $0$ \\
					& $[(123)(45)]$ & $1/72$ & $0$ \\
					& $[(123)(45)(67)]$ & $13/315$ & $1/840$ \\
					& $[(1234567)]$ & $2/105$ & $1/420$ \\
					& $[(12345)(67)]$ & $59/1680$ & $1/840$ \\
					& $[(1234)(567)]$ &$181/5040$ & $1/2520$ \\
					\hline
				$8$ & $[(12)]$ & $331/5760$ & $1/3360$ \\
					& $[(12)(34)]$ & $137/2730$ & $1/10080$ \\
					& $[(12)(34)(56)]$ & $89/6720$ & $1/20160$ \\
					& $[(12)(34)(56)(78)]$ & $131/5040$ & $1/20160$ \\
					& $[(123)]$ & $1129/40320$ & $1/4032$ \\
					& $[(123)(456)]$ & $89/13440$ & $1/20160$ \\
					& $[(1234)]$ & $5/288$ & $1/6720$ \\
					& $[(1234)(56)]$ & $43/5040$& $1/10080$ \\
					& $[(1234)(56)(78)]$ & $11/1260$ & $0$ \\
					& $[(1234)(5678)]$ & $19/1890$ & $1/20160$ \\
					& $[(12345)]$ & $587/40320$ & $1/6720$ \\
					& $[(123456)]$ & $37/4480$ & $1/5040$ \\
					& $[(123456)(78)]$ & $1/384$ & $1/20160$ \\
					& $[(123)(45)]$ & $53/4480$ & $1/6720$ \\
					& $[(123)(45)(67)]$ & $187/20160$ & $1/6720$ \\
					& $[(123)(456)(78)]$ & $187/13440$ & $1/3360$ \\
					& $[(1234567)]$ & $373/40320$ & $0$ \\
					& $[(12345678)]$ & $1/180$ & $1/3360$ \\
					& $[(12345)(67)]$ & $5/504$ & $1/6720$ \\
					& $[(1234)(567)]$ & $113/13440$ & $1/6720$ \\
					& $[(12345)(678)]$ & $527/40320$ & $1/3360$ \\
					\hline
		\end{tabular}
		\end{minipage}
		\begin{minipage}[t]{0.45\textwidth}
		\begin{tabular}{|l|l|c|c|} \hline
			$n$ & $C$ & $\alpha(S_n,C)$ & $\epsilon_\mathrm{eff}^\mathbb{Q}(C)$ \\ \hline
				$9$ & $[(12)]$ & $ 17819/362880$ & $11/181440$ \\
					& $[(12)(34)]$ & $1349/90720 $ & $1/30240 $\\
					& $[(12)(34)(56)]$ & $3187/362880 $ & $1/90720 $\\
					& $[(12)(34)(56)(78)]$ & $479/45360 $ & $1/20160 $\\
					& $[(123)]$ & $ 2971/120960$ & $ 1/60480$\\
					& $[(123)(456)]$ & $ 2113/362880 $ & $ 1/22680$\\
					& $[(123)(456)(789)]$ & $107/12960 $ & $ 1/20160$\\
					& $[(1234)]$ & $ 2449/181440$ & $1/25920 $\\
					& $[(1234)(56)]$ & $ 373/72576$ & $1/45360 $\\
					& $[(1234)(56)(78)]$ & $ 211/68040$ & $0$\\
					& $[(1234)(5678)]$ & $ 929/181440 $ & $1/45360 $\\
					& $[(12345)]$ & $ 859/90720$ & $1/20160 $\\
					& $[(123456)]$ & $ 149/30240$ & $1/36288 $\\
					& $[(123456)(78)]$ & $ 157/90720$ & $1/30240 $\\
					& $[(123456)(789)]$ & $ 1/648$ & $ 0$\\
					& $[(123)(45)]$ & $  461/60480$ & $ 1/60480$\\
					& $[(123)(45)(67)]$ & $241/45360 $ & $1/36288 $\\
					& $[(123)(45)(67)(89)]$ & $67/5670 $ & $ 1/20160$\\
					& $[(123)(456)(78)]$ & $ 943/181440$ & $1/36288 $\\
					& $[(1234567)]$ & $ 5/1008$ & $ 1/45360$\\
					& $[(12345678)]$ & $ 67/18144$ & $1/60480 $ \\
					& $[(123456789)]$ & $ 4/2835$ & $1/22680 $ \\
					& $[(12345)(67)]$ & $ 373/90720$ & $1/181440 $ \\
					& $[(12345)(67)(89)]$ & $3067/362880 $ & $ 1/25920$\\
					& $[(1234)(567)]$ & $ 389/90720$ & $1/45360 $ \\
					& $[(1234)(567)(89)]$ & $ 377/72576$ & $1/60480 $ \\
					& $[(1234567)(89)]$ & $ 607/90720$ & $11/181440 $ \\
					& $[(12345)(678)]$ & $ 41/8064$ & $1/25920 $ \\
					& $[(12345)(6789)]$ & $ 85/12096$ & $1/36288 $ \\
					\hline
		\end{tabular}
		\end{minipage}
		\begin{caption}{\label{tbl:small-Sn-appendix} Linnik exponents in the sense of \eqref{eqn:LinnikExponent} for non-identity classes in $S_n$, $7 \leq n \leq 9$, along with the effective range of $\epsilon$.  See Theorem~\ref{thm:small-Sn-intro}.}\end{caption}
	\end{table}}
	
	{\tiny
	\begin{table}[h]
		\begin{minipage}[t]{0.5\textwidth}
		\begin{tabular}{|l|l|c|c|} \hline
			$n$ & $C$ & $\alpha(S_n,C)$ & $\epsilon_\mathrm{eff}^\mathbb{Q}(C)$ \\ \hline
			$10$& $[(1 2)]$ & $78941/1814400$ & $1/181440$ \\
				& $[(1 2)(3 4)]$ & $20681/1814400$ & $1/151200$ \\
				& $[(1 2)(3 4)(5 6)(7 8)(9 \,10)]$ & $2123/201600$ & $19/1814400$ \\
				& $[(1 2)(3 4)(5 6)]$ & $541/103680$ & $1/362880$ \\
				& $[(1 2)(3 4)(5 6)(7 8)]$ & $337/57600$ & $11/1814400$ \\
				& $[(1 2 3)]$ & $13753/604800$ & $1/201600$ \\
				& $[(1 2 3)(4 5 6)]$ & $187/57600$ & $1/181440$ \\
				& $[(1 2 3)(4 5 6)(7 8 9)]$ & $1111/259200$ & $1/453600$ \\
				& $[(1 2 3 4)]$ & $419/40320$ & $1/129600$ \\
				& $[(1 2 3 4)(5 6)(7 8)(9 \,10)]$ & $329/129600$ & $1/302400$ \\
				& $[(1 2 3 4)(5 6)]$ & $859/259200$ & $1/453600$ \\
				& $[(1 2 3 4)(5 6 7 8)(9 \,10)]$ & $2209/907200$ & $1/604800$ \\
				& $[(1 2 3 4)(5 6 7 8)]$ & $503/259200$ & $1/453600$ \\
				& $[(1 2 3 4)(5 6)(7 8)]$ & $7979/3628800$ & $1/604800$ \\
				& $[(1 2 3 4 5)]$ & $21283/3628800$ & $1/201600$ \\
				& $[(1 2 3 4 5)(6 7 8 9 \,10)]$ & $239/113400$ & $1/259200$ \\
				& $[(1 2 3)(4 5)]$ & $257/45360$ & $1/201600$ \\
				& $[(1 2 3 4 5 6)]$ & $11509/3628800$ & $1/259200$ \\
				& $[(1 2 3)(4 5 6)(7 8)(9 \,10)]$ & $14179/3628800$ & $1/86400$ \\
				& $[(1 2 3)(4 5)(6 7)(8 9)]$ & $5753/1814400$ & $1/201600$ \\
				& $[(1 2 3)(4 5)(6 7)]$ & $4457/1814400$ & $1/362880$ \\
			\hline
		\end{tabular}
		\end{minipage}
		\begin{minipage}[t]{0.49\textwidth}
		\begin{tabular}{|l|l|c|c|} \hline
			$n$ & $C$ & $\alpha(S_n,C)$ & $\epsilon_\mathrm{eff}^\mathbb{Q}(C)$ \\ \hline
			$10$& $[(1 2 3)(4 5 6)(7 8)]$ & $659/403200$ & $0$ \\
				& $[(1 2 3 4 5 6)(7 8)(9 \,10)]$ & $1081/604800$ & $1/362880$ \\
				& $[(1 2 3 4 5 6)(7 8)]$ & $361/317520$ & $1/1814400$ \\
				& $[(1 2 3 4 5 6)(7 8 9)]$ & $631/453600$ & $1/907200$ \\
				& $[(1 2 3 4 5 6 7)]$ & $4237/1814400$ & $1/453600$ \\
				& $[(1 2 3 4 5 6 7 8)(9 10)]$ & $1/800$ & $1/1814400$ \\
				& $[(1 2 3 4 5 6 7 8)]$ & $2137/1209600$ & $1/453600$ \\
				& $[(1 2 3 4 5 6 7 8 9)]$ & $4993/3628800$ & $1/226800$ \\
				& $[(1 2 3 4 5)(6 7)]$ & $773/302400$ & $1/453600$ \\
				& $[(1 2 3 4 5)(6 7)(8 9)]$ & $2063/907200$ & $1/362880$ \\
				& $[(1 2 3 4 5 6 7 8 9 \,10)]$ & $1/3150$ & $1/226800$ \\
				& $[(1 2 3 4)(5 6 7)(8 9 \,10)]$ & $449/120960$ & $1/113400$ \\
				& $[(1 2 3 4)(5 6 7)]$ & $631/302400$ & $1/302400$ \\
				& $[(1 2 3 4 5 6)(7 8 9 \,10)]$ & $823/403200$ & $1/302400$ \\
				& $[(1 2 3 4)(5 6 7)(8 9)]$ & $47/33600$ & $1/453600$ \\
				& $[(1 2 3 4 5 6 7)(8 9)]$ & $3151/1814400$ & $1/302400$ \\
				& $[(1 2 3 4 5)(6 7 8)]$ & $6547/3628800$ & $1/362880$ \\
				& $[(1 2 3 4 5)(6 7 8 9)]$ & $83/45360$ & $1/604800$ \\
				& $[(1 2 3 4 5 6 7)(8 9 \,10)]$ & $491/259200$ & $1/201600$ \\
				& $[(1 2 3 4 5)(6 7 8)(9 \,10)]$ & $9481/3628800$ & $1/181440$ \\
			\hline
		\end{tabular}
		\end{minipage}
		
		\begin{caption}{\label{tbl:S10-appendix} Linnik exponents in the sense of \eqref{eqn:LinnikExponent} for non-identity classes in $S_{10}$.}\end{caption}
	\end{table}
	}

\clearpage

\section{Linnik exponents for some simple groups}

	{\tiny
	\begin{table}[h]
		\begin{minipage}[t]{0.49\textwidth}
		\begin{tabular}{|l|l|c|c|} \hline
			$G$ & Cycle type & $\#$ Conj. & $\alpha(G,C)$ \\ \hline
			$M_{11}$ 
				& $(2)(2)(2)(2)$ & $1$ & $1/45$ \\
				& $(3)(3)(3)$ & $1$ & $37/880$ \\
				& $(4)(4)$ & $1$ & $5/1188$ \\
				& $(5)(5)$ & $1$ & $29/3168$ \\
				& $(6)(3)(2)$ & $1$ & $7/495$ \\
				& $(8)(2)$ & $2$ & $29/6930$ \\
				& $(11)$ & $2$ & ${131/2880}$ \\ \hline
			$M_{12}$
				& $(2)(2)(2)(2)$ & $1$ & $1/45$ \\
				& $(3)(3)(3)$ & $1$ & $37/880$ \\
				& $(4)(4)(2)(2)$ & $1$ & $5/1188$ \\
				& $(5)(5)$ & $1$ & $29/3168$ \\
				& $(6)(3)(2)$ & $1$ & $7/495$ \\
				& $(8)(4)$ & $2$ & $29/6930$ \\
				& $(11)$ & $2$ & ${131/2880}$ \\ \hline
			$M_{22}$
				& $(2)(2)(2)(2)(2)(2)(2)(2)$ & $1$ & $73/6160$ \\
				& $(3)(3)(3)(3)(3)(3)$ & $1$ & $129/49280$ \\
				& $(4)(4)(4)(4)(2)(2)a$ & $1$ & $335/66528$ \\
				& $(4)(4)(4)(4)(2)(2)b$ & $1$ & $13/2376$ \\
				& $(5)(5)(5)(5)$ & $1$ & $1/384$ \\
				& $(6)(6)(3)(3)(2)(2)$ & $1$ & $11/3840$ \\
				& $(7)(7)(7)$ & $2$ & $2263/253440$ \\
				& $(8)(8)(4)(2)$ & $1$ & $391/77616$ \\
				& $(11)(11)$ & $2$ & ${7331/161280}$ \\ \hline
		\end{tabular}
		\end{minipage}
		\begin{minipage}[t]{0.49\textwidth}
		\begin{tabular}{|l|l|c|c|} \hline
			$G$ & Cycle type & $\#$ Conj. & $\alpha(G,C)$ \\ \hline
			$M_{23}$ 
				& $(2)(2)(2)(2)(2)(2)(2)(2)$ & $1$ & $2059/182160$ \\
				& $(3)(3)(3)(3)(3)(3)$ & $1$ & $13/6720$ \\
				& $(4)(4)(4)(4)(2)(2)$ & $1$ & $127/36432$ \\
				& $(5)(5)(5)(5)$& $1$ & $1/1792$ \\
				& $(6)(6)(3)(3)(2)(2)$ & $1$ & $11/11520$ \\
				& $(7)(7)(7)$ & $2$ & $48287/5829120$ \\
				& $(8)(8)(4)(2)$ & $1$ & $133/36432$ \\
				& $(11)(11)$ & $2$ & $7331/3709440$ \\
				& $(14)(7)(2)$ & $2$ & $629/80640$ \\
				& $(15)(5)(3)$ & $2$ & $29/11520$ \\
				& $(23)$ & $2$ & $38567/1774080$ \\ \hline
			$M_{24}$
				& $(2)(2)(2)(2)(2)(2)(2)(2)(2)(2)(2)(2)$ & $1$ & $250049/61205760$ \\
				& $(2)(2)(2)(2)(2)(2)(2)(2)$ & $1$ & $163/28336$ \\
				& $(3)(3)(3)(3)(3)(3)(3)(3)$ & $1$ & $5003/3400320$ \\
				& $(3)(3)(3)(3)(3)(3)$ & $1$ & $39643/36723456$ \\
				& $(4)(4)(4)(4)(4)(4)$ & $1$ & $ 3175/3060288$ \\
				& $(4)(4)(4)(4)(2)(2)(2)(2)$ & $1$ & $ 22555/18361728$ \\
				& $(4)(4)(4)(4)(2)(2)$ & $1$ & $ 7547/7650720$ \\
				& $(5)(5)(5)(5)$ & $1$ & $2627/6800640$ \\
				& $(6)(6)(6)(6)$ & $1$ & $1075/2623104$ \\
				& $(6)(6)(3)(3)(2)(2)$ & $1$ & $9011/30602880$ \\
				& $(7)(7)(7)$ & $2$ & $51277/46632960$ \\
				& $(8)(8)(4)(2)$ & $1$ & $857/1064448$ \\
				& $(10)(10)(2)(2)$ & $1$ & $25/61824$ \\
				& $(11)(11)$ & $1$ & $95/96768$ \\
				& $(12)(12)$ & $1$ & $1709/3400320$ \\
				& $(12)(6)(4)(2)$ & $1$ & $9619/30602880$ \\
				& $(14)(7)(2)$ & $2$ & $233/645120$ \\
				& $(15)(5)(3)$ & $2$ & $1187/2128896$ \\
				& $(21)(3)$ & $2$ & $233/380160$ \\
				& $(23)$ & $2$ & $925607/42577920$ \\ \hline
		\end{tabular}
		\end{minipage}
		\begin{caption}{\label{tbl:simple-exponents-Mathieu} Linnik exponents in the sense of \eqref{eqn:LinnikExponent} for the Mathieu groups $M_n$.  The cycle type of the class in the degree $n$ permutation representation, and typically determines the rational equivalence class.  In the event that it doesn't, we add the label $a$, $b$, etc.  The column ``$\#$ Conj.'' indicates the number of conjugacy classes comprising the rational class.}\end{caption}
	\end{table}
	}
	
	{\tiny
	\begin{table}[h]
		\begin{minipage}[t]{0.45\textwidth}
		\begin{tabular}{|l|l|c|c|} \hline
			$G$ & Cycle type & $\#$ Conj. & $\alpha(G,C)$ \\ \hline
			$\mathrm{PSL}_2(\mathbb{F}_7)$ 
				& $(2)(2)(2)(2)$ & $1$ & $1/14$ \\
				& $(3)(3)$ & $1$ & $3/112$ \\
				& $(4)(4)$ & $1$ & $1/14$ \\
				& $(7)$ & $2$ & $7/96$ \\ \hline
			$\mathrm{PSL}_2(\mathbb{F}_8)$
				& $(2)(2)(2)(2)$ & $1$ & $2/63$ \\
				& $(3)(3)(3)$ & $1$ & $19/112$ \\
				& $(7)$ & $3$ & $1/108$ \\
				& $(9)$ & $3$ & $25/448$ \\ \hline
			$\mathrm{PSL}_2(\mathbb{F}_9)$
				& $(2)(2)(2)(2)$ & $1$ & $1/30$ \\
				& $(3)(3)(3)a$ & $1$ & $11/240$ \\
				& $(3)(3)(3)b$ & $1$ & $11/240$ \\
				& $(4)(4)$ & $1$ & $2/135$ \\
				& $(5)(5)$ & $2$ & $29/288$ \\ \hline
			$\mathrm{PSL}_2(\mathbb{F}_{11})$
				& $(2)(2)(2)(2)(2)(2)$ & $1$ & $1/22$ \\
				& $(3)(3)(3)(3)$ & $1$ & $19/440$ \\
				& $(5)(5)$ & $2$ & $5/528$ \\
				& $(6)(6)$ & $1$ & $7/55$ \\
				& $(11)$ & $2$ & $11/240$ \\ \hline
			$\mathrm{PSL}_2(\mathbb{F}_{13})$
				& $(2)(2)(2)(2)(2)(2)$ & $1$ & $2/91$ \\
				& $(3)(3)(3)(3)$ & $1$ & $5/364$ \\
				& $(6)(6)$ & $1$ & $6/455$ \\
				& $(7)(7)$ & $3$ & $67/936$ \\
				& $(13)$ & $2$ & $13/336$ \\ \hline
		\end{tabular}
		\end{minipage}
		\begin{minipage}[t]{0.45\textwidth}
		\begin{tabular}{|l|l|c|c|} \hline
			$G$ & Cycle type & $\#$ Conj. & $\alpha(G,C)$ \\ \hline
			$\mathrm{PSL}_{2}(\mathbb{F}_{16})$ 
				& $(2)(2)(2)(2)(2)(2)(2)(2)$ & $1$ & $4/255$ \\
				& $(3)(3)(3)(3)(3)$ & $1$ & $3/272$ \\
				& $(5)(5)(5)$ & $2$ & $7/1088$ \\
				& $(15)$ & $4$ & $1/238$ \\
				& $(17)$ & $8$ & $113/3840$ \\ \hline
			$\mathrm{PSL}_2(\mathbb{F}_{17})$
				& $(2)(2)(2)(2)(2)(2)(2)(2)$ & $1$ & $5/306$ \\
				& $(3)(3)(3)(3)(3)(3)$ & $1$ & $23/544$ \\
				& $(4)(4)(4)(4)$ & $1$ & $7/918$ \\
				& $(8)(8)$ & $2$ & $4/1071$ \\
				& $(9)(9)$ & $3$ & $121/2176$ \\
				& $(17)$ & $2$ & $17/576$ \\ \hline
			$\mathrm{PSL}_2(\mathbb{F}_{19})$
				& $(2)(2)(2)(2)(2)(2)(2)(2)(2)(2)$ & $1$ & $1/38$ \\
				& $(3)(3)(3)(3)(3)(3)$ & $1$ & $7/760$ \\
				& $(5)(5)(5)(5)$ & $2$ & $137/1368$ \\
				& $(9)(9)$ & $3$ & $9/3040$ \\
				& $(10)(10)$ & $2$ & $154/1539$ \\
				& $(19)$ & $2$ & $19/720$ \\ \hline
			$\mathrm{PSL}_2(\mathbb{F}_{23})$ 
				& $(2)(2)(2)(2)(2)(2)(2)(2)(2)(2)(2)(2)$ & $1$ & $1/46$ \\
				& $(3)(3)(3)(3)(3)(3)(3)(3)$ & $1$ & $85/2024$ \\
				& $(4)(4)(4)(4)(4)(4)$ & $1$ & $127/1518$ \\
				& $(6)(6)(6)(6)$ & $1$ & $127/1012$ \\
				& $(11)(11)$ & $5$ & $11/5520$ \\
				& $(12)(12)$ & $2$ & $232/2783$ \\
				& $(23)$ & $2$ & $23/1056$ \\ \hline
		\end{tabular}
		\end{minipage}
		\begin{caption}{\label{tbl:simple-exponents-PSL2} Linnik exponents in the sense of \eqref{eqn:LinnikExponent} for linear groups $\mathrm{PSL}_2(\mathbb{F}_q)$.  The cycle type of the class in the degree $q+1$ permutation representation, and typically determines the rational equivalence class.  In the event that it doesn't, we add the label $a$, $b$, etc.  The column ``$\#$ Conj.'' indicates the number of conjugacy classes comprising the rational class.}\end{caption}
	\end{table}
	}
	
\clearpage

\bibliographystyle{alpha}
\bibliography{references}

\begin{thebibliography}{LOTZ24}

\bibitem[BCP97]{Magma}
W.~Bosma, J.~Cannon, and C.~Playoust.
\newblock The {M}agma algebra system. {I}. {T}he user language.
\newblock {\em J. Symbolic Comput.}, 24(3-4):235--265, 1997.
\newblock Computational algebra and number theory (London, 1993).

\bibitem[Bel16]{Bellaiche-Chebotarev}
J.~Bella\"iche.
\newblock Th\'eor\`eme de {C}hebotarev et complexit\'e{} de {L}ittlewood.
\newblock {\em Ann. Sci. \'Ec. Norm. Sup\'er. (4)}, 49(3):579--632, 2016.

\bibitem[Bru06]{Brumley-EffectiveMultiplicityOne}
F.~Brumley.
\newblock Effective multiplicity one on {${\rm GL}_N$} and narrow zero-free
  regions for {R}ankin-{S}elberg {$L$}-functions.
\newblock {\em Amer. J. Math.}, 128(6):1455--1474, 2006.

\bibitem[BS96]{BachSorenson-ResidueClasses}
E.~Bach and J.~Sorenson.
\newblock Explicit bounds for primes in residue classes.
\newblock {\em Math. Comp.}, 65(216):1717--1735, 1996.

\bibitem[CK21]{ChoKim-SignChange}
P.~J. Cho and H.~H. Kim.
\newblock The smallest prime in a conjugacy class and the first sign change for
  automorphic {$L$}-functions.
\newblock {\em Proc. Amer. Math. Soc.}, 149(3):923--933, 2021.

\bibitem[CLOZ]{github}
P.~J. Cho, R.~J. Lemke~Oliver, and A.~Zaman.
\newblock Github repository for this paper.
\newblock
  \url{https://github.com/lemkeoliver/lemkeoliver.github.io/tree/main/docs/code/leastprimecycle}.

\bibitem[CLOZ25]{CLOZ}
P.~J. Cho, R.~J. Lemke~Oliver, and A.~Zaman.
\newblock {Effective Brauer-Siegel theorems for Artin $L$-functions}, 2025.

\bibitem[Ell71]{Elliott-LeastPrimePowerResidue}
P.~D. T.~A. Elliott.
\newblock The least prime {$k-{\rm th}$}-power residue.
\newblock {\em J. London Math. Soc. (2)}, 3:205--210, 1971.

\bibitem[FH91]{FultonHarris}
W.~Fulton and J.~Harris.
\newblock {\em Representation theory}, volume 129 of {\em Graduate Texts in
  Mathematics}.
\newblock Springer-Verlag, New York, 1991.
\newblock A first course, Readings in Mathematics.

\bibitem[FJ24]{FiorilliJouve}
D.~Fiorilli and F.~Jouve.
\newblock Distribution of {F}robenius elements in families of {G}alois
  extensions.
\newblock {\em J. Inst. Math. Jussieu}, 23(3):1169--1258, 2024.

\bibitem[GM19]{GrenieMolteni-ExplicitCDTGRH}
L.~Greni\'e and G.~Molteni.
\newblock An explicit {C}hebotarev density theorem under {GRH}.
\newblock {\em J. Number Theory}, 200:441--485, 2019.

\bibitem[GMP19]{GeMilinovichPollack-Split}
Z.~Ge, M.~B. Milinovich, and P.~Pollack.
\newblock A note on the least prime that splits completely in a nonabelian
  {G}alois number field.
\newblock {\em Math. Z.}, 292(1-2):183--192, 2019.

\bibitem[KNW19]{KadiriNgWong-LeastPrime}
H.~Kadiri, N.~Ng, and P.-J. Wong.
\newblock The least prime ideal in the {C}hebotarev density theorem.
\newblock {\em Proc. Amer. Math. Soc.}, 147(6):2289--2303, 2019.

\bibitem[Li12]{Li-NonSplit}
X.~Li.
\newblock The smallest prime that does not split completely in a number field.
\newblock {\em Algebra Number Theory}, 6(6):1061--1096, 2012.

\bibitem[LMO79]{LagariasMontgomeryOdlyzko-LeastPrime}
J.~C. Lagarias, H.~L. Montgomery, and A.~M. Odlyzko.
\newblock A bound for the least prime ideal in the {C}hebotarev density
  theorem.
\newblock {\em Invent. Math.}, 54(3):271--296, 1979.

\bibitem[LO77]{LagariasOdlyzko-EffectiveChebotarev}
J.~C. Lagarias and A.~M. Odlyzko.
\newblock Effective versions of the {C}hebotarev density theorem.
\newblock In {\em Algebraic number fields: {$L$}-functions and {G}alois
  properties ({P}roc. {S}ympos., {U}niv. {D}urham, {D}urham, 1975)}, pages
  409--464. Academic Press, London-New York, 1977.

\bibitem[LOS24]{LemkeOliver-Smith-FaithfulArtin}
R.~J. Lemke~Oliver and A.~Smith.
\newblock {Faithful Artin induction and the Chebotarev density theorem}, 2024.

\bibitem[LOTZ24]{LemkeOliverThornerZaman-ApproximateArtin}
R.~J. Lemke~Oliver, J.~Thorner, and A.~Zaman.
\newblock An approximate form of {A}rtin's holomorphy conjecture and
  non-vanishing of {A}rtin {$L$}-functions.
\newblock {\em Invent. Math.}, 235(3):893--971, 2024.

\bibitem[Mat12]{Matomaki-SignChange}
K.~Matom\"aki.
\newblock On signs of {F}ourier coefficients of cusp forms.
\newblock {\em Math. Proc. Cambridge Philos. Soc.}, 152(2):207--222, 2012.

\bibitem[MM97]{MurtyMurty-Nonvanishing}
M.~R. Murty and V.~K. Murty.
\newblock {\em Non-vanishing of {$L$}-functions and applications}.
\newblock Modern Birkh\"{a}user Classics. Birkh\"{a}user/Springer Basel AG,
  Basel, 1997.
\newblock [2011 reprint of the 1997 original] [MR1482805].

\bibitem[MMS88]{MurtyMurtySaradha-ChebotarevI}
M.~R. Murty, V.~K. Murty, and N.~Saradha.
\newblock Modular forms and the {C}hebotarev density theorem.
\newblock {\em Amer. J. Math.}, 110(2):253--281, 1988.

\bibitem[MP15]{MurtyPatankar-Tate}
V.~K. Murty and V.~M. Patankar.
\newblock Tate cycles on {A}belian varieties with complex multiplication.
\newblock {\em Canad. J. Math.}, 67(1):198--213, 2015.

\bibitem[Mur97]{Murty-ChebotarevII}
V.~K. Murty.
\newblock Modular forms and the {C}hebotarev density theorem. {II}.
\newblock In {\em Analytic number theory ({K}yoto, 1996)}, volume 247 of {\em
  London Math. Soc. Lecture Note Ser.}, pages 287--308. Cambridge Univ. Press,
  Cambridge, 1997.

\bibitem[MV10]{MichelVenkatesh-SubconvexityGL2}
P.~Michel and A.~Venkatesh.
\newblock The subconvexity problem for {${\rm GL}_2$}.
\newblock {\em Publ. Math. Inst. Hautes \'Etudes Sci.}, (111):171--271, 2010.

\bibitem[Oes79]{Oesterle-CDTGRH}
J.~Oesterl\'e.
\newblock Versions effectives du th\'eor\`{e}me de {C}hebotarev sous
  l'hypoth\`ese de {R}iemann g\'en\'eralis\'ee.
\newblock {\em Ast\'erisque}, 61:165--167, 1979.

\bibitem[Pol14]{Pollack-PrimeSplittingAbelian}
P.~Pollack.
\newblock Prime splitting in abelian number fields and linear combinations of
  {D}irichlet characters.
\newblock {\em Int. J. Number Theory}, 10(4):885--903, 2014.

\bibitem[Qu10]{Qu-SignChange}
Y.~Qu.
\newblock Sign changes of {F}ourier coefficients of {M}aass eigenforms.
\newblock {\em Sci. China Math.}, 53(1):243--250, 2010.

\bibitem[Ser77]{Serre-LinearRepresentations}
J.-P. Serre.
\newblock {\em Linear representations of finite groups}, volume Vol. 42 of {\em
  Graduate Texts in Mathematics}.
\newblock Springer-Verlag, New York-Heidelberg, french edition, 1977.

\bibitem[Ser79]{Serre-LocalFields}
J.-P. Serre.
\newblock {\em Local fields}, volume~67 of {\em Graduate Texts in Mathematics}.
\newblock Springer-Verlag, New York-Berlin, 1979.
\newblock Translated from the French by Marvin Jay Greenberg.

\bibitem[SL96]{LenstraStevenhagen-Chebotarev}
P.~Stevenhagen and H.~W. Lenstra, Jr.
\newblock Chebotar\"ev and his density theorem.
\newblock {\em Math. Intelligencer}, 18(2):26--37, 1996.

\bibitem[Sta74]{Stark-EffectiveBrauerSiegel}
H.~M. Stark.
\newblock Some effective cases of the {B}rauer-{S}iegel theorem.
\newblock {\em Invent. Math.}, 23:135--152, 1974.

\bibitem[TZ17]{ThornerZaman-Explicit}
J.~Thorner and A.~Zaman.
\newblock An explicit bound for the least prime ideal in the {C}hebotarev
  density theorem.
\newblock {\em Algebra Number Theory}, 11(5):1135--1197, 2017.

\bibitem[TZ19]{ThornerZaman-Unified}
J.~Thorner and A.~Zaman.
\newblock A unified and improved {C}hebotarev density theorem.
\newblock {\em Algebra Number Theory}, 13(5):1039--1068, 2019.

\bibitem[TZ24]{ThornerZhang}
J.~Thorner and Z.~Zhang.
\newblock {A new bound on the relative error in the Chebotarev density
  theorem}, 2024.

\bibitem[VL66]{VinogradovLinnik-LeastPrimeQuadraticResidue}
A.~I. Vinogradov and Yu.~V. Linnik.
\newblock Hypoelliptic curves and the least prime quadratic residue.
\newblock {\em Dokl. Akad. Nauk SSSR}, 168:259--261, 1966.

\bibitem[Wei83]{Weiss}
A.~Weiss.
\newblock The least prime ideal.
\newblock {\em J. Reine Angew. Math.}, 338:56--94, 1983.

\bibitem[Zam17]{Zaman-Least}
A.~Zaman.
\newblock Bounding the least prime ideal in the {C}hebotarev density theorem.
\newblock {\em Funct. Approx. Comment. Math.}, 57(1):115--142, 2017.

\bibitem[Zam18]{Zaman-NonSplit}
A.~Zaman.
\newblock The least unramified prime which does not split completely.
\newblock {\em Forum Math.}, 30(3):651--661, 2018.

\bibitem[Zhu20]{Zhu-Thesis}
R.~Zhu.
\newblock {\em The {L}east {P}rime whose {F}robenius is an $n$-cycle}.
\newblock PhD thesis, University of Toronto, 2020.

\end{thebibliography}

\end{document}